\documentclass[10pt,reqno]{amsart} % please use amsart at 11pt
\usepackage{amssymb,latexsym}
\usepackage{cite} % to get Refs. [1,2,3] typeset as [1--3] automatically
\usepackage[height=190mm,width=130mm]{geometry} % this is the journal
\theoremstyle{plain}
\newtheorem{theorem}{Theorem}
\newtheorem{lemma}{Lemma}
\newtheorem{corollary}{Corollary}
\newtheorem{proposition}{Proposition}

\theoremstyle{definition}

\theoremstyle{remark}

\newcommand{\half}{\frac{1}{2}}

\newcommand{\Z}{\mathbb{Z}}
\newcommand{\C}{\mathbb{C}}

\newcommand{\vac}{\mathbf 1}

\newcommand{\bbar}{\overline{b}}

\numberwithin{equation}{section} % to get equations numbered
\newcommand{\rmap}{\longrightarrow}

\newcommand{\F}{\ensuremath{\mathcal{F}}}

\newcommand{\D}{\ensuremath{\mathcal{D}}}

\newcommand{\g}{\ensuremath{\Gamma}}

\newcommand{\ps}{{\raise 1pt\hbox{\tiny (}}}

\newcommand{\pss}{{\raise 1pt\hbox{\tiny [}}}
\newcommand{\pdd}{{\raise 1pt\hbox{\tiny ]}}}
\newcommand{\pd}{{\raise 1pt\hbox{\tiny )}}}

\newcommand{\bs}{{\raise 1pt\hbox{\tiny [}}}
\newcommand{\bd}{{\raise 1pt\hbox{\tiny ]}}}

\def\cross{\mathinner{\mathrel{\raise0.8pt\hbox{$\scriptstyle>$}}
                 \joinrel\mathrel\triangleleft}}

\usepackage{citehack}
\usepackage{amsmath,amsthm}
\usepackage{amsfonts}
\usepackage{amssymb}
\usepackage{longtable}
\usepackage[matrix,arrow,curve]{xy}

\def\D{\mathcal{D}}
\def\C{\mathbb{C}}

\def\F{\mathcal{F}}

\newcommand{\nn}{\nonumber \\}

\newcommand{\Cg}{\mathfrak{C}_g}
\newcommand{\Chat}{\widehat\C}

\theoremstyle{definition}
\theoremstyle{plain}

\DeclareMathOperator{\Res}{Res}

\newcommand{\End}{\textup{End}}
\newcommand{\Id}{\textup{Id}}

\newcommand{\SL}{\textup{SL}}

\newcommand{\wt}{{\rm wt}}
\begin{document}
\title[Torsor structure of level-raising operators]    
{Torsor structure of level-raising operators}    
%%
%%%%%%%%%%%%%%%%%%%%%%%%%%%%%%%%%%%%%%%%%%%%%%%%%%%%%%%%%%%%%%%%%%%%%%%%%%%%%%%
%%
\author{A. Zuevsky} 
\address{Institute of Mathematics \\ Czech Academy of Sciences\\ Zitna 25, 11567 \\ Prague\\ 
Czech Republic}
\email{zuevsky@yahoo.com}
%%
%%%%%%%%%%%%%%%%%%%%%%%%%%%%%%%%%%%%%%%%%%%%%%%%%%%%%%%%%%%%%%%%%%%%%%%%%
%%
\begin{abstract}
%%
%%%%%%%%%%%%%%%%%%%%%%%%%%%%%%%%%%%%%%%%%%%%%%%%%%%%%%%%%%%%%%%%%%%%%%%%%%
%%%%
We consider families of reductive complexes
 related by level-raising operators and 
originating from an associative algebra. 
%%
%%%%%%%%%%%%%%%%%%%%%%%%%%%%%%%%%%%%%%%%%%%%%%%%
%%
In the main theorem it is shown that the 
multiple cohomology of that complexes 
is given by the factor space of products 
of reduction operators. 
%%
%%%%%%%%%%%%%%%%%%%%%%%%%%%%%%%%%%%%%%%%%%%%%%%%%%%%%%%%%%%%%%%%%%%%%5
%%%%
In particular, we compute explicit torsor structure of  
the genus $g$ multiple cohomology 
of the families of horizontal complexes with spaces of 
of canonical converging reductive differential forms  
for a $C_2$-cofinite quasiconformal 
strong-conformal field theory-type 
 vertex operator algebra  
associated to a complex curve. 
%%
%%%%%%%%%%%%%%%%%%%%%%%%%%%%%%%%%%%%%%%%%%%%%%%%%%%%%%%%%
%%%%
That provides an equivalence of multiple cohomology 
to factor spaces of products of sums of
 reduction functions 
with actions of the group of local coordinates automorphisms. 
%%

%%%%%%%%%%%%%%%%%%%%%%%%%%%%%%%%%%%%%%%%%%%%%%%%%%%%%%%
%% 
\end{abstract}

\vskip12pt  % insert '\vskip12pt' while using '\twocolumn' command
%\vskip28pt % if there is no keywords
%% 
\maketitle
%%
%%%%%%%%%%%%%%%%%%%%%%%%%%%%%%%%%%%%%%%%%%%%%%%%%%%%%%%%%%%%%%%%%%%%%%%%%%%%%%%%%%%%%%%%%%%%%%
%%%%%%%%%%%%%%%%%%%%%%%%%%%%%%%%%%%%%%%%%%%%%%%%%%%%%%%%%%%%%%%%%%%%%%%%%%%%%%%%%%%%%%%%%%%%%%
\begin{center}
Conflict of Interest and Data availability Statements
\end{center}
The author states that: 

1.) The paper does not contain any potential conflicts of interests. 

2.) The paper does not use any datasets. No dataset were generated during and/or analyzed 
during the current study. 

3.) The paper includes all data generated or analysed during this study. 

4.) Data sharing is not applicable to this article as no datasets were generated or analysed during the current study.

5.) The data of the paper can be shared openly.  

6.) No AI was used to write this paper. 
%% 
%%%%%%%%%%%%%%%%%%%%%%%%%%%%%%%%%%%%%%%%%%%%%%%%%%%%%%%%%%%%%%%%%%%%%%%%%%%%%%%%%%%%%%%%%%%%%%
%%%%%%%%%%%%%%%%%%%%%%%%%%%%%%%%%%%%%%%%%%%%%%%%%%%%%%%%%%%%%%%%%%%%%%%%%%%%%%%%%%%%%%%%%%%%%
%%
\section{Introduction}
%%
%%%%%%%%%%%%%%%%%%%%%%%%%%%%%%%%%%%%%%%%%%%%%%%%%%%%%%%%%%%%%%%%%%%%%%%%%
%%%%
In order to introduce cohomology of a geometric structure, 
it is often useful to attach an algebraic construction 
with easier computable cohomology. 
%%
%%%%%%%%%%%%%%%%%%%%%%%%%%%%%%%%%%%%%%%%%%%%%%%%%%%%%%%%%%%%%%%%%%%%%%%%
%%%%
In many examples that strategy becomes effective.  
The question of computation of cohomology 
via various algebraic structure \cite{BZF, Bott, BS, F, Ga, Ko} 
associated complex manifolds 
continues to attract attention.   
%%
%%%%%%%%%%%%%%%%%%%%%%%%%%%%%%%%%%%%%%%%%%%%%%%%%%%%%%%%%%%%%%%%%%%%%%%%%%%
%%%%
The way we attach an auxiliary structure to a manifold determines 
the success of cohomology computation methods. 
%%
%%%%%%%%%%%%%%%%%%%%%%%%%%%%%%%%%%%%%%%%%%%%%%%%%%%%%%%%%%%%%%%%%%%%%%%%%%%%
%%%%
In \cite{BDSK1, BDSK2, Huang, G, LQW} 
cohomology theories for vertex algebras   
\cite{BZF, FHL, LL, DL, K} and related structures \cite{LQW} 
 were introduced.
In particular, in \cite{Huang} cohomology of formal maps associated 
with extensions of 
modules for vertex algebras was defined. 
%%
%%%%%%%%%%%%%%%%%%%%%%%%%%%%%%%%%%%%%%%%%%%%%%%%%%%%%%%%%%%%%%%
%%%% 
In such formulations, matrix elements of vertex operators 
 have formal parameters identified with 
local coordinates on Riemann surfaces 
\cite{Z1, MT1, MT2, MT3, MTZ, T, TW, TW1, TZ1}.     
%%

%%%%%%%%%%%%%%%%%%%%%%%%%%%%%%%%%%%%%%%%%%%%%%%%%%%%%%%%%%%%%%%%%%%%%%%%%%%%%%%%
%%%%
In this paper, we start with a general setup 
and a cohomology theory spaces of converging functional 
depending on elements of graded associative algebras, and constituting  
 families of horizontal chain complexes. 
%%
%%%%%%%%%%%%%%%%%%%%%%%%%%%%%%%%%%%%%%%%%%%%%
%%%%
It is natural to look for an extension of 
the original horizontal complex (which we call level zero)  
to a family of $\kappa\ge 0$ complexes 
with the corresponding horizontal differential acting 
on higher level spaces. 
%%
%%%%%%%%%%%%%%%%%%%%%%%%%%%%%%%%%%%%%%%%%%%%%%%%%%%%%%%%%%%%%%%%%%
%%
Elements of a particular level 
in the families of horizontal complexes are coherently 
related by standardly defined homology-based 
 level-raising operators.   
It is assumed that coboundary operators for horizontal complexes are 
given by reduction operators for corresponding functionals. 
%%
%%%%%%%%%%%%%%%%%%%%%%%%%%%%%%%%%%%%%%%%%%%%%%%%%%%%%%%%%%%%%%%%%%%%%%%%%%%
%%%%
Having the structure described above, we determine the standard cohomology 
of arbitrary level horizontal complexes. Since complexes are related by 
geometrically-motivated 
level-raising operators the resulting cohomology is called multiple. 
%% 
%%%%%%%%%%%%%%%%%%%%%%%%%%%%%%%%%%%%%%%%%%%%%%%%%%%%%%%%%%%%%%%%%%%%%%%%%%%%%%%%%%%%%%%%%%%%%%
%%%%
The notion of multiple cohomology  
constructed in this paper 
enriches cohomology structure 
of a horizontal complex of level zero. 
%% 
%%%%%%%%%%%%%%%%%%%%%%%%%%%%%%%%%%%%%%%%%%%%%%%%%%%%%%%%%%%%%%%%%%
%%%% 
In the case of spaces of functionals invariant with respect to 
the action of a group $\mathcal G$, the corresponding cohomology becomes 
$\mathcal G$-invariant. 
%%

%%%%%%%%%%%%%%%%%%%%%%%%%%%%%%%%%%%%%%%%%%%%%%%%%%%%%%%%%%%%%%%%%%%%%%%
%%%%
Particular representations of the construction above is given by 
the example of admissible vertex operator algebras (see the definitions 
in the Subsection \ref{kundor}) \cite{BZF, FHL, DL, LL, K}  
considered on Riemann surfaces.  
%%
%%%%%%%%%%%%%%%%%%%%%%%%%%%%%%%%%%%%%%%%%%%%%%%%%%%%%%%%%%%%%%%%%%%%%%%
%%%%
In that formulation, the spaces $C^n(V, g)$ 
of horizontal complexes of a particular level 
$\kappa$ are formed by $n$-point differential forms  
 for a an admissible vertex operator 
algebra $V$ with formal parameters identified with local coordinates 
on a genus $g=\kappa$ Riemann surface 
formed in Schottky uniformization procedure \cite{Be1, Be2, McIT, M, TW, TW1}.   
%%
%%%%%%%%%%%%%%%%%%%%%%%%%%%%%%%%%%%%%%%%%%%%%%%%%%%%%%%%%%%%%%%%%%%%%%%
%%%%
 We use the level-raising transformations \eqref{normata} 
\cite{TW, MT1, MT2, MT3, TZ1, M}
in order to introduce spaces of higher-level invariant differential forms, 
as well as to form families of chain double complexes. 
That leads to more sophisticated structure of cohomology.   
%% 
%%%%%%%%%%%%%%%%%%%%%%%%%%%%%%%%%%%%%%%%%%%%%%%%%%%%%%%%%%%%%%%%%%%%%%%%%
%%%%
The horizontal coboundary operators are defined via the vertex operator 
reduction formulas \cite{Z1, Z2, Y, TW, T, MT3, MT2, GT, MTZ} for corresponding 
genus.  
%%
%%%%%%%%%%%%%%%%%%%%%%%%%%%%%%%%%%%%%%%%%%%%%%%%%%%%%%%%%%%%%%%%%%%%%%%%%%%%%%%
%%%% 
In examples of the Sections \ref{genuszero0}--\ref{apala} of genus zero and 
arbitrary genus Riemann surfaces, 
we consider a specific ansatz leading to the accomplishment of chain conditions
for coboundary operators.  
%%
%%%%%%%%%%%%%%%%%%%%%%%%%%%%%%%%%%%%%%%%%%%%%%%%%%%%%%%%%%%%%%%%%%%%%%%%%%%
%%%% 
Then it is possible to find 
cohomology expressions in terms of Bers forms \cite{Be1, Be2, TW1} 
 on Riemann surfaces. 
%%
%%%%%%%%%%%%%%%%%%%%%%%%%%%%%%%%%%%%%%%%%%%%%%%%%%%%%%%%%%%%%%%%%%%%%%%%%%%%%%
%%
In this case, corresponding cohomology is given by the factor space of the space of 
products of sums of Bers quasiforms  
$\Pi^{(g)}({\bf a}_n.{\bf z}_n)$ with transformed arguments.  
%%

%%%%%%%%%%%%%%%%%%%%%%%%%%%%%%%%%%%%%%%%%%%%%%%%%%%%%%
%%%%
The plan of the paper is the following. 
In this Section we formulate the general abstract setup 
for associative algebras, and end up with the formulation 
of the main theorem. 
%%
%%%%%%%%%%%%%%%%%%%%%%%%%%%%%%%%%%%%%%%%%%%%%%%%%%%%%%%%%%%%%%%%%%%%
%%%%
The Section \ref{formul} discusses the case of vertex operator 
algebras, and contains a prove Theorem \ref{coca}  
concerning the torsor structure of genus $g$ multiple 
cohomology of families of horizontal complexes. 
%%
%%%%%%%%%%%%%%%%%%%%%%%%%%%%%%%%%%%%%%%%%%%%%%%%%%%%%
%%%%
In the Sections \ref{genuszero0}--\ref{apala} we provide examples of 
genus zero and arbitrary genus $g$ complexes of 
reductive $n$-point functions for vertex operator algebras. 
%%
%%%%%%%%%%%%%%%%%%%%%%%%%%%%%%%%%%%%%%%%%%%%%%%%%%%%%%%%
%%
The Appendix \ref{vtorog} recalls information about 
the Bers quasiforms \cite{Be1, Be2, TW1}
needed for formulation of the genus $g$ reduction formulas. 
%%
%%%%%%%%%%%%%%%%%%%%%%%%%%%%%%%%%%%%%%%%%%%%%%%%%%%%%%%%%%
%%
The Appendix \ref{generererv} states definitions and basic properties 
for vertex operator algebras. 
%%
%%%%%%%%%%%%%%%%%%%%%%%%%%%%%%%%%%%%%%%%%%%%%%%%%%%%%%%%
%%%%
\subsection{The general setup}
\label{mangaloid}
Let us fix our general setup and notations.    
 We denote a finite sequence of elements $(a_1$, $\ldots$, $a_n)$ as $\bf a_n$. 
An operator $\mathcal T$ acting on the $i$-th element of a sequence of elements 
$(a_1$, $\ldots$, $a_n)$ is denoted by $\mathcal T_i$, $1 \le i \le n$.  
%%
%%%%%%%%%%%%%%%%%%%%%%%%%%%%%%%%%%%%%%%%%%%%%%%%%%%%%%%%%%%%%%%%%%%%%%%%%
%%%%
In particular, denote by $T_i(a)$ the insertion operator 
$T_i(a)(a_1, \ldots, a_n)=(a_1, \ldots, a.a_i, \ldots, a_n)$ 
with some action $a.a_i$ of $a$ on $a_i$.  
%%
%%%%%%%%%%%%%%%%%%%%%%%%%%%%%%%%%%%%%%%%%%%%%%%%%%%%%%%%%%%%%%%%%%%%%%%%%%%%%
%%%%
Consider an associative graded algebra  
$\mathcal V=\bigoplus_{s\in \Z} \mathcal V_{(s)}$,    
endowed with a non-degenerate bilinear pairing 
 $(.,.): \mathcal V'\times \mathcal V \to \C$,  
where $\mathcal V'$ is the dual to $\mathcal V$ with respect to $(.,.)$, 
%%
%%%%%%%%%%%%%%%%%%%%%%%%%%%%%%%%%%%%%%%%%%%%%%%%%%%%%%%%%%%%%%%%%%%%%%%%%%%
%%%%
Let us call the level $\kappa=0$ spaces
 $C^n(\mathcal V, 0)$, $n \ge 0$ the spaces of $\C$-valued converging  
functionals $\F_{\mathcal V}^{(0)}({\bf x}_n)$ depending on 
$x_i \in \mathcal V$, $1 \le i \le n$ elements. 
%%
%%%%%%%%%%%%%%%%%%%%%%%%%%%%%%%%%%%%%%%%%%%%%%%%%%%%%%%%%%%%%%%%%%%%%%%%%
%%%%
Let $r_k=(b'_k, b_k)$ denote a pair of elements 
 $b'_k \in {\mathcal V}'$, $b_k \in \mathcal V$.  
%%
%%%%%%%%%%%%%%%%%%%%%%%%%%%%%%%%%%%%%%%%%%%%%%%%%%%%%%%%%%%%%%%%%%%
%%%%
Assuming convergence of the action of the 
insertion operator 
 $T(r_k)$ acting on 
 elements of $C^n(\mathcal V, 0)$, $n \ge 0$,    
we define inductively 
the level-raising operator and corresponding level $\kappa \ge 0$ 
spaces $C^n(\mathcal V,  \kappa)$ of converging functionals 
$\F_{\mathcal V}^{(\kappa)}({\bf x}_n)$ depending on $n \ge 0$ $\mathcal V$-elements. 
%% 
%%%%%%%%%%%%%%%%%%%%%%%%%%%%%%%%%%%%%%%%%%%%%%%%%%%%%%%%%%%%%%%%%%%%%%%%%%%
%%%%
With a basis $\{b_k\}$ of $\mathcal V_{(k)}$, we obtain   
%%
%%%%%%%%%%%%%%%%%%%%%%%%%%%%%%%%%%%%%%%%%%%%%%%%%%%%%%%%%%%%%%%%%%%%%%%%%%%
%%%%
\begin{eqnarray}
\label{trubas}
\Delta^{(1)}(r_k): C^n(\mathcal V,  \kappa) \to C^n(\mathcal V,  \kappa+1),   
\quad \Delta^{(1)}(r_k)= \sum_{b_k \in \mathcal V_{(k)}} T(r_k),      
%%
%%%%%%%%%%%%%%%%%%%%%%%%%%%%%%%%%%%%%%%%%%%%%%%%%%%%%%%%%%%%%%%%%%%%%%%
\nn
\label{normata}
\F_{\mathcal V}^{(\kappa+1)}({\bf x}_n)=\Delta^{(1)}(r_k).\; 
\F_{\mathcal V}^{(\kappa)}({\bf x}_n) 
\sum_{b_k \in \mathcal V_{(k)}} T(r_k).\F_{\mathcal V}^{(\kappa)}({\bf x}_n) 
= \sum_{b_k \in \mathcal V_{(k)}} \F_{\mathcal V}^{(\kappa)}({\bf x}_n, r_k).  
\end{eqnarray} 
%% 
%%%%%%%%%%%%%%%%%%%%%%%%%%%%%%%%%%%%%%%%%%%%%%%%%%%%%%%%%%%%%%%%%%%%%%%%%%%%
%%%%
 The following diagram corresponds to 
the families of horizontal complexes  
for $r_1$, $r'$, $r''$, $r''' \in \C$,  
%%  
%%%%%%%%%%%%%%%%%%%%%%%%%%%%%%%%%%%%%%%%%%%%%%%%%%%%%%%%%%%%%%%%%
%%
\begin{eqnarray}
\label{buzovaish}
&& C^0\left(\mathcal V, 0\right) 
  \stackrel{\delta^{(0)}(x_1)} {\longrightarrow}  C^1\left(\mathcal V, 0\right)    
  \stackrel{\delta^{(0)}({\bf x}_2)} {\rmap} 
 \cdots  
\nn
&&   \downarrow  \Delta^{(1)}(r) 
\qquad \qquad \quad  \downarrow  \Delta^{(1)}(r') 
\nn
&& C^0\left(\mathcal V, 1\right) 
  \stackrel{\delta^{(1)}(x'_1)} {\longrightarrow}  C^1\left(\mathcal V, 1\right)    
 \stackrel{\delta^{(1)}({\bf x}'_2)} {\rmap} 
 \cdots  
\nn
&&   \downarrow  \Delta^{(1)} (r'')
\qquad \qquad \quad  \downarrow  \Delta^{(1)}(r''')
\nn
&&  \vdots \qquad \qquad \qquad \qquad \qquad \vdots 
\end{eqnarray} 
%% 
%%%%%%%%%%%%%%%%%%%%%%%%%%%%%%%%%%%%%%%%%%%%%%%%%%%%%%%%%%%%%%%%%%%%%%%%
%%%%
One can compose the higher level-raising operators in the coherent way 
\begin{eqnarray}
\label{govnata}
\Delta^{(\kappa)}({\bf r}_{\kappa}) 
= \Delta^{(1)} (r_1) \circ \cdots \circ \Delta^{(1)}(r_\kappa)   
= \sum\limits_{ {\bf b}_{k_i} \in {\mathcal V}_{(k_i)}} T({\bf r}_\kappa),    
\end{eqnarray}
%%
%%%%%%%%%%%%%%%%%%%%%%%%%%%%%%%%%%%%%%%%%%%%%%%%%%%%%
%%%%
where the  
summation is over $1 \le i \le \kappa$ independent  
basises ${\bf b}_{k_i} \in {\mathcal  V}_{(k_i)}$  
and their duals ${\bf b}'_{k_i} \in {\mathcal V}'_{(k_i)}$ 
with respect to the corresponding bilinear pairing $(.,.)$.  
%%
%%%%%%%%%%%%%%%%%%%%%%%%%%%%%%%%%%%%%%%%%%%%%%%%%%%%%%%%%%%%%%%%%%
%%%%
One can also think of a notion of negative level.   
%%

%%%%%%%%%%%%%%%%%%%%%%%%%%%%%%%%%%%%%%%%%%%%%%%%%%%%%%%%%%%%%%%%%%%%%%%
%%%%
Assume that all elements of the spaces $C^n(\mathcal V, \kappa)$
are reductive: 
%%
%%%%%%%%%%%%%%%%%%%%%%%%%%%%%%%%%%%%%%%%%%%%%%%%%%%%%%%%%%%%%%%%%
%%%%
an element $\F_{\mathcal V}^{(\kappa)}({\bf x}_n)\in C^n(V, \kappa)$  
is representable as 
%%
%%%%%%%%%%%%%%%%%%%%%%%%%%%%%%%%%%%%%%%%%%%%%%%%%%%%%%%%%%%%%%%% 
%%
$\F_{\mathcal V}^{(\kappa)}({\bf x}_n)=\mathcal B^{(\kappa)}({\bf x}_n).     
 \F_{\mathcal V}^{(\kappa)}({\bf x}_{n-1})$,     
where $\mathcal B^{(\kappa)}({\bf x}_n)$ is an operator depending 
on ${\bf x}_n$ and $\kappa$.  
%%
%%%%%%%%%%%%%%%%%%%%%%%%%%%%%%%%%%%%%%%%%%%%%%%%%%%%%%%%%%%%%%%%%%%%%%%%%%
%%%%
The operator $\mathcal B^{(\kappa)}({\bf x}_n)$ may 
 depend on a particular representation of elements $x_i \in \mathcal V$, 
and the kind of functions $\F_{\mathcal V}^{(\kappa)}({\bf x}_n)$. 
%%
%%%%%%%%%%%%%%%%%%%%%%%%%%%%%%%%%%%%%%%%%%%%%%%%%%%%%%%%%%%%%%%%%%%%%%%%%%%%%%%%
%%%%
In certain cases, the 
 the reduction operators $\mathcal B^{(\kappa)}({\bf x}_n)$   
may have the form with a function $\phi(x_n)$ of $x_n$, 
%%
%%%%%%%%%%%%%%%%%%%%%%%%%%%%%%%%%%%%%%%%%%%%%%%%%%%%%%%%%%%%%%%%% 
%%
$\mathcal B^{(\kappa)}({\bf x}_n).=   
\sum_{i=1}^n \sum_{l \ge 0} 
f^{(\kappa)}_{i, l}({\bf x}_n) \; T_i(\phi(x_n))$.    
%%%%
In the Sections \ref{genuszero0}--\ref{apala} 
 we will see explicit examples of these formulas.   
%%
%%%%%%%%%%%%%%%%%%%%%%%%%%%%%%%%%%%%%%%%%%%%%%%%%%%%%%%%%%%%%%%%%%%%%%
%%%%
Due to \eqref{normata}, \eqref{govnata}, and reductivity 
one is able to introduce families of horizontal complexes 
with the differentials given by  
 $\delta^{(\kappa)}({\bf x}_{n+1})=\mathcal B^{(\kappa)}({\bf x}_{n+1}):  
 C^n(\mathcal V, \kappa)  \to C^{n+1}(\mathcal V, \kappa)$,  
for all $({\bf x}_n)$. 
%%
%%%%%%%%%%%%%%%%%%%%%%%%%%%%%%%%%%%%%%%%%%%%%%%%%%%%%%%%%%%
%%%%
The families of horizontal complexes are related by 
the vertical level-raising operators \eqref{trubas} 
$\Delta^{(\kappa)}(r_\kappa)$. 
%% 
%%%%%%%%%%%%%%%%%%%%%%%%%%%%%%%%%%%%%%%%%%%%%%%%%%%%%%%%%%%%%%%%%%%%%%%%%%%%%%%
%%%%
The chain condition for the differentials of the horizontal complexes  
 for $n \ge 0$, and all ${\bf x}_{n+1}$, ${\bf x}_n$ is 
$\delta^{(\kappa)}({\bf x}_{n+2})\;  \delta^{(\kappa)}({\bf x}_{n+1}).    
 \F_{\mathcal V}^{(\kappa)}({\bf x}_n)=0$.    
%%
%%%%%%%%%%%%%%%%%%%%%%%%%%%%%%%%%%%%%%%%%%%%%%%%%%%%%%%%%%%%%%%%%%%%%%%%%
%%%%
Note that in the classical chain condition considered in the examples of 
the Sections \ref{genuszero0}--\ref{apala} we may take ${\bf x}_{n+1}={\bf x}_n$.  
%%
%%%%%%%%%%%%%%%%%%%%%%%%%%%%%%%%%%%%%%%%%%%%%%%%%%%%%%%%%%%%%%%%%%%%%%%%%
%%%%
With the families of horizontal complexes one associates naturally 
the notion of the multiple level $\kappa$ 
standardly defined cohomology $H^n(\mathcal V, \kappa)$. 
%%
%%%%%%%%%%%%%%%%%%%%%%%%%%%%%%%%%%%%%%%%%%%%%%%%%%%%%%%%%%%%%%%%%%%%%%%%%%%%%
%%%%
It is worth to mention that both coboundary operators as well as cohomology 
do depend on choices of extra $x_{n+1}$ elements raising the total number 
of arguments. That fact is extremely convenient in the cohomology description 
of objects which can be associated with $\mathcal V$. 
%%

%%%%%%%%%%%%%%%%%%%%%%%%%%%%%%%%%%%%%%%%%%%%%%%%%%%%%%%%%%%%%%%%%%%%%%%%%%%%%%%%
%%%%
Taking into account the reductivity of $C^n(\mathcal V, \kappa)$-functionals  
the chain conditions induce an algebra of operators forming   
 sequences of ideals 
$\mathcal I(n)$, $n\ge 0$, i.e., 
$\left\{ \mathcal B^{(\kappa)}({\bf x}_n) \right.$,
 $\ldots$,  $\left. \mathcal B^{(\kappa)}(x_0)  \right\}$,  
such that  
$\mathcal B^{(\kappa)}({\bf x}_{n-1})$ $\ldots$ $\mathcal B^{(\kappa)}(x_0)$ $\ne 0$ 
 but $\mathcal B^{(\kappa)}({\bf x}_n) \ldots \mathcal B^{(\kappa)}(x)$ vanishes. 
%%
%%%%%%%%%%%%%%%%%%%%%%%%%%%%%%%%%%%%%%%%%%%%%%%%%%%%%%%%%%%%%%%%%%%%%%%%%%%%%%%%%%%%%%% 
%%%%
In general \cite{Ko}, such sequences lead to existence of related cohomology invariants. 
Namely, for a continual parameter $t\in \C$, a natural product $\cdot$,  
and with certain conditions on $\mathcal B^{(\kappa)}({\bf x}_n)$, 
the cohomology class 
%%
%%%%%%%%%%%%%%%%%%%%%%%%%%%%%%%%%%%%%%%%%%%%%%%%%%%%%%%%%%%%%%%%%%%%%%%%%%%%%%%%%%
%%
$\left[ \partial_t\mathcal B^{(\kappa)}({\bf x}_n)\right.$ 
$\cdot$ $\mathcal B^{(\kappa)}({\bf x}_n)$ $\cdot$ 
$\left(\partial^{(\kappa)}({\bf x}_{n+1})\right)^{p+1}$ 
 $\left. \mathcal B^{(\kappa)}({\bf x}_n)\right]$, $p \ge 0$,   
becomes invariant with respect to the additive group of formal parameters. 
%%
%%%%%%%%%%%%%%%%%%%%%%%%%%%%%%%%%%%%%%%%%%%%%%%%%%%%%%%%%%%%%%%%%%%%%%%%%%%%%%%%%%%%%
%%%%
The reductivity $C^n(\mathcal C, \kappa)$ brings about also a projective structure 
with respect to the space of level zero elements.  
%%
%%%%%%%%%%%%%%%%%%%%%%%%%%%%%%%%%%%%%%%%%%%%%%%%%%%%%%%%%%%%%%%%%%%%%%%%%%
%%%%
By applying extra chain conditions on horizontal differentials as well 
as on level-raising operators,  
 it is possible to introduce the total complex for $l=n+\kappa$,   
%%
%%%%%%%%%%%%%%%%%%%%%%%%%%%%%%%%%%%%%%%%%%%%%%%%%%%%%%%%%%%%%%%%%%%%%%%
%%%%
$\widehat\delta^{(\kappa)}({\bf x}_n)= 
\delta^{(\kappa)}({\bf x}_n) + (-1)^\kappa \Delta^{(\kappa)}({\bf r}_\kappa)$.   
%%
%%%%%%%%%%%%%%%%%%%%%%%%%%%%%%%%%%%%%%%%%%%%%%%%%%%%%%%%%%%%%%%%%%%%%%%%%%%%%%%%%%
%%%%
Nevertheless, such extra conditions restrict strongly the spaces 
$C^n(\mathcal V,  \kappa)$.   
A suitable version of the total complex will be considered in a separate paper. 
%% 
%%%%%%%%%%%%%%%%%%%%%%%%%%%%%%%%%%%%%%%%%%%%%%%%%%%%%%%%%%%%%%%%%%%%%%%%%%%%%
%%%%
 In the case when elements of spaces for families of complex \eqref{buzovaish}  
are invariant (see the Sections \ref{genuszero0}--\ref{apala} for examples)
with respect to a group $\mathcal G$ (e.g., a modular group)  
 we obtain a $\mathcal G$-invariant multiple cohomology.  
%%
%%%%%%%%%%%%%%%%%%%%%%%%%%%%%%%%%%%%%%%%%%%%%%%%%%%%%%%%%%%%%%%%%%%%%%%%%%%%%%%%%%%%%%
%%%%
The natural consequence of the construction above constists in the following theorem. 
%%
%%%%%%%%%%%%%%%%%%%%%%%%%%%%%%%%%%%%%%%%%%%%%%%%%%%%%%%%%%%%
%%%%
\begin{theorem}
\label{cocazero}
With operators $\mathcal B^{(\kappa)}({\bf x}_n)$, $n \ge 0$
satisfying the general chain complex property, 
the multiple cohomology $H^n(\mathcal V, \kappa)$ of $\mathcal V$  
of level $\kappa$ horizontal complexes    
$\left(C^n(\mathcal V, \kappa), 
\delta^{(\kappa)}({\bf x}_n) \right)$   
are equivalent to the factor space  
of vanishing products of reduction functions  
$\prod_{j=l}^j \mathcal B^{(\kappa)}({\bf x}_l)$ 
for $l$ equal $n+1$ over $n$.   \hfill $\qed$
\end{theorem}
%%
%%%%%%%%%%%%%%%%%%%%%%%%%%%%%%%%%%%%%%%%%%%%%%%%%%%%%%%%%%%%%%%%%%%%%%%%%%%%%%%%
%%%%  
We can express directly the formal multiple cohomology associated to the 
families of the complex \eqref{buzovaish}   
 using recursively the vertex operator algebra 
reduction formulas 
%%
%%%%%%%%%%%%%%%%%%%%%%%%%%%%%%%%%%%%%%%%%%%%%%%%%%%%%%%%%%%%%%%%%%%%%%%%%%%%
%%
$H^n$ $(\mathcal V, \kappa)$ 
 $=$ ${\rm Ker}$ $\delta^{(\kappa)}({\bf x}_{n+1})/{\rm Im}$   
$\delta^{(\kappa)}({\bf x}_{n-1})$ $=$  
$\prod_{i=1}^n$ $\mathcal B^{(\kappa)}$ $(x_i)/$ $\prod_{j=1}^{n-1}$  
$\mathcal B^{(\kappa)}(x_j)$.      
%%
%%%%%%%%%%%%%%%%%%%%%%%%%%%%%%%%%%%%%%%%%%%%%%%%%%%%%%%%%%%%%%%%%%%%%%%%%%%%%%%%%%%%%%%%%%% 
%%%%%%%%%%%%%%%%%%%%%%%%%%%%%%%%%%%%%%%%%%%%%%%%%%%%%%%%%%%%%%%%%%%%%%%%%%%%%%%%%%%%%%%%%%%
%%
\section{The multiple cohomology associated with vertex operator algebras}
\label{formul}
%%
%%%%%%%%%%%%%%%%%%%%%%%%%%%%%%%%%%%%%%%%%%%%%%%%%%%%%%%%%%%%%%%%%%%%%%%%%%%%%%%%%%
%%%%
In this Section, using the invariance resulting from the torsor formulation 
of $n$-point vertex operator algebras functions, 
 we show how to construct a coordinate-invariant canonical 
intrinsic cohomology of Riemann surfaces associated 
to an admissible vertex operator algebras. 
%%
%%%%%%%%%%%%%%%%%%%%%%%%%%%%%%%%%%%%%%%%%%%%%%%%%%%%%%%%%%%%%%%%%%%%%%%%
%%%%
Though, the original cohomology associated to a vertex operator algebra 
does depend on the choice of vertex operator algebra raising elements, 
by using the torsor approach to show that that cohomology is 
actually canonical. 
%%
%%%%%%%%%%%%%%%%%%%%%%%%%%%%%%%%%%%%%%%%%%%%%%%%%%%%%%%%%%%%%%%%%%%%%%%
%%%%
It is very important to have a version of cohomology invariant 
with respect to changes of coordinates. 
%%
%%%%%%%%%%%%%%%%%%%%%%%%%%%%%%%%%%%%%%%%%%%%%%%%%%%%%%%%%%%%%%%%%%%%%%%
%%
\subsection{The vertex operator algebra setup}  
\label{kundor}
%%
%%%%%%%%%%%%%%%%%%%%%%%%%%%%%%%%%%%%%%%%%%%%%%%%%%%%%%%%%%%%%%%%%%%%%
%%%%
Let $\mathcal V=V$ be a vertex operator algebra.  
In this Section we introduce the spaces $C^n(V, \kappa)$ 
 of level $\kappa=g$ differential forms  
$\F_{\mathcal V}^{(g)}({\bf x}_n)$ 
depending on $n$ arguments, $x_i=(v_i, z_i)$, $1 \le i \le n$,  
${\bf x}_n$ $=$ $(v_1$, $z_1$; $\ldots$; $v_n$, $z_n)$ 
for $v$, $v_1$, $\ldots$, $v_n \in V$,   
and formal parameters $(z_1, \ldots, z_n)$  
considered as local coordinates on a genus $g$ Riemann surface  
for a vertex operator algebra $V$.  
%% 
%%%%%%%%%%%%%%%%%%%%%%%%%%%%%%%%%%%%%%%%%%%%%%%%%%%%%%%%%%%%%%%%%%%%%%%%%%%%%%%%
%%%%
Let us specify our notation suitable for vertex operator algebra purposes. 
 We denote a product of differentials a 
$dz_1\ldots dz_n$ as ${\bf dz}_n$. 
%%
%%%%%%%%%%%%%%%%%%%%%%%%%%%%%%%%%%%%%%%%%%%%%%%%%%%%%%%%%%%%%%%%%%%%%%%%%%%%%%%%
%%%%
Let $\wt(a)$ denote the weight of a homogeneous 
 vertex operator algebra element $a\in V$ 
with respect to the zero Virasoro mode $L(0)a=\wt(a)a$  
(see the Appendix \ref{generererv}).  
%%
%%%%%%%%%%%%%%%%%%%%%%%%%%%%%%%%%%%%%%%%%%%%%%%%%%%%%%%%%%%%%%%%%%%%%%%%%%%%%%%%%
%%%%
Weighted product of differentials $dz_1^{\wt(a_1)} \ldots dz_n^{\wt(a_n)}$ 
 appears as ${\bf dz}_n^{\wt({\bf a}_n)}$. 
%% 
%%%%%%%%%%%%%%%%%%%%%%%%%%%%%%%%%%%%%%%%%%%%%%%%%%%%%%%%%%%%%%%%%%%%%%%%%%%%%
%%
Let $V$ be 
a simple $C_2$-cofinite quasiconformal vertex operator algebra 
(the general facts about   
vertex operator algebras and their properties
 is recalled in the Appendix \ref{generererv}) 
 of strong conformal field theory type 
 with $V$ isomorphic to the contragredient module $V'$ 
\cite{Z1, G, LL, BZF, TW, T}. 
We call such vertex operator algebras admissible. 
%%

%%%%%%%%%%%%%%%%%%%%%%%%%%%%%%%%%%%%%%%%%%%%%%%%%%%%%%%%%%%%%%%%%%%%%%%
%%%%
The notion of the level corresponds to the genus of a Riemann surface 
on which formal parameters $z_1$, $\ldots$, $z_n$ are considered 
as local coordinates. 
%%
%%%%%%%%%%%%%%%%%%%%%%%%%%%%%%%%%%%%%%%%%%%%%%%%%%%%%%%%%%%%%%%%%%%%%%%
%%%%
 One defines the genus zero differential form  
(corresponding to a genus zero $n$-point function for the corresponding 
vertex operator algebra $V$)  
given by the expression containing $n$ vertex operators 
$Y(x_i)$ for $v'\in V'$ dual to $v$ 
 by means of a dual pairing $\langle .,.\rangle_1$ (see the Subsection \ref{voa}), 
%%
%%%%%%%%%%%%%%%%%%%%%%%%%%%%%%%%%%%%%%%%%%%%%%%%%%%%%%
%%%%
\begin{eqnarray}
\label{bucharas}
 \F_V^{(0)}({\bf x}_n)=\langle v', {\bf Y}({\bf x}_n).v\rangle_1 
{\bf dz}^{\wt({\bf v}_n)},
\end{eqnarray} 
for $\rho=1$ defined on $V$. 
%% 
%%%%%%%%%%%%%%%%%%%%%%%%%%%%%%%%%%%%%%%%%%%%%%%%%%%%%%%%%%%%%%%%%%%%%%%%%%%%%%
%%%%
Then $C^n(V, 0)$ is the space of all such differential forms. 
%%
%%%%%%%%%%%%%%%%%%%%%%%%%%%%%%%%%%%%%%%%%%%%%%%%%%%%%%%%%%%%%%%%%%%%%%%%
%%%%
Let $y=((b', w'), (b, w))$, $b \in V$, $b'\in V'$ be  
dual to $V$, $w$, $w' \in \C$.  
%%
%%%%%%%%%%%%%%%%%%%%%%%%%%%%%%%%%%%%%%%%%%%%%%%%%%%%%%%%%%%%%%%%%%%
%%%%
Assuming convergence of the action of the 
insertion operator 
 $T(y)$ acting on the genus zero differential forms 
 of $C^n(V, 0)$,    
we define 
the genus-raising operator given by   
$\Delta^{(1)}(w_{\pm 1})= \sum_{ b_k \in V_{(k)}} T(y)$,  
with respect to the corresponding bilinear pairing denoted on $V$.  
%%
%%%%%%%%%%%%%%%%%%%%%%%%%%%%%%%%%%%%%%%%%%%%%%%%%%%%%%%%%
%%%%
For each $1 \le a \le g$,  
let $\{b_a \}$  denote a homogeneous  $V$-basis and let $\{\bbar_a\}$ 
be the dual basis.    
%%
%%%%%%%%%%%%%%%%%%%%%%%%%%%%%%%%%%%%%%%%%%%%%%%%%%%%%%%%
%%%%
Define for $1 \le a \le g$, by \eqref{gensher}  
for formal $\rho_a$. 
Then $\{b_{-a}\}$
 is a dual basis with respect to the bilinear pairing
  $\langle .,. \rangle_{\rho_a}$ 
with adjoint given by \eqref{tuta} 
for $u$ quasiprimary of weight $N$.
%%
%%%%%%%%%%%%%%%%%%%%%%%%%%%%%%%%%%%%%%%%%%%%%%%%%%%%%%%%%
%%%%
Let ${\bf b}_+=b_1\otimes\ldots \otimes b_g$
 denote an element of a $V^{\otimes g}$-basis. 
Let $w_a$ for $a\in\{-1, \ldots, - g, 1, \ldots, g\}$ 
be $2g$ formal variables.   
%% 
%%%%%%%%%%%%%%%%%%%%%%%%%%%%%%%%%%%%%%%%%%%%%%%%%%%%%%%%%%%%
%%%%
Denote ${\bf t}_g$ $=$ $({\bf w}_{\pm g}$, ${\bf \rho}_g)$ $=$  
 $\left(w_1, w_{-1}\right.$, $\rho_1$, $\ldots$, $w_g$, $w_{-g}$, 
$\rho_g$, $w_1$, $w_{-1}$, $\rho_1$, $\ldots$, $w_g$,  
$w_{-g}$, $\left. \rho_g \right)$.  
%%
%%%%%%%%%%%%%%%%%%%%%%%%%%%%%%%%%%%%%%%%%%%%%%%%%%%%%%%%%%%%%%%%%%%%%%
%%%%
For ${\bf w}_{\pm g}= (w_1, w_{-1},  \ldots, w_g, w_{-g})$,   
 the insertion operator $T({\bf t}_g)$  
 defines the genus-raising operator  
$\Delta^{(g)}({\bf w}_{\pm \kappa})$ by 
%%
%%%%%%%%%%%%%%%%%%%%%%%%%%%%%%%%%%%%%%%%%%%%%%%%%%%%%%%%%%%%%%%%%%%%%%
%%%%
\begin{eqnarray}
\label{benderoz}
&& \F_V^{(g)} ({\bf x}_n, {\bf w}_{\pm g})  
= \Delta^{(g)} ({\bf w}_{\pm g}).\F_V^{(0)} ({\bf x}_n)  
%%
%%%%%%%%%%%%%%%%%%%%%%%%%%%%%%%%%%%%%%%%%%%%%%%%%%%%%%%%%%%%%%%%
\\
&& = \sum\limits_{{\bf b}_+} T\left({\bf Y}({\bf t}_g) 
\prod_{a=1}^g \rho_a^{\wt(b_a)} \right).  
\langle w', {\bf Y}({\bf x}_n)\rangle_1 \; {\bf dz}^{\wt({\bf v}_n)}  
%%
%%%%%%%%%%%%%%%%%%%%%%%%%%%%%%%%%%%%%%%%%%%%%%%%%%%%%%%%%%%%%%%%
\nn
\nonumber
  && = \sum\limits_{{\bf b}_+} 
\F_V^{(0)}({\bf x}_n, b_1, w_1; \bbar_1, w_{-1}; \ldots; b_g, w_g;  
\bbar_g, w_{-g})    
\prod_{a=1}^g \rho_a^{\wt(b_a)} {\bf dw}^{\wt({\bf b}_+)}  
{\bf dz}^{\wt({\bf v}_n)}. 
\end{eqnarray} 
%%
%%%%%%%%%%%%%%%%%%%%%%%%%%%%%%%%%%%%%%%%%%%%%%%%%%%%%%%%%%%%%%%%%%%%%%%%%%%%%%%%%%%%%
%%%%
 In \eqref{benderoz} the sum is over any basis $\{{\bf b}_+\}$ of $V^{\otimes g}$. 
This corresponds to the genus $g$ vertex operator algebra $V$ $n$-point 
function in the Schottky parametrization \cite{TW, T, Zo}. 
%%
%%%%%%%%%%%%%%%%%%%%%%%%%%%%%%%%%%%%%%%%%%%%%%%%%%%%%%%%%%%%%%%%%%
%%%%
Note that in \eqref{benderoz}  
we insert a sequence of vertex operators corresponding to 
$\left(b_1 \right.$, $w_1$; $\bbar_1$, $w_{-1}$; $\ldots$; $b_\kappa$, $w_g$;  
$\bbar_g$, $\left. w_{-g}\right)$ after the vertex operators for 
the arguments ${\bf x}_n$. 
%%

%%%%%%%%%%%%%%%%%%%%%%%%%%%%%%%%%%%%%%%%%%%%%%%%%%%%%%%%%%%%%%%%
%%%%
One might define \eqref{benderoz} containing a different order of 
vertex operators taking into account the comutation properties of $V$.  
%%
%%%%%%%%%%%%%%%%%%%%%%%%%%%%%%%%%%%%%%%%%%%%%%%%%%%%%%%%%%%%%%%%%%%%%%%%%%%%%%%%%%
%%%% 
This definition is motivated by the sewing relation \eqref{telefon}
 and ideas in \cite{TW1, TW, MT2}. 
This is similar to the sewing analysis employed in \cite{Z2, G}. 
%%
%%%%%%%%%%%%%%%%%%%%%%%%%%%%%%%%%%%%%%%%%%%%%%%%%%%%%%%%%%%%%%%%%%%%%
%%%%
For all ${\bf x}_n$,    
$C^n(V,  g)$, $g \ge 1$, $n \ge 0$, are the spaces 
of all genus $g$ differential forms obtained via the applications 
of genus-raising operators 
$\Delta^{(g)}({\bf w}_{\pm g})$.  
%%
%%%%%%%%%%%%%%%%%%%%%%%%%%%%%%%%%%%%%%%%%%%%%%%%%%%%%%%%%%%%%%%%%%%%%%  
%%%%
As we see from the definition of the level-raising operator \eqref{benderoz},  
the differential forms 
$\F_V^{(g)}({\bf x}_n)$     
depend on the parameters $\rho_a$   
via the dual vectors ${\bf b}_-=b_{-1}\otimes\ldots \otimes b_{-g}$
 as in \eqref{gensher}.
%% 
%%%%%%%%%%%%%%%%%%%%%%%%%%%%%%%%%%%%%%%%%%%%%%%%%%%%%%%%%%%%%%%%%%%%%
%%%%
 In particular, setting $\rho_a=0$ for some $1 \le a\le g$,  
$\F_V^{(g)}({\bf x}_n)$ then degenerates to a level   
$g-1$ differential form.  
%%
%%%%%%%%%%%%%%%%%%%%%%%%%%%%%%%%%%%%%%%%%%%%%%%%%%%%%%%%%%%%%%%%%%
%%%%
Note that in all examples of our construction 
given in the Sections \ref{genuszero0}--\ref{apala}, the corresponding convergence of 
$n$-point vertex operator algebra functions was shown in \cite{TW}.  
%%

%%%%%%%%%%%%%%%%%%%%%%%%%%%%%%%%%%%%%%%%%%%%%%%%%%%%%%%%%%%%%%%%%%%%%%%%%%%%%
%%%%
In the vertex operator algebra case, the genus $g$ horizontal differentials 
 are $\delta^{(g)}({\bf x}_{n+1})$ $=$ $B^{(g)}({\bf x}_{n+1})$ 
acting on differential forms 
$\F_V^{(g)}\left({\bf x}_n\right) \in C^n(V, g)$ giving  
an element of $C^{n+1}(V, g)$, where 
%%   
%%%%%%%%%%%%%%%%%%%%%%%%%%%%%%%%%%%%%%%%%%%%%%%%%%%%%%%%%%%%%%%%%%%%%%
%%%%
$B^{(g)}({\bf x}_{n+1})=\sum_{i=1}^n \sum_{l\ge 0}   
f^{(g)}_{i, l}({\bf x}_{n+1})   
\;T_{i'}(v_n(l'))$.    
%%
%%%%%%%%%%%%%%%%%%%%%%%%%%%%%%%%%%%%%%%%%%%%%%%%%%%%%%%%%%%%%%%%%%%%%%%%%%
%%%%
The actual form of the operator $B^{(g)}({\bf x_n})$  
 depends on $g$  
and on the way $\F_V^{(g)}({\bf x}_n)$ is defined 
\cite{Z1, TW, TW1, TZ1, MTZ}.  
%%
%%%%%%%%%%%%%%%%%%%%%%%%%%%%%%%%%%%%%%%%%%%%%%%%%%%%%%%%%
%%%%
Then the chain condition for this operator,  
in particular, in examples of the Sections \ref{genuszero0}--\ref{apala}, 
 is given by  
%%
%%%%%%%%%%%%%%%%%%%%%%%%%%%%%%%%%%%%%%%%%%%%%%%%%%%%%%%%%%%%%
%%%%
\begin{eqnarray}
\label{prodo}
&& \delta^{(g)}(v_{n+2}, z_{n+2}; {\bf x}_{n+1}) 
\; \delta^{(g)}(v_{n+1}, z_{n+1}; {\bf x}_n)  
\\
&&
\nonumber 
\quad =\sum\limits_{k=1}^n  \sum\limits_{j \ge 0 \atop j'\ge 0} 
\sum\limits_{k'=1}^{n+1} 
f^{(g)}_{N, j} (x_{n+2}, z_k) \;    
f^{(g)}_{N', j'}(x_{n+1}, z_{k'}) \;       
 T_k(v_{n+2}(j)) \;  
  T_{k'}(v_{n+1}(j')). 
\end{eqnarray}
%%
%%%%%%%%%%%%%%%%%%%%%%%%%%%%%%%%%%%%%%%%%%%%%%%%%%%%%%%%%%%%%%%%%%%%%%%%%
%%%% 
In the Sections \ref{genuszero0}--\ref{apala}
 we consider specific examples of this Section construction  
provided by consideration of a vertex operator algebra $n$ functions 
considered on genus $g$
 Riemann surfaces \cite{FK} forms in the Schottky uniformization 
procedure. 
%%
%%%%%%%%%%%%%%%%%%%%%%%%%%%%%%%%%%%%%%%%%%%%%%%%%%%%%%%%%%%%%%%%%%%%%%%%%
%%%%
When the formal parameters ${\bf z}_n$ are associated to local 
coordinates on Riemann surfaces, 
the general vertex operator algebra 
 reduction formulas take their explicit form \cite{TW}. 
%%
%%%%%%%%%%%%%%%%%%%%%%%%%%%%%%%%%%%%%%%%%%%%%%%%%%%%%%%%%%%%%%%%%%%%%%%%
%%%%
We will see that the functions 
$f^{(g)}_{N, j} (x, z_k)$ depend not only on $z$ from $x=(v, z)$, 
but also on $v$. 
%%
%%%%%%%%%%%%%%%%%%%%%%%%%%%%%%%%%%%%%%%%%%%%%%%%%%%%%%%%%%%%%%%%%%
%%%%
We will show also that the chain condition \eqref{prodo}
 may be related, in particular, to 
 the corresponding  
Ward indentity conditions \eqref{wardy} and of the Proposition \eqref{kurst}.  
%%
%%%%%%%%%%%%%%%%%%%%%%%%%%%%%%%%%%%%%%%%%%%%%%%%%%%%%%%%%%%%%%%%
%%%%
The results of Theorem 13.1 of \cite{G}, 
(see also \cite{Z2})  shows that 
for an admissible vertex operator algebra $V$, elements 
$\F_V^{(\g)} ({\bf x}_n) \in C^n(V, g)$ are 
  absolutely and locally uniformly convergent
 on the corresponding sewing domain.   
%%
%%%%%%%%%%%%%%%%%%%%%%%%%%%%%%%%%%%%%%%%%%%%%%%%%%%%%%%%%%%%%%%%%%%
%%%% 
Thus, the consideration above leads us to the following   
\begin{proposition}
\label{normodor}
For an admissible vertex operator algebra $V$, 
 application of the genus-raising operator \eqref{benderoz}    
results in families of complexes  
$\left(C^n(V,  g), \delta^{(g)}({\bf x}_n)\right)$  
of convergent canonical differential forms. \hfill $\qed$
\end{proposition}
%%
%%%%%%%%%%%%%%%%%%%%%%%%%%%%%%%%%%%%%%%%%%%%%%%%%%%%%%%%%%%%%%%%%%%%%%%%%%%
%%
\subsection{The main result: the torsor structure of cohomology} 
\label{puzan}
%%%%
The chain condition \eqref{prodo} for the coboundary operator 
may be solved as a functional equation in various ways. 
%%
%%%%%%%%%%%%%%%%%%%%%%%%%%%%%%%%%%%%%%%%%%%%%%%%%%%%%%%%%%%%%%%
%%%%
In the examples of the Sections \ref{genuszero0}--\ref{apala}
 we involve an ansatz  
leading to corresponding Ward identities.  
%%
%%%%%%%%%%%%%%%%%%%%%%%%%%%%%%%%%%%%%%%%%%%%%%%%%%%%%%%%%%%%%%%%
%%%%
The main idea is to cut off an infinite expansion of the leading term 
of the form $1/(x-y)$ by terms compensation of $f_l(x)z^l$ terms summed 
for $u_{n+1} \in V$ states by means of the reduction formulas.  
%%
%%%%%%%%%%%%%%%%%%%%%%%%%%%%%%%%%%%%%%%%%%%%%%%%%%%%%%%%%%%%%%%%
%%%%
Now we formulate the main result of this paper.  
%%
%%%%%%%%%%%%%%%%%%%%%%%%%%%%%%%%%%%%%%%%%%%%%%%%%%%%%%%%%%%%
%%%%
\begin{theorem}
\label{coca}
The invariant multiple cohomology $H^n(V, g)$ of $V$  
with the reduction operators $B^{(g)}({\bf x}_n)$, $n \ge 0$ 
having the chain property for  
of families of complexes   
$\left(C^n(V, g), \delta^{(g)}({\bf x}_n) \right)$ 
considered on genus $g$ Riemann surfaces formed in the Schottky procedure, 
are equivalent to the factor space of the spaces of products of sums of 
$\mathcal L \left(f^{(g)}_{i, l}({\bf a}^{-1}_{i, n, l}.\widetilde{\bf x}_n)\right)$ of 
reduction functions  
$f^{(g)}_{i, l}({\bf a}^{-1}_{i, n, l}.\widetilde{\bf x}_n)$, 
$l \ge 0$, $1 \le i$, $j \le n$,   
with transformed arguments $\widetilde{\bf x}_n$
 according to the corresponding vertex algebra elements.    
\end{theorem}
%%
%%%%%%%%%%%%%%%%%%%%%%%%%%%%%%%%%%%%%%%%%%%%%%%%%%%%%%%%%%%%%%%%%%%%%%%%%%%
%%%%
\begin{proof} 
In \cite{BZF} it was shown that the genis zero $n$-point differential form 
 \eqref{bucharas} 
is invariant with respect to the group ${\rm Aut}\; \mathcal O_n$ of 
independent transformations  
 of local variables $(z_1, \ldots, z_n)$ on a complex curve. 
%%
%%%%%%%%%%%%%%%%%%%%%%%%%%%%%%%%%%%%%%%%%%%%%%%%%%%%%%%%%%%%%%%%%%%%%%%%%%%
%%%%
As a result of application of the level-raising operator $\Delta^{(1)}(y_1)$
on \eqref{bucharas} we obtain another differential form 
containing a matrix element multiplied with $v$-part of $y=((b',w_+),(b,w_-))$ 
weighted differentials.  
%%
%%%%%%%%%%%%%%%%%%%%%%%%%%%%%%%%%%%%%%%%%%%%%%%%%%%%%%%%%%%%%%%%%%%%%%%%%
%%%%
Thus, the resulting differential form $\F_V^{(1)}({\bf x}_n) \in C^n(V,  1)$
 remains invariant with respect to the transformations of ${\rm Aut}\; \mathcal O_n$. 
The same argument is applicable to the result of the 
genus-raising operator $\Delta^{(g)}({\bf y}_g)$ action. 
%%
%%%%%%%%%%%%%%%%%%%%%%%%%%%%%%%%%%%%%%%%%%%%%%%%%%%%%%%%%%%%%
%%%%
Now let us show that the reduction operators $B^{(g)}({\bf x}_n)$
can be presented in a canonical form. 
%%
%%%%%%%%%%%%%%%%%%%%%%%%%%%%%%%%%%%%%%%%%%%%%%%%%%%%%%%%%%%%%
%%%%
Let us first recall the definitions of torsors and twists 
with respect to a group required for the proof. 
%% 
%%%%%%%%%%%%%%%%%%%%%%%%%%%%%%%%%%%%%%%%%%%%%%%%%%%%%%%%%%%%
%%%%
%%
Let $\mathfrak G$ be a group, and $M$ a non-empty set.  
Then $M$ is called a $\mathfrak G$-torsor  
if it is equipped with a simply transitive right action of $\mathfrak G$, 
i.e., given $\eta$, $\widetilde \eta \in M$, 
there exists a unique $h \in \mathfrak G$ such that 
$\eta \cdot h = \widetilde \eta$, 
where for $h$, $\widetilde{h} \in \mathfrak G$ the right action is given by 
$\eta \cdot (h \cdot \widetilde{h}) = (\eta \cdot  h) \cdot \widetilde{h}$. 
The choice of any $\eta \in M$ allows 
us to identify $M$ with $\mathfrak G$ by sending
 $\eta  \cdot h$ to $h$.
%%
%%%%%%%%%%%%%%%%%%%%%%%%%%%%%%%%%%%%%%%%%%%%%%%%%%%%%%%%%%%%%%%%%%%%%%%%
%%%%
Let $\mathfrak G$ be the group ${\rm Aut}\; \mathcal O$ of  
coordinate changes on a smooth complex curve 
$S$ generated by the transformations  
$t \mapsto \rho (t)$.       
%%
%%%%%%%%%%%%%%%%%%%%%%%%%%%%%%%%%%%%%%%%%%%%%%%%%%%%%%%%%%%%%%%%%%%%%%%%%%%%%%%%%%%%%%%%%%%%%%
%%%%
Let $V$ be a vertex operator algebra. 
For a ${\rm Aut}\; \mathcal O$-torsor $\Xi$,    
one defines the $\Xi$-twist of $V$ as the set  
$\mathcal |V_\Xi = V \; {{}_\times \atop {}^{  {\rm Aut} \; \mathcal O}     } \Xi    
=  V \times  \Xi/  \left\{ ({\rm v}.v, \eta)  \sim  (v, a.\eta) \right\}$, 
for $\eta \in \C$, $a \in {\rm Aut}\; \mathcal O$, ${\rm v} \in {\rm End}(V)$, and $v\in V$. 
%% 
%%%%%%%%%%%%%%%%%%%%%%%%%%%%%%%%%%%%%%%%%%%%%%%%%%%%%%%%%%%%%%%%%%%%%%%%%%%%%%%%%%%%%%%%%%
%%%%
We denote by ${\it Aut}_{ p }$  the set of all 
coordinates $t_p$ on a disk $D_p$.       
%%
%%%%%%%%%%%%%%%%%%%%%%%%%%%%%%%%%%%%%%%%%%%%%%%%%%%%%%%%
%%%
It was proven in \cite{BZF} that   
the group ${\rm Aut}\; \mathcal O$   
 acts naturally on ${\it Aut}_p$,   
 and it is an ${\rm Aut}\; \mathcal O$-torsor.    
In what follows, we assume that all elements of the group 
${\rm Aut}\; \mathcal O$ are invertible. 
%%

%%
%%%%%%%%%%%%%%%%%%%%%%%%%%%%%%%%%%%%%%%%%%%%%%%%%%%%%%%%%%%%%%%%%%%%%
%%%%
Recall that $B^{(g)}({\bf x}_n)$ are in general operators 
combining multiplication of an element  
$\F_V^{(g)}({\bf x}_n)\in C^n(V, \kappa)$  
by the reduction functions $f^{(g)}_{i,l}({\bf x}_n)$, 
 $1 \le i \le n$, $l \ge 0$, depending on 
a vertex operator algebra $V$ element $v_n$ 
with insertion of $v_n(l)$-mode into $i$-th position at 
$\F_V^{(g)}({\bf x}_{n-1})$.  
%%
%%%%%%%%%%%%%%%%%%%%%%%%%%%%%%%%%%%%%%%%%%%%%%%%%%%%%%%%%%%%%%%%% 
%%%%
In \cite{BZF} the torsor structure of zero-level differential forms 
was used in order to show their canonicity. 
%%
%%%%%%%%%%%%%%%%%%%%%%%%%%%%%%%%%%%%%%%%%%%%%%%%%%%%%%%%%%%%%%%%
%%%%
In that formulation,
for $V$-automorphisms of $V$ represented by the action of $v$ on ${\bf v}_i$
$1 \le i \le n$,  
  torsors are defined in terms of the equivalence  
%%
%%%%%%%%%%%%%%%%%%%%%%%%%%%%%%%%%%%%%%%%%%%%%%%%%%%%%%%%%%%%%%%%%
%%%%
$\F_V^{(0)}\left(v_{n+1}.{\bf v}_n, {\bf z}_n\right)  
\sim \F_V^{(0)}\left({\bf v}_n, {\bf a}_n.{\bf z}_n\right)$,   
where ${\bf a}_n$ denote corresponding automorphisms 
 of parameters ${\bf z}_n$. 
%%
%%%%%%%%%%%%%%%%%%%%%%%%%%%%%%%%%%%%%%%%%%%%%%%%%%%%%%%%%%%%%%%%%
%%%%
Using the definition \eqref{benderoz},  
it is easy to see that the last equivalence extends to the higher genus case. 
%%
%%%%%%%%%%%%%%%%%%%%%%%%%%%%%%%%%%%%%%%%%%%%%%%%%%%%%%%%%%%%%%%%%%%
%%%% 
By applying that equivalence to \eqref{benderoz} 
we transfer the action all endomorphisms $v_i(l).$ into  
${\bf a}_n$ of reduction operators $B^{(g)}({\bf x}_j)$.  
and functions $f^{(g)}_{i, l}({\bf x}_j)$.  
%%
%%%%%%%%%%%%%%%%%%%%%%%%%%%%%%%%%%%%%%%%%%%%%%%%%%%%%%%%%%%%%%%%% 
%%%%
Using the reduction formulas of the Subsections \ref{perdas} and \ref{systemo}
we find the multiple cohomology formulas for an admissible $V$
 considered on Riemann surfaces. 
%%
%%%%%%%%%%%%%%%%%%%%%%%%%%%%%%%%%%%%%%%%%%%%%%%%%%%%%%%%%%%%%%%%%%%%%%%%%%%%%
%%%%
According to the general vertex operator algebra reduction formulas mentioned above, 
 the differentials forms 
$\F_V^{(g)}({\bf x}_{n+1})$ are expanded in the unique way 
in terms of the differential forms $\F_V^{(g)}({\bf x}_n)$.  
%%
%%%%%%%%%%%%%%%%%%%%%%%%%%%%%%%%%%%%%%%%%%%%%%%%%%%%%%%%%%%%%%
%%%%
Denote by ${\bf a}_{n+1}$ we denote the set of automorphism elements 
corresponding to actions $v_{n+1}(l).v_i$,
 ${\bf a}^{-1}_{n+1}$ their inverse elements.  
%%
%%%%%%%%%%%%%%%%%%%%%%%%%%%%%%%%%%%%%%%%%%%%%%%%%%%%%%%%%%%%%%%%%%%
%%
Since elements of the group of local coordinates transformation 
are assumed invertible, there exist 
$\widetilde{x}_{n+1}=({\bf v}_{n+1}, {\bf a}^{-1}_{n+1}. {\bf x}_{n+1})$, and 
%% 
%%%%%%%%%%%%%%%%%%%%%%%%%%%%%%%%%%%%%%%%%%%%%%%%%%%%%%%%%%%%%%%%%%%%%%%%%
%%
\begin{eqnarray*} 
\delta^{(g)}({\bf x}_{n+1}).\F_V^{(g)}({\bf x}_n) 
= B^{(g)}({\bf x}_{n+1}).\F_V^{(g)}\left({\bf x}_n\right) 
=\sum\limits_{i=1}^n \sum\limits_{l \ge 0} 
f^{(g)}_{i, l}({\bf x}_{n+1}) \; T_i(v_{n+1}(l)). 
 \F_V^{(g)}({\bf x}_n) &&
%%
%%%%%%%%%%%%%%%%%%%%%%%%%%%%%%%%%%%%%%%%%%%%%%%%%%%%%%%%
%%
\nn
  = \sum\limits_{i=1}^n \sum\limits_{l \ge 0} 
 f^{(g)}_{i, l}({\bf x}_{n+1}).  
 \F_V^{(g)}\left(T_i(v_{n+1}(l)).{\bf x}_n\right) 
\qquad \qquad \qquad \qquad \qquad \qquad \qquad \qquad &&
%%
%%%%%%%%%%%%%%%%%%%%%%%%%%%%%%%%%%%%%%%%%%%%%%%%%%%%%%5
\nn
 = \sum\limits_{i=1}^n \sum\limits_{l \ge 0} 
 f^{(g)}_{i, l}({\bf x}_{n+1}).   
 \F_V^{(g)}\left(v_1, z_1;  \ldots, v_{n+1}(l).v_i, z_i;  
\ldots; v_n, z_n\right)   
\qquad \qquad \qquad &&
%%
%%%%%%%%%%%%%%%%%%%%%%%%%%%%%%%%%%%%%%%%%%%%%%%%%%%%%%%%%%%%%%%%%%
%%
\nn
 =  \sum\limits_{i=1}^n \sum\limits_{l \ge 0}  
 f^{(g)}_{i, l}({\bf x}_{n+1}).   
 \F_V^{(g)}\left(v_1, z_1;  \ldots, v_i, a_{i, n+1, l}.z_i;  
\ldots; v_n, z_n\right)   \qquad \qquad \qquad &&
%%
%%%%%%%%%%%%%%%%%%%%%%%%%%%%%%%%%%%%%%%%%%%%%%%%%%%%%%%%%%%%%%%%%%
%%
\nn
 = \sum\limits_{i=1}^n \sum\limits_{l \ge 0}  
 f^{(g)}_{i, l}({\bf x}_{n+1})\; T_i(a_{i, n+1, l}).   
 \F_V^{(g)}\left({\bf x}_n\right)
 = \sum\limits_{i=1}^n \sum\limits_{l \ge 0}  
 f^{(g)}_{i, l}({\bf a}^{-1}_{n+1}.\widetilde{\bf x}_{n+1}).    
 \F_V^{(g)}\left(\widetilde{\bf x}_n\right) \qquad &&
%%
%%%%%%%%%%%%%%%%%%%%%%%%%%%%%%%%%%%%%%%%%%%%%%%%%%%%%%%%%%%%%%%%%%
%%
\nn
= B^{(g)}({\bf a}^{-1}_{n+1}.\widetilde{\bf x}_{n+1}). 
\F_V^{(g)}\left(\widetilde{\bf x}_n\right) 
=\delta^{(g)}({\bf a}^{-1}_{n+1}.\widetilde{\bf x}_{n+1}). 
\F_V^{(g)}\left(\widetilde{\bf x}_n\right).  \qquad \qquad &&
\end{eqnarray*}
%%
%%%%%%%%%%%%%%%%%%%%%%%%%%%%%%%%%%%%%%%%%%%%%%%%%%%%%%%%%%%%%%%%%%%
%%%%
One can see from the last formulas that the coboundary operators 
as well as the corresponding cohomology is expressible through 
a set of transformed local coordinates while vertex operator algebra elements 
play the role of parameters. 
%%
%%%%%%%%%%%%%%%%%%%%%%%%%%%%%%%%%%%%%%%%%%%%%%%%%%%%%%%%%%%%%%%%%%%%%%%
%%
It is easy to write $a_{i, n, l}$ in the exact form as in \cite{BZF} 
(see Appendix \ref{groupppp}). 
%%
%%%%%%%%%%%%%%%%%%%%%%%%%%%%%%%%%%%%%%%%%%%%%%%%%%%%%%%%%%%%%%%%%%%
%%%%
Indeed, the action 
of any endomorphism ${\rm v}$ can be represented as a homomorphism  
\cite{BZF, GR}, 
corresponding to an automorphism $a_{i, n, l}$
%%
%%%%%%%%%%%%%%%%%%%%%%%%%%%%%%%%%%%%%%%%%%%%%%%%%%%%%%%%%%%%%%%%%%%%%%
%%
\begin{eqnarray}
\label{poperator}
 &&v_n(l).v = P\left( a_{i, n, l} \right).v =  
\exp \left( \sum\limits_{k >  0} (k+1)\; \beta_k \; L(k) \right) \beta_0^{L_W(0)}.v,   
%%
%%%%%%%%%%%%%%%%%%%%%%%%%%%%%%%%%%%%%%%%%%%%%%%%%%%%%%%%%%%5
\nn
\label{prostoryad0}
&&a_{i, n, l}.z_i 
= \exp \left(  \sum\limits_{k > -1} \beta_k\; z_i^{k+1} \partial_{z_i} \right) 
\beta_0^{z_i \partial_{z_i}}.z_i
= \sum\limits_{p \ge 1} a_{p, i, n, l}\;  z_i^p, 
\end{eqnarray}
%%%%
where coefficients $a_{p, i, n, l}$ are expressed explicitly in terms of $\beta_k$. 
%%
%%%%%%%%%%%%%%%%%%%%%%%%%%%%%%%%%%%%%%%%%%%%%%%%%%%%%%%%%%%%%%%%%%%%%
%%%%
According to the equivalence of 
$\F_V^{(0)}\left(v_{n+1}.{\bf v}_n, {\bf z}_n\right)$ and   
 $\F_V^{(0)}\left({\bf v}_n, {\bf a}_n.{\bf z}_n\right)$,   
 the action of elements $a_{i, n+1, l}$ on formal parameters  
$z_i$, $1 \le i \le n$, corresponds to the action 
 of vertex operator algebra modes $v_{n+1}(l).$ acting on $v_i$. 
%%
%%%%%%%%%%%%%%%%%%%%%%%%%%%%%%%%%%%%%%%%%%%%%%%%%%%%%%%%%%%%%%%%%%%%%
%%%%
The last identity gives us an automorphism-based, coordinate independent, 
and vertex operator algebra mode parametrized 
form of the Zhu 
reduction formulas.   
%%
%%%%%%%%%%%%%%%%%%%%%%%%%%%%%%%%%%%%%%%%%%%%%%%%%%%%%%%%%%%%%%
%%%%
In addition to that it establishes a relation between 
the coboundary operators
 $\delta^{(g)}({\bf x}_{n+1})$ acting by $T_i(v(l))$-insertions  
with the invariant form $\delta^{(g)}({\bf a}_{n+1}.{\bf x}_{n+1})$ of  
 coboundary operators acting by $T_i(a_{i, n+1, l})$-insertions.  
%% 
%%%%%%%%%%%%%%%%%%%%%%%%%%%%%%%%%%%%%%%%%%%%%%%%%%%%%%%%%%%%%%%%%%%%%%%%%%%%
%%%%
Note that according to the construction of differential forms \eqref{benderoz},   
for a fixed set of arguments ${\bf x}_n$, 
$\F_V^{(g)}\left({\bf x}_n\right)$ is defined uniquely 
up to a set of complex parameters ${\bf w}_{\pm g}$ not involved 
in the action of corresponding differential $\delta^{(g)}({\bf x}_n)$.  
%%
%%%%%%%%%%%%%%%%%%%%%%%%%%%%%%%%%%%%%%%%%%%%%%%%%%%%%%%%%%%%%%%%%%%%%%%%%%
%%%%
Taking into account that the reduction formulas, i.e.,  
$\F_V^{(g)}\left({\bf x}_n\right)
= \prod_{j=n}^1 
\sum_{l_j \ge 0}
 \sum_{i_j=1}^j f^{(g)}({\bf x}_j)\; T_{i_j}(a_{i_j, l_j})$.   
$\F_V^{(g)}\left(x_0 \right)$. 
%%
%%%%%%%%%%%%%%%%%%%%%%%%%%%%%%%%%%%%%%%%%%%%%%%%%%%%%%%%%%%%%%%%%%%%%%%%%%%%
%%%% 
are applicable to 
any $n+1$-point differential form of the space $C^{n+1}(V, g)$, 
then, with the fixed set of arguments ${\bf x}_n$,  
all differential forms $\F_V^{(g)}\left({\bf x}_n\right)$ 
do belong to ${\rm Im}\; \delta^{(g)} ({\bf x}_{n-1})$. 
%% 
%%%%%%%%%%%%%%%%%%%%%%%%%%%%%%%%%%%%%%%%%%%%%%%%%%%%%%%%%%%%%%%%%%%%%%%%%%%%%%  
%%%%
Suppose the coboundary operators $\delta^{(g)}({\bf x}_{n+1})$, $n \ge 0$,
 satisfy the chain conditions \eqref{prodo}. 
%% 
%%%%%%%%%%%%%%%%%%%%%%%%%%%%%%%%%%%%%%%%%%%%%%%%%%%%%%%%%%%%%%%%%%%%%%%%%%%%%%%
%%%%
Define the space $\mathcal L({\bf a}^{-1}_n.\widetilde{\bf x}_n)$ 
$=$ $\left\{\prod_{j=n}^1 \right.$  
$\sum_{l_j \ge 0}$
 $\sum_{i_j=1}^j$ $f^{(g)}$ $\left.({\bf a}^{-1}_j.\widetilde{\bf x}_j) \right\}$  
%%
%%%%%%%%%%%%%%%%%%%%%%%%%%%%%%%%%%%%%%%%%%%%%%%%%%%%%%%%%%%%%%%%%%%%
%%
$=$ $\left\{\prod_{j=n}^1 \right.$  
$\sum_{l \ge 0, \; 1 \le i \le n}$ 
$f^{(g)}_{i, l}$ $(a^{-1}_{i, j, l}$. $\left. \widetilde{x}_j) \right\}$ 
 of products of sums of transformed reduction functions    
$f^{(g)}_{i, l}\left(a^{-1}_{i, j, l}.\widetilde{x}_j\right)$, $1 \le i \le n$,  
 $l\ge 0$, $1 \le j \le n$. 
%%
%%%%%%%%%%%%%%%%%%%%%%%%%%%%%%%%%%%%%%%%%%%%%%%%%%%%%%%%%%%%%%%%%%%%%%%%%%%%%%%%
%%%%
By adapting the last formula of the previous Section for a 
vertex operator algebra setup, we see that, since the reduction is 
performed to the level of the partition function on both sides of the 
factor space, we obtain the following expression 
for $n$-th cohomology 
$H^n(V, g)=\mathcal L({\bf a}^{-1}_n.\widetilde{\bf x}_n)/
\mathcal L({\bf a}^{-1}_{n-1}.\widetilde{\bf x}'_{n-1})$.  
\end{proof}
%% 
%%%%%%%%%%%%%%%%%%%%%%%%%%%%%%%%%%%%%%%%%%%%%%%%%%%%%%%%%%%%%%%
%%%%
The form of cohomology given in Theorem \eqref{coca}
is more useful since it is expressed in terms of reduction functions
 depending on complex variables 
with vertex operator algebra elements as extra parameters.
%%
%%%%%%%%%%%%%%%%%%%%%%%%%%%%%%%%%%%%%%%%%%%%%%%%%%%%%%%%%%%%%%%%%%%%%%%%%%%%%%%%%%%%%%
%%%%
It is important that in the case of modular invariant $n$-point functions, 
 the standard cohomology remains modular invariant. 
%%
%%%%%%%%%%%%%%%%%%%%%%%%%%%%%%%%%%%%%%%%%%%%%%%%%%%%%%%%%%%%%%%%%%%%%%%%%%%%%%%%%%%%%%%
%%%%
In examples of the Sections \ref{genuszero0}--\ref{apala} 
we see from the formulas of Lemma \ref{meerertrtwttr}, Proposition \ref{mordot} 
\eqref{mardot}, and Proposition \ref{kartavo} that   
the standard cohomology of 
families of complexes for various $g$ is modular invariant. 
%%
%%%%%%%%%%%%%%%%%%%%%%%%%%%%%%%%%%%%%%%%%%%%%%%%%%%%%%%%%%%%%%%%%%%%%%
%%%%
 Then $H^n(V,  g)= \prod_{i=1}^n B^{(g)}(\gamma.x_i) 
/\prod_{j=1}^{n-1} B^{(g)}(\gamma.x_j)$ is modular invariant. 
%%
%%%%%%%%%%%%%%%%%%%%%%%%%%%%%%%%%%%%%%%%%%%%%%%%%%%%%%%%%%%%%%%%%%%
%%%%%
Note that in many cases 
(as we will see in the examples given in the Sections \ref{genuszero0}--\ref{apala})
 the coefficient functions 
$f^{(g)}_{i, l}({\bf a}_n.{\bf x}_n)$ have 
 vertex operator algebra elements as parameters.
%%
%%%%%%%%%%%%%%%%%%%%%%%%%%%%%%%%%%%%%%%%%%%%%%%%%%%%%%%%%%%%%%%%%%%%%%%%%%5
%%%%
Nevertheless, that functions may depend on elements $u$, $u'\in V$
or their weights $\wt(u)$, $\wt(u')$ 
of $V$ components of $x_{n+1}$ and $x_n$. 
%%
%%%%%%%%%%%%%%%%%%%%%%%%%%%%%%%%%%%%%%%%%%%%%%%%%%%%%%%%%%%%%%%%%%%%%%%%%%%%%%%%%%%%%%
%%%%%%%%%%%%%%%%%%%%%%%%%%%%%%%%%%%%%%%%%%%%%%%%%%%%%%%%%%%%%%%%%%%%%%%%%%%%%%%%%%%%%%
%%
\section{Example: genus zero multiple cohomology on Riemann surfaces}
\label{genuszero0}
In the Sections \ref{genuszero0}--\ref{apala} we consider particular examples 
of vertex operator algebra cohomology considered on Riemann surfaces.
%%
%%%%%%%%%%%%%%%%%%%%%%%%%%%%%%%%%%%%%%%%%%%%%%%%%%%%%%%%%%%%%%%%%%%%%%
%%%%
We let $\mathcal P_n$ denote the space of polynomials 
with complex coefficients of degree at most $n$.  
%%
%%%%%%%%%%%%%%%%%%%%%%%%%%%%%%%%%%%%%%%%%%%%%%%%%%%%%%%%%%%%%%%%%%%%%%%%%%%%%%%%%%%
%%%%
\subsection{The definition of $n$-point functions}
\label{logika}
The Schottky uniformization is a particular application to Riemann surfaces 
 of the general method of increasing homology level by attaching abstract loops 
to an algebraic/geometric structure. 
%%
%%%%%%%%%%%%%%%%%%%%%%%%%%%%%%%%%%%%%%%%%%%%%%%%%%%%%%%%%%%%%%%%%%%%%%%%%%%%%%%%%%%%%%%
%%%%
Consider a compact marked Riemann surface $\Sigma^{(g)}$ of genus $g$, 
e.g., \cite{FK} 
with the canonical homology basis $\alpha_a$, $\beta_a$ for $1 \le a \le g$.  
%%
%%%%%%%%%%%%%%%%%%%%%%%%%%%%%%%%%%%%%%%%%%%%%%%%%%%%%%%%%%%%%%%%%%%%%%%%%%%%%%%%%%%%%
%%
For a review of the  
construction of a genus $g$ Riemann surface $\Sigma^{(g)}$ using the Schottky 
uniformization where we sew $g$ handles to the Riemann sphere 
$\Sigma^{(0)} \cong\Chat=\C\cup \{\infty\}$, 
see \cite{Fo, Be2, TW}.  
%%
%%%%%%%%%%%%%%%%%%%%%%%%%%%%%%%%%%%%%%%%%%%%%%%%%%%%%%%%%%%%%%%%%%%%%%%%%%%
%%%%
Every Riemann surface can be Schottky uniformized in a non-unique way. 
The main thing we are going to involve is the sewing relation for $1 \le a \le g$,  
 $w_{\pm a}$, $\rho_a \in \C$,   
%% 
%%%%%%%%%%%%%%%%%%%%%%%%%%%%%%%%%%%%%%%%%%%%%%%%%%%%%%%%%%%%%%%%%%%%%%%%%
%%
\begin{eqnarray}
\label{telefon}
(z'-w_{-a})(z-w_a)=\rho_a. 
\end{eqnarray}
%%
%%%%%%%%%%%%%%%%%%%%%%%%%%%%%%%%%%%%%%%%%%%%%%%%%%%%%%%%%%%%%%%%%%%%%%%% 
%%%%
We are using the explicit set of $n$-point functions 
and corresponding reduction formulas derived in \cite{TW}. 
These functions constitute  
the genus $0$, $1$, $g$ 
elements of $C^n(V, g)$-subspaces of the corresponding abstract families of 
complexes \eqref{buzovaish}. 
%%
%%%%%%%%%%%%%%%%%%%%%%%%%%%%%%%%%%%%%%%%%%%%%%%%%%%%%%5
%%%%
We recall here some general properties of genus zero 
$n$-point correlation functions including 
M\"obius transformation properties and the genus zero Zhu reduction \cite{Z1}.  
%%
%%%%%%%%%%%%%%%%%%%%%%%%%%%%%%%%%%%%%%%%%%%%%%%%%%%%%%%%%%%%%%%%%%%%%%%%
%%%%
 We review here also generalized Ward identities
 for the genus zero $n$-point functions associated 
with any quasiprimary vector of weight $N$. 
%% 
%%%%%%%%%%%%%%%%%%%%%%%%%%%%%%%%%%%%%%%%%%%%%%%%%%%%%%%%%%%%%%%%%%%%%%%%
%%%%
In what follows, the  
superscript $(g)$ in $Z_V^{(g)}({\bf v,y})$ refers to the genus $g$.
%%
%%%%%%%%%%%%%%%%%%%%%%%%%%%%%%%%%%%%%%%%%%%%%%%%%%%%%%%%%%%%%%%%%%
%%%%
Define the genus zero
 $n$-point correlation function for ${\bf v}=(v_1,\ldots, v_n)$ 
inserted at ${\bf y}=(y_1, \ldots, y_n)$, respectively,  
%%
%%%%%%%%%%%%%%%%%%%%%%%%%%%%%%%%%%%%%%%%%%%%%%%%%%%%%%%%%%%%%%%%%%%%%%%%%%%%%%%%%%%%%%%%
%%%%
$Z_V^{(0)}({\bf v,y})
=Z_V^{(0)}(\ldots;v_k, y_k;\ldots)=\langle \vac,{\bf Y(v,y)}\vac
\rangle_1$,  
for bilinear pairing   $\langle$ $.$, $.$ $\rangle_1$ and 
${\bf Y(v,y)}$ $=$ $Y(v_1,y_1)$  $\ldots$  $Y (v_n,y_n)$.   
%%   
%%%%%%%%%%%%%%%%%%%%%%%%%%%%%%%%%%%%%%%%%%%%%%%%%%%%%%%%%%%%%%%%%%%%%%%%%%%%
%%%%
The function $Z_V^{(0)}({\bf v,y})$ can be extended to
 a rational function 
in ${\bf y}$ in the domain $|y_1|>\ldots >|y_n|$.  
%%
%%%%%%%%%%%%%%%%%%%%%%%%%%%%%%%%%%%%%%%%%%%%%%%%%%%%%%%%%%%%%%%%%%%%%%%%%%
%%
We define the $n$-point correlation  
differential form for $v_k$ of weight $\wt(v_k)$,
with ${\bf dy^{\wt(v)}}=\prod_{k=1}^n dy_k^{\wt(v_k)}$, 
extend by linearity for non-homogeneous vectors
%%
%%%%%%%%%%%%%%%%%%%%%%%%%%%%%%%%%%%%%%%%%%%%%%%%%%%%%%%%%%%%%%%
%%%%
\begin{eqnarray}
\label{mendotor}
\mathcal F^{(0)}_V({\bf v,y})=Z_V^{(0)}({\bf v,y})\; {\bf dy^{\wt(v)}}.
\end{eqnarray}
%%
%%%%%%%%%%%%%%%%%%%%%%%%%%%%%%%%%%%%%%%%%%%%%%%%%%%%%%%%%%%%%%%%%%%%%%%%%%%%%%%%%%%
%%%%
\subsection{Modular property}
In \cite{TW} the following Lemma was proven. 
\begin{lemma} 
\label{meerertrtwttr}	
Let $v_k$ be quasiprimary of weight 
$\wt(v_k)$ for $k=1$, $\ldots$, $n$. Then for all 
$\gamma=\left(\begin{smallmatrix}a&b\\c&d\end{smallmatrix}\right)\in\SL_2(\C)$ we have
%%
%%%%%%%%%%%%%%%%%%%%%%%%%%%%%%%%%%%%%%%%%%%%%%%%%%%%%%%%%%%%%%%%%%%%%%%%%%%%%%
%%
${\mathcal F}_V^{(0)}({\bf v,y})= 
{\mathcal F}_V^{(0)}\left(v_1,\gamma.y_1; \ldots ; v_n, \gamma.y_n \right)$.    
\hfill $\qed$  
\end{lemma} 
%%
%%%%%%%%%%%%%%%%%%%%%%%%%%%%%%%%%%%%%%%%%%%%%%%%%%%%%%%%%%%%%%%%%%%%%%%%
%%
For quasiprimary $v_1$, $\ldots$, $v_n \in V$ we therefore find that 
${\mathcal F}_V^{(0)}({\bf v,y})$ is a genus zero 
meromorphic form in each $y_k$ of weight $\wt(y_k)$.  
%% 
%%%%%%%%%%%%%%%%%%%%%%%%%%%%%%%%%%%%%%%%%%%%%%%%%%%%%%%%%%%%%%%%%%%%%
%%%%
In general, ${\mathcal F}_V^{(0)}({\bf v,y})$ 
is not a meromorphic form and Lemma \ref{meerertrtwttr} 
generalizes as follows \cite{TW} 
%%
%%%%%%%%%%%%%%%%%%%%%%%%%%%%%%%%%%%%%%%%%%%%%%%%%%%%%%%%%%%%%%%%%%%%%%%%%%
%%%%
\begin{proposition}
\label{mordot}
Let $v_k$ be weight $\wt(v_k)$ for $1 \le k \le n$.  
 Then for all $\gamma=
\left(\begin{smallmatrix}a&b\\c&d\end{smallmatrix}\right)\in\SL_2(\C)$,  
%%
%%%%%%%%%%%%%%%%%%%%%%%%%%%%%%%%%%%%%%%%%%%%%%%%%%%%%%%%%%%%
%%
$Z_V^{(0)}({\bf v,y})= Z_V^{(0)}\left(
\ldots ;e^{ -c(cy_k+d)L(1)}	
\left(cy_k+d\right)^{-2\wt(v_k)}v_k,\gamma y_k;
\ldots \right)$.  
\hfill $\qed$ 
\end{proposition}
%%
%%%%%%%%%%%%%%%%%%%%%%%%%%%%%%%%%%%%%%%%%%%%%%%%%%%%%%%%%%%%%%%
%%%%
Note that Lemma \ref{meerertrtwttr} represents a particular example 
of invariance described in the Subsection \ref{puzan}  
and the Subsection \ref{groupppp} of the Appendix \ref{generererv} \cite{BZF}.   
%%
%%%%%%%%%%%%%%%%%%%%%%%%%%%%%%%%%%%%%%%%%%%%%%%%%%%%%%%%%%%%%%
%%
\subsection{The genus zero Ward identity}
Let $u$ be quasiprimary of weight $N$.   
%%
%%%%%%%%%%%%%%%%%%%%%%%%%%%%%%%%%%%%%%%%%%%%%%%%%%%%%%%%%%%%%%%%%%%%
%%
In \cite{TW} we find a general genus zero Ward identity which is  
 a genus zero analogue of\cite{Z1}, Proposition 4.3.1 in \cite{TW}.   
%%
%%%%%%%%%%%%%%%%%%%%%%%%%%%%%%%%%%%%%%%%%%%%%%%%%%%%%%%%%%%%%%%%
%% 
\begin{proposition}
\label{ponom}
Let $u$ be quasiprimary of weight $N$. 
%%
%%%%%%%%%%%%%%%%%%%%%%%%%%%%%%%%%%%%%%%%%%%%%%%%%%%%%%%%%%%%%%
%% 
Then for all $p\in \mathcal P_{2N-2}$ we have
\begin{eqnarray}
\label{wardy}
\sum\limits_{k=1}^n \sum\limits_{l=0}^{2N-2} 
\frac{1}{l!} \left(\partial^l_{y_k} p(y_k)\right) \;  
Z_V^{(0)}(\ldots;u(l).v_k, y_k;\ldots)= 0.
\hfill  \qed
\end{eqnarray} 
\end{proposition}
%%
%%%%%%%%%%%%%%%%%%%%%%%%%%%%%%%%%%%%%%%%%%%%%%%%%%%%%%%%%%%%%%%
%%
Proposition \ref{ponom}  
is the current algebra Ward identity for $N=1$ with $u\in V_{(1)}$  
and the conformal Ward identity for $N=2$ 
for the conformal vector $u=\omega\in V_{(2)}$.  
%%
%%%%%%%%%%%%%%%%%%%%%%%%%%%%%%%%%%%%%%%%%%%%%%%%%%%%%%%%%%%%%%%%%%
%%
\subsection{The genus zero Zhu recursion}
\label{perdas}
In \cite{TW} the genus zero Zhu recursion formulas were developed 
 for $Z_V^{(0)}(u, z; {\bf v,y})$ for quasiprimary 
$u$ of weight $N\ge 1$. 
%%
%%%%%%%%%%%%%%%%%%%%%%%%%%%%%%%%%%%%%%%%%%%%5 
%%
The Proposition \ref{ponom} implies that 
we can take as the Zhu reduction function 
%%
%%%%%%%%%%%%%%%%%%%%%%%%%%%%%%%%%%%%%%%%%%%%%%%%%%%%%%%%%%%%
%%
\begin{eqnarray}
\label{eschechl}
\pi^{(0)}_N(z, y)= 
\frac{1}{z-y}+\sum\limits_{l=0}^{2N-2} f_l(z)y^l , 
\end{eqnarray}
%% 
%%%%%%%%%%%%%%%%%%%%%%%%%%%%%%%%%%%%%%%%%%%%%%%%%%%%%%%%%%%%
%%
for any formal Laurent series $f_l (z)$ for $l=0$, $\ldots$, $2N-2$. 
%%
%%%%%%%%%%%%%%%%%%%%%%%%%%%%%%%%%%%%%%%%%%%%%%%%%%%%%%%%%%%%%%%%%%%%%%%%%%%%%
%%
In the case of quasiprimary genus zero Zhu recursion we have the following 
proposition \cite{TW}. 
%%
%%%%%%%%%%%%%%%%%%%%%%%%%%%%%%%%%%%%%%%%%%%%%%%%%%%%%%%%%%%%%%%%%%%%%%%%%%%%
%%
\begin{proposition} 
\label{raisi}
Let $u$ be weight $N\ge 1$ quasiprimary. For 
$f^{(0)}_{N, j}(z,y)$ $=$ $\frac{\partial^j_y}{j!}$ $\pi^{(0)}_N(z, y)$,  
%%
%%%%%%%%%%%%%%%%%%%%%%%%%%%%%%%%%%%%%%%%%%%%%%%%%%%%%%%%%%%%%%%
%%%%
\begin{eqnarray}
\label{gordon}
Z_V^{(0)}(u, z; {\bf v,y})= 
\sum\limits_{k=1}^n\sum\limits_{j\ge 0} 
f^{(0)}_{N, j}(z, y_k)\; Z_V^{(0)}(\ldots;u(j).v_k,y_k;\ldots).  
\hfill \qed 
\end{eqnarray}
\end{proposition}
%%
%%%%%%%%%%%%%%%%%%%%%%%%%%%%%%%%%%%%%%%%%%%%%%%%%%%%%%%%%%%%%%%%%%%%%%%%%%
%%
Proposition \ref{raisi} implies Lemma 2.2.1 of \cite{Z1}
 for a particular choice of 
Laurent series $f_l$ in \eqref{eschechl}. 
We note that $\pi^{(0)}_N(z, y)$ is independent of the vertex operator algebra  
$V$ and that the right hand side  of \eqref{gordon} 
is independent of the choice of $f_l$   
due to the Ward identity of Proposition \ref{ponom}. 
%%  
%%%%%%%%%%%%%%%%%%%%%%%%%%%%%%%%%%%%%%%%%%%%%%%%%%%%%%%%%%%%%%%%%%%%%%%%%%%
%%
Define $L^{(k)}(-1)=\frac{1}{k!}\left(L(-1)\right)^k$   
 so that $Y\left(L^{(k)}(-1) u, z\right)=\frac{1}{k!} \partial_z^k Y(u, z)$.    
%%
%%%%%%%%%%%%%%%%%%%%%%%%%%%%%%%%%%%%%%%%%%%%%%%%%%%%%%%%%%%%%%%%%%%%%%%%%%%5
%%
In the general genus zero Zhu recursion case we then have 
%%
%%%%%%%%%%%%%%%%%%%%%%%%%%%%%%%%%%%%%%%%%%%%%%%%%%%%%%%%%%%%%%%%%%%%%%%5
%%
\begin{corollary} 
\label{ezdap}
Let $L^{(i)}(-1)u$ be a quasiprimary descendant $u$ of $\wt(u)=N$. Then
for $f^{(0)}_{N, i, j}(z, y) =  
\frac{1}{i!} \frac{1}{j!} \partial^i_z \partial^j_y \pi^{(0)}_N (z,y)$, 
%%
%%%%%%%%%%%%%%%%%%%%%%%%%%%%%%%%%%%%%%%%%%%%%%%%%%%%%%%%%%%%%%%%%%%%%
%%
\begin{eqnarray}
\label{mudo1}
Z_V^{(0)} \left(L^{(i)}(-1)u,z;{\bf v,y}\right)
=\sum\limits_{k=1}^n \sum\limits_{j\ge 0} f^{(0)}_{N, i, j} 
(z,y_k) \; Z_V^{(0)} \left(\ldots;u(j).v_k;y_k;\ldots \right). 
\hfill \qed
\end{eqnarray} 
\end{corollary}
%%
%%%%%%%%%%%%%%%%%%%%%%%%%%%%%%%%%%%%%%%%%%%%%%%%%%%%%%%%%%%%%%%%%%%% 
%%%%
Consider formal differential forms \eqref{mendotor} with
  $\Pi^{(0)}_N(z,y)=\pi^{(0)}_N(z,y)\;dz^N\;dy^{1-N}$ 
and using \eqref{genroooor} then Proposition \ref{ponom} 
and  \ref{raisi} are  
equivalent \cite{TW} to  
%% 
%%%%%%%%%%%%%%%%%%%%%%%%%%%%%%%%%%%%%%%%%%%%%%%%%%%%%%%%%%%%%%%%%%%%%
%%
\begin{proposition} 
\label{hamen}
For weight $N\ge 1$  quasiprimary $u\in V$ and 
 all $p\in \mathcal P_{2N-2}$, we have 
%%
%%%%%%%%%%%%%%%%%%%%%%%%%%%%%%%%%%%%%%%%%%%%%%%%%%%%%%%%%%%%55
%%
\begin{eqnarray}
\label{septemb}
\sum\limits_{k=1}^n \sum\limits_{l=0}^{2N-2}
 \frac{1}{l!} \left( \partial^l_{y_k} p(y_k) \right) \; 
{\mathcal F}_V^{(0)}(\ldots;u(l).v_k,y_k;\ldots)\;dy_k^{l+1-N}=0, &&
\\
%%%%%%%%%%%%%%%%%%%%%%%%%%%%%%%%%%%%%%%%%%%%%%%%%%%%%%%%%%%%%%%%%%%
%%  
\label{korysto}
{\mathcal F}_V^{(0)}(u,z;{\bf v,y})=\sum\limits_{k=1}^n
\sum\limits_{j\ge 0}
\frac{1}{j!} \left( \partial^j_{y_k}  \Pi^{(0)}_N(z, y_k) \right) \; 
{\mathcal F}_V^{(0)} (\ldots;u(j).v_k,y_k;\ldots)\;dy_k^j.
 \hfill \qed &&
\end{eqnarray}
\end{proposition}
%%
%%%%%%%%%%%%%%%%%%%%%%%%%%%%%%%%%%%%%%%%%%%%%%%%%%%%%%%%%%%%%%%
%%%%
Note that the reduction function $f^{(0)}_{N, j}(z,y)$ 
does depend  
on $u$ in the form of $N$, and $j$ (since it is $j$-th derivative) 
but does not depend on ${\bf v}_n$ in this particular case.   
%%
%%%%%%%%%%%%%%%%%%%%%%%%%%%%%%%%%%%%%%%%%%%%%%%%%%%%%%%%%%%%%%%%%%
%%
In comparison to the general form $f^{(0)}_{l, j}(z,y)$ we keep $N$ here 
as the index characterizing $u$ above.  
%%
%%%%%%%%%%%%%%%%%%%%%%%%%%%%%%%%%%%%%%%%%%%%%%%%%%%%%%%%%%%%%%%%%%
%%
The formal residue is defined as follows 
%%
%%%%%%%%%%%%%%%%%%%%%%%%%%%%%%%%%%%%%%%%%%%%%%%%%%%%%%%%%%%%%%%%%%
%%%%
$\Res_{z-y_k}$ $\left(z-y_k\right)^j$ ${\mathcal F}_V^{(0)}$ $(u$, $z$; ${\bf v,y})$ 
$=$ ${\mathcal F}_V^{(0)}$ $(\ldots$; $u(j).v_k$, $y_k$; $\ldots)\;$ $dy_k^{j+1-N}$.   
%%
%%%%%%%%%%%%%%%%%%%%%%%%%%%%%%%%%%%%%%%%%%%%%%%%%%%%%%%%%%%%%%%
%%%%
\subsection{The genus zero chain conditions}
%%
%%%%%%%%%%%%%%%%%%%%%%%%%%%%%%%%%%%%%%%%%%%%%%%%%%%%%%%%%%%%
%%%%
For the quasiconformal genus zero case 
the general chain condition has the form \eqref{prodo} with 
$f^{(0)}_{ij}({\bf z}_n)=f^{(0)}_{ij}(z, {\bf y}_n)$. 
%%
%%%%%%%%%%%%%%%%%%%%%%%%%%%%%%%%%%%%%%%%%%%%%%%%%%%%%%%%%%%%%
%%%%
Since the form of the action of the coboundary operators 
$\delta^{(0)}(u, z_{n+1}; {\bf x}_n)$  
and $\delta^{(0)}(u', z'_n; {\bf x}'_{n-1})$ is fixed, then 
the accomplishment of the corresponding chain condition 
depends explicitly on $u$, $u' \in V$, 
and ${\bf x}_{n+1}$, ${\bf x}'_n$.  
%%
%%%%%%%%%%%%%%%%%%%%%%%%%%%%%%%%%%%%%%%%%%%%%%%%%%%%%%%%%%%%
%%%%
There are a few ways how to make \eqref{prodo} vanishing. 
%%
%%%%%%%%%%%%%%%%%%%%%%%%%%%%%%%%%%%%%%%%%%%%%%%%%%%%%%%%%%%%
%%%%
In the genus zero and $g$ examples provided in this paper, 
we exploit a one particular approach related to corresponding Ward 
identities. 
%%
%%%%%%%%%%%%%%%%%%%%%%%%%%%%%%%%%%%%%%%%%%%%%%%%%%%%%%%%%%%%%%%%
%%%% 
Namely, we pick $u$, $u'$ and ${\bf x}_{n+1}$, ${\bf x}'_n$ such that 
the combination $\delta^{(0)}(u, z_{n+1}; {\bf x}_n)$  
$\delta^n(u', z'_n; {\bf x}'_{n-1})\; Z_V^{(0)}({\bf v, y})$ vanish  
while $\delta^n(u', z'; {\bf x}'_{n-1}) Z_V^{(0)}({\bf v, y})$ is not necessary zero. 
%% 
%%%%%%%%%%%%%%%%%%%%%%%%%%%%%%%%%%%%%%%%%%%%%%%%%%%%%%%%%%%%%%%%%%%
%%%% 
Note that we may also explore the classical chain property mentioned in 
the Section \ref{mangaloid} with ${\bf x}_{n+1}={\bf x}_{n+1}$ which is also 
satisfied.  
%%
%%%%%%%%%%%%%%%%%%%%%%%%%%%%%%%%%%%%%%%%%%%%%%%%%%%%%%%%%%%%%%%
%%%%
Other options to make $\delta^{(0)}(u, z_n; {\bf x}_{n-1})$ 
 to satisfy to the chain condition
will be considered in a separate paper. 
%% 
%%%%%%%%%%%%%%%%%%%%%%%%%%%%%%%%%%%%%%%%%%%%%%%%%%%%%%%%%%%%%%%%%%
%%%%
With $u\in V$, $N>1$, $u(s+2N-1).{\bf 1}=0$, 
and $u \in V$, $N'>1$, $u'(s'+2N'-1){\bf 1}=0$, 
(note that $y_i$, $1 \le i \le n+1$ includes also $z_{n+1}$) 
we have the following 
%%
%%%%%%%%%%%%%%%%%%%%%%%%%%%%%%%%%%%%%%%%%%%%%%%%%%%%%%%%%%%%%%%%%%%%%%%%%%%%%%%%
%%
\begin{proposition}
Let $L^{(i)}(-1)u$, be a quasiprimary descendant of $u\in V$ of $\wt(u)=N$.   
With the conditions \eqref{pendozo1} on corresponding 
domains of $y_k$, $1 \le k \le n+1$ and $y_{k'}$, $1 \le k' \le n$, 
  the coboundary operator $\delta^{(0)}({\bf x}_n)$ 
\eqref{mudo1} resulting from Corollary \ref{ezdap}   
 satisfies the chain property \eqref{prodo}.   
\end{proposition}
%%
%%%%%%%%%%%%%%%%%%%%%%%%%%%%%%%%%%%%%%%%%%%%%%%%%%%%%%%%%%%%%%%%%
%%%%
\begin{proof}
The action of $T_i(u(j)) \; T_i(u'(j'))$ on $Z_V^{(0)}({\bf v, y})$ 
is equivariant. 
%%
%%%%%%%%%%%%%%%%%%%%%%%%%%%%%%%%%%%%%%%%%%%%%%%%%%%%%%%%%%%%%%%%%%%%%%%
%%%%
That means that 
$T_i(u(j)) \; T_{i'}(u(j')).Z_V^{(0)}({\bf v, y})= 
Z_V^{(0)}(T_i(u(j)) \; T_i(u'(j')).{\bf v, y})=A^{ii'}_{jj'}$. 
%%
%%%%%%%%%%%%%%%%%%%%%%%%%%%%%%%%%%%%%%%%%%%%%%%%%%%%%%%%%%%%%%%%%%%%%%%%%
%%%%
The single action of $\delta^{(0)}(u', z'_n; {\bf x}'_{n-1})$, i.e.,  
 $\sum_{i=1}^n \sum_{j\ge 0}  
f^{(0)}_{N, i', j'}(z, y'_{k'})     
 \; T_{i'}(u'(j')).\;Z_V^{(0)}({\bf v, y})$   
is assumed to be not necessary zero.  
%%
%%%%%%%%%%%%%%%%%%%%%%%%%%%%%%%%%%%%%%%%%%%%%%%%%%%%%%%%%%%%%%%%%%%%%5
%%%%
The quadruple summation of 
$\sum_{i=1}^{n+1}$ $\sum_{j\ge 0}$   
$\sum_{i'=1}^n$  $\sum_{j' \ge 0}$    
$f^{(0)}_{N, i, j}$ $(z, y_k)$     
$f^{(0)}_{N', i', j'}$ $(z', y'_{k'})$ $A^{ii'}_{jj'}$,  
should give polynomial coefficients  
where $A^{ii'}_{jj'}$ collects result of the action of $T$-operators. 
%%
%%%%%%%%%%%%%%%%%%%%%%%%%%%%%%%%%%%%%%%%%%%%%%%%%%%%%%%%%%%%%%%%%%%%
%%%%
We consider the subspaces of $C^{n+1}$(V, 0), $C^n(V, 0)$ formed by     
functions $Z_V^{(0)}(u, z; {\bf v, z})$ and $Z_V^{(0)}$ $(u'$, $z'$; $u$, $z$;  
${\bf v, z})$ 
correspondingly, 
for some $u$, $u'\in V$ of weights $N$, $N' >1$, 
and corresponding suitable values of paramteres   
$z$, $z'$, $y_k$, $y'_{k'}$. 
%%
%%%%%%%%%%%%%%%%%%%%%%%%%%%%%%%%%%%%%%%%%%%%%%%%%%%%%%%%%%%%% 
%%%%
Let us perform summation over $N_m \ge 1$, $m \ge 1$, 
for variating $u$, 
such that for a quasiprimary descendant $L^{(i_m)}(-1)u$ of wait $N_m$   
with $i_m$, $N_m$ and $z_m$ depending on $m$ such that     
$N_m$ is chosen in that way that the only finite number of values of $s$ 
appear in the final sum. 
%%
%%%%%%%%%%%%%%%%%%%%%%%%%%%%%%%%%%%%%%%%%%%%%%%%%%%%%%%%%%%%%%%%%%%%%%%%%%%%
%%%%
Then, \eqref{prodo} results in the form of \eqref{septemb} in $y_k$.  
%%
%%%%%%%%%%%%%%%%%%%%%%%%%%%%%%%%%%%%%%%%%%%%%%%%%%%%%%%%%%%%%%%%%%%%%%%%%% 
%%
The differentiation of the second summand vanishes after $j=2N_m-2$. 
%%
%%%%%%%%%%%%%%%%%%%%%%%%%%%%%%%%%%%%%%%%%%%%%%%%%%%%%%%%%%%%%%%%%%%%%%%%%%%%%%%%
%%
To simplify the representation, let us denote by $\sum= 
\sum_{N_{m \ge 1}} \sum_{N'_{m' \ge 1}}  
\sum_{k=1}^{n+1} \sum_{j \ge 0} \sum{j'\ge 0}   
\sum_{k'=1}^n$. 
%%
%%%%%%%%%%%%%%%%%%%%%%%%%%%%%%%%%%%%%%%%%%%%%%%%%%%%%%%%%%%%%%%%%%%%%%%%%%%%
%%
For the chain property, 
 we have for $z_m\ne 0$, $|\frac{y_k}{z_m}|<1$,  
%% 
%%%%%%%%%%%%%%%%%%%%%%%%%%%%%%%%%%%%%%%%%%%%%%%%%%%%%%%%%%%%%%%%%%%%%%%%%%%%%%%5
%%
\begin{eqnarray*}
\sum
\frac{\partial^{i_m}_{z_m} \partial^j_{y_k}}{i_mj!}    
  \left( \frac{1}{z_m-y_k}+\sum\limits_{l_m=0}^{2N_m-2}  
 \frac{\partial^{i_m}_{z_m}}{i_m!}  f_{l_m}(z_m)(y_k)^{l_m} \right)   %&&  
%%
%%%%%%%%%%%%%%%%%%%%%%%%%%%%%%%%%%%%%%%%%%%%%%%%%%%%%%%%%%%%%%%%%%%%%%
%% 
\frac{\partial^{j'}_{z'_{m'}} \partial^{j'}_{y'_{k'}} }{j'!j'!}   
 \left(\pi^{(0)}_{N'_{m'}}(z'_{m'}, y'_{k'})\right)  
A^{i_m i_m'}_{jj'}  &&
\end{eqnarray*}
%%
%%%%%%%%%%%%%%%%%%%%%%%%%%%%%%%%%%%%%%%%%%%%%%%%%%%%%%%%%%%%%%%%%%%%%
%\nn
%%
\begin{eqnarray*} 
=\sum
\left(\frac{(-1)^j }{i_m!j!(z_m-y_k)^{i_m+j+1}}  
 + \sum\limits_{l_m=0}^{2N_m-2}  
 \frac{\partial^{i_m}_{z_m}}{i_m!j!}  f_{l_m}(z_m)  l_m(l_m-1)\ldots(l_m-j+1)  
  y_k^{l_m-j} \right)  &&
\end{eqnarray*}
%%
%%%%%%%%%%%%%%%%%%%%%%%%%%%%%%%%%%%%%%%%%%%%%%%%%%%%%%%
%%
\begin{eqnarray*}
\frac{\partial^{i'_{m'}}_{z'_{m'}} \partial^{j'}_{y'_{k'}}}{i'_{m'}!j'!} 
\left(\pi^{(0)}_{N'_{m'}}(z'_{m'}, y'_{k'}) \right)  
A^{i_m i'_{m'}}_{jj'}  \qquad &&
\end{eqnarray*}
%%
%%%%%%%%%%%%%%%%%%%%%%%%%%%%%%%%%%%%%%%%%%%%%%%%%%%%%%%%%%%%%%%%%%%%%%%%%%%%%
%%
\begin{eqnarray*} 
=\sum
\left( \frac{(-1)^j}{i_m!j! z_m^{i_m+j+1}}\left(1 + \sum\limits_{s\ge 1}  
\frac{(i_m+j+1)(i_m+j+2)\ldots(i_m+j+s)}{s!}  
\left(\frac{y_k}{z_m} \right)^s \right) \right. &&
\end{eqnarray*}
%%
%%%%%%%%%%%%%%%%%%%%%%%%%%%%%%%%%%%%%%%%%%%%%%%%%%%%%%%%%%%%%%%%%%% 
%%
\begin{eqnarray*}
\left. 
+ \sum\limits_{l_m=0}^{2N_m-2}   
\frac{\partial^{i_m}_{z_m}}{i_m!j!} 
  f_{l_m}(z_m) \; l_m(l_m-1)\ldots (l_m-j+1) 
 \;  y_k^{l_m-j} \right)  
\frac{ \partial^{i'_{m'}}_{z'_{m'}}  \partial^{j'}_{y'_{k'}}}{i'_{m'}!j'!}  
\left( \pi^{(0)}_{N'_{m'}}(z'_{m'}, y'_{k'}) \right) 
A^{i_mi'_{m'}}_{jj'} 
 && 
\end{eqnarray*}
%% 
%%%%%%%%%%%%%%%%%%%%%%%%%%%%%%%%%%%%%%%%%%%%%%%%%%%%%%%%%%%%%%%%%%%%%%%%%%%%%
%%
\begin{eqnarray*} 
=\sum
\left( \frac{(-1)^j}{ i_m!j!  z_m^{i_m+j+1} }  
%%
%%%%%%%%%%%%%%%%%%%%%%%%%%%%%%%%%%%%%%%%%%%%%%%%%%%%%%%%%%%%%%%%%%%%%%%%
%%
\left(1 + \sum\limits_{s\ge 1}  
\frac{(i_m+j+1)(i_m+j+2)\ldots(i_m+j+s)}{s!}  
\left(\frac{y_k}{z_m} \right)^s \right) \right. && 
\end{eqnarray*}
%%
%%%%%%%%%%%%%%%%%%%%%%%%%%%%%%%%%%%%%%%%%%%%%%%%%%%%%%%%%%%%%%%%%%% 
%%
\begin{eqnarray*}
\left. 
+ \sum\limits_{l_m=0}^{2N_m-2}   
\frac{\partial^{i_m}_{z_m} }{i_m!j!}   
  f_{l_m}(z_m) \; l_m(l_m-1)\ldots (l_m-j+1) 
 \;  y_k^{l_m-j} \right)  &&
\end{eqnarray*}
%%
%%%%%%%%%%%%%%%%%%%%%%%%%%%%%%%%%%%%%%%%%%%%%%%%%%%%%%%%%%%%%%%%%%%%%%%%%
%%
\begin{eqnarray*}
=\sum\left( \frac{(-1)^{j'} }{ {i'_{m'}}! {j'}! z'_{m'}{}^{i'_{m'}+j'+1}}
\left(\sum\limits_{s'\ge 0}   
\frac{(i'_{m'}+j'+1)(i'_{m'}+j'+2)\ldots(i'_{m'}+j'+s')}{s'!}  
\left(\frac{y'_{k'}}{z'_{m'}} \right)^{s'} \right) \right. &&
\end{eqnarray*}
%%
%%%%%%%%%%%%%%%%%%%%%%%%%%%%%%%%%%%%%%%%%%%%%%%%%%
%%
\begin{eqnarray*}
\left. 
+ \sum\limits_{l'_{m'}=0}^{2N'_{m'}-2}    
\frac{1}{i'_{m'!}} \frac{1}{j'!} \partial^{i'_{m'}}_{z'_{m'}} 
  f_{l'_{m'}}(z'_{m'}) \; l'_{m'}(l'_{m'}-1)\ldots (l'_{m'}-j'+1)  
 \;  y_{k'}{'}^{l'_{m'}-j'} \right) 
A^{i_mi'_{m'}}_{jj'}. \qquad \qquad \qquad  &&
\end{eqnarray*}
%%
%%%%%%%%%%%%%%%%%%%%%%%%%%%%%%%%%%%%%%%%%%%%%%%%%%%%%%%%%%%%%%%%%%%%
%%%%
Since we differentiate fixed positive powers $l_m$ of $y_k$ in the expressions
above, no negative powers of $y_k$ appear. 
%%
%%%%%%%%%%%%%%%%%%%%%%%%%%%%%%%%%%%%%%%%%%%%%%%%%%%%%%%%%%%%%%%%%%%%%%%%%
%%%%
Let us introduce the notations  
(and corresponding $\alpha'$, $\beta'$, $\xi'$ with primed arguments) 
%%
%%%%%%%%%%%%%%%%%%%%%%%%%%%%%%%%%%%%%%%%%%%%%%%%%%%%%%%%%%%%%%%%%%%%
%%%%
\begin{eqnarray}
   && \alpha=\alpha(i_m, j, m)= \frac{(-1)^j}{ i_m!j!  z_m^{i_m+j+1} },   
%%
%%%%%%%%%%%%%
\\
\label{merda}
  && \beta=\beta\left(i_m, j, s \right)=   
\frac{(i_m+j+1)(i_m+j+2)\ldots(i_m+j+s)}{s!},   
%%
%%%%%
\nn
\nonumber
  && \xi=\xi\left(i_m, j, m \right)= 
\sum\limits_{l_m=0}^{2N_m-2}   
\frac{1}{i_m!} \frac{1}{j!} \partial^{i_m}_{z_m} 
  f_{l_m}(z_m) \; l_m(l_m-1)\ldots (l_m-j+1).  
\end{eqnarray}
%% 
%%%%%%%%%%%%%%%%%%%%%%%%%%%%%%%%%%%%%%%%%%%%%%%%%%%%%%%%%%%%%%%%%%%%%%%
%%%%
At this point we have several options to make \eqref{prodo} vanishing. 
The first option is to make all coefficients near powers of $y_k$ 
vanish to make \eqref{prodo} equal to zero. 
%%
%%%%%%%%%%%%%%%%%%%%%%%%%%%%%%%%%%%%%%%%%%%%%%%%%%%%%%%%%%%%%%%%%%%%%%%%
%%%%
The other option is to cut off the infinite series in $y_k$ in the last 
formulas so that Ward identity \eqref{septemb} would be applicable. 
%% 
%%%%%%%%%%%%%%%%%%%%%%%%%%%%%%%%%%%%%%%%%%%%%%%%%%%%%%%%%%%%%%%%%%%%%%
%%%%
Let us start with the first option.
%% 
%%%%%%%%%%%%%%%%%%%%%%%%%%%%%%%%%%%%%%%%%%%%%%%%%%%%%%%%%%%%%%%%%%%%%%%% 
%%%%
Combining all possible zero power terms we obtain 
 $\alpha \alpha' + \alpha \xi'|_{l'_{m'}-j'=0} + \alpha \beta \xi'|_{s+l'_{m'}-j'=0} 
 + \xi\alpha'|_{l_m-j=0} + \xi\alpha'\beta'|_{l_m-j+s'} + \xi\xi'|_{l_m-j+l'_{m'}-j'}=0$ 
%%
%%%%%%%%%%%%%%%%%%%%%%%%%%%%%%%%%%%%%%%%%%%%%%%%%%%%%%%%%%%%%%%%%%%%
%%%%
Then it follows that $s=0$ and $s'=0$ which contradicts the ranges of $s$ and
 $s'$.  
%%
%%%%%%%%%%%%%%%%%%%%%%%%%%%%%%%%%%%%%%%%%%%%%%%%%%%%%%%%%%%%%%%%%%%%%%%%%%%%%%%%%%
%%%%
The other possibility is to kill infinite powers of $y_k$
and make the expressions  
for $\delta^{(0)}({\bf x}_{n+2})$ and $\delta^{(0)}({\bf x}_{n+1})$ polynomial. 
%%
%%%%%%%%%%%%%%%%%%%%%%%%%%%%%%%%%%%%%%%%%%%%%%%%%%%%%%%%%%%%%%%%%%%%%%%%%%%%%%%%%
%%%%
Note that in the last expression only the second sum depends on $N_m$.  
Thus, when we sum over $N_m$, we can 
 eliminate all powers of $y_k$  
except for a finite number.   
%%
%%%%%%%%%%%%%%%%%%%%%%%%%%%%%%%%%%%%%%%%%%%%%%%%%%%%%%%%%%%%%%%%%%%%%%%%%
%%%%
For that purpose we consider 
an infinite sum of action of $\frac{1}{i_m!} (L(-1))^{i_m} u$,  
$\wt (u_m)=N_m=m$, $m \ge 1$. 
%% 
%%%%%%%%%%%%%%%%%%%%%%%%%%%%%%%%%%%%%%%%%%%%%%%%%%%%%%%%%%%%%%%%%%%%%%%%
%%%%
 Note that for each fixed $m$ we have $m$ identical terms corresponding to 
each $s$ in the sum in $\beta$-term corresponding to the power $y_k^s$. 
At the same time there are $2N_m-2+1=m-1$ different terms of the 
double sum of $\xi$-part corresponding to the powers $y_k^{l_m-j}$. 
%%
%%%%%%%%%%%%%%%%%%%%%%%%%%%%%%%%%%%%%%%%%%%%%%%%%%%%%%%%%%%%%%%%%%%%%%%%%%%%%%%%
%%%%
Thus, in our setup, the $\xi$-part powers of $y_k$ kill all $\beta$-part powers
when we identify $m \beta y^s = \xi y^{l_m-j}$. 
We can choose $s_0\ge 0$, then 
with $s=l_m-j$, we have 
%%
%%%%%%%%%%%%%%%%%%%%%%%%%%%%%%%%%%%%%%%%%%%%%%%%%%%%%%%%%%%%%%%%%%%%%%%%%
%%%%
\begin{eqnarray}
\label{pendozo1}
  \partial^{i_m}_{z_m} f_{l_m}(z_m) = 
-\frac{(-1)^j\; m \;i_m!\; (i_m+j+1)(i_m+j+2) \ldots (i_m+l_m)} 
{(l_m-j)! \; l_m(l_m-1) \ldots (l_m-j+1) \; z_m^{l_m+i_m+s+1}}. 
\end{eqnarray}
%%
%%%%%%%%%%%%%%%%%%%%%%%%%%%%%%%%%%%%%%%%%%%%%%%%%%%%%%%%%%%%%%%%%%%%%
%%%%
Moreover, since in the $\xi$-term we have derivatives 
$\partial_z^{i_m} f_l(z_m)$
of the Laurent series we choose $f_l(z_m)$ such that 
derivatives multiplied by powers $y_k^{l_m-j}$ vanish for a given 
$l_m-j>n_m$. 
%%%%
Thus, the total series in $y_k$ remains polynomial. 
%%
%%%%%%%%%%%%%%%%%%%%%%%%%%%%%%%%%%%%%%%%%%%%%%%%%%%%%%%%%%%%%%%%%%%%%%%%%%
%%
Note that due to our construction, the coefficients of the resulting 
Laurent series $\partial^{i_m}_{z_m} f_l(z_m)$ are convergent in $z_m$. 
%%
%%%%%%%%%%%%%%%%%%%%%%%%%%%%%%%%%%%%%%%%%%%%%%%%%%%%%%%%%%%%%%%%%%%%%%%%%%
%%
The same procedure should be also done for $\delta^{(0)}(u', z'_n; {\bf x}'_{n-1})$.  
%%

%%%%%%%%%%%%%%%%%%%%%%%%%%%%%%%%%%%%%%%%%%%%%%%%%%%%%%%%%%%%%%%%%%%%%%%%%%%
%%%%
Similarly,  
for the opposite case, i.e., for $y_k \ne 0$,  
 $|\frac{z_m}{y_k}|<1$,  
property we have  
%% 
%%%%%%%%%%%%%%%%%%%%%%%%%%%%%%%%%%%%%%%%%%%%%%%%%%%%%%%%%%%%%%%%%%%%%%%%%%%%%
%%
\begin{eqnarray*}
\sum
\left(\frac{(-1)^j}{i_m!j!y_k^{i_m+j+1}} \left( 1 + \sum\limits_{s\ge 1}  
\frac{(i_m+j+1)(i+j+2) \ldots (i_m+j+s)}{s!}  
\left(\frac{z_m}{y_k} \right)^s \right) \right.    
\nn
\left. 
+ \sum\limits_{l_m=0}^{2N_m-2}    
 \frac{\partial^{i_m}_{z_m} f_{l_m}(z_m)}{i_m!j!}   
l_m(l_m-1)\ldots(l_m-j+1)  y_k^{l_m-j} \right)  
%%
%\nn
%%
\frac{\partial^{i'_{m'}}_{z'_{m'}}\partial^{j'}_{y'_{k'}} }{i'_{m'}!j'!}    
\left( \pi^{(0)}_{N'_{m'}}\left(z'_{m'}, y'_{k'}\right) \right)  
A^{i'_{m'} i_m}_{j'j}. &&
\end{eqnarray*}
%%
%%%%%%%%%%%%%%%%%%%%%%%%%%%%%%%%%%%%%%%%%%%%%%%%%%%%%%%%%%%%%%%%%%%
%%%%
After summation over $N_m$, in order to compensate the 
 $y_k^{-i_m-j-1}$-term, we may 
 identify $z_m=y_k^{\frac{i_m+j+1}{s}+\alpha_m}$, $s\ne 0$, 
(note that actually $s=s_m$). 
%%
%%%%%%%%%%%%%%%%%%%%%%%%%%%%%%%%%%%%%%%%%%%%%%%%%%%%%%%%%%%%%%%%%%
%%%% 
Since it is assumed that $|\frac{z_m}{y_k}|<1$, 
we should set $0 < \alpha_m <1$. 
%%
%%%%%%%%%%%%%%%%%%%%%%%%%%%%%%%%%%%%%%%%%%%%%%%%%%%%%%%%%%%%%%%%%%%%%5  
%%%%
But $l_m-j \ge 0$, we see that with the identification 
 $\alpha_m=l_m-j$,  
 the conditions above are not possible.  
Thus, we obtain \eqref{wardy} and the chain condition \eqref{prodo} is fulfilled.  
\end{proof}
%%
%%%%%%%%%%%%%%%%%%%%%%%%%%%%%%%%%%%%%%%%%%%%%%%%%%%%%%%%%%%%%
%%%%
The $n$ genus zero cohomology $H_N^n(V, 0)$ 
involving quasiprimary vertex operator algebra 
states of weight $N$ 
are found according to Corollary \ref{ezdap}  
with the reduction operators given by $\pi^{(0)}(z, y)$.  
%%
%%%%%%%%%%%%%%%%%%%%%%%%%%%%%%%%%%%%%%%%%%%%%%%%%%%%%%%%%%%%%%%%%%%%%%%%%%%%%%%%%%%%%%%%%%% 
%%%%%%%%%%%%%%%%%%%%%%%%%%%%%%%%%%%%%%%%%%%%%%%%%%%%%%%%%%%%%%%%%%%%%%%%%%%%%%%%%%%%%%%%%%%
%%%%
\section{Example: genus $g$ multiple cohomology on Riemann surfaces}
\label{apala}
%% 
%%%%%%%%%%%%%%%%%%%%%%%%%%%%%%%%%%%%%%%%%%%%%%%%%%%%%%%%
%%%%
In this Section we will use the genus  
$g$ partition functions in the Schottky sewing 
scheme with $f_l(z)$ given by \eqref{pwrepefk}.
%% 
%%%%%%%%%%%%%%%%%%%%%%%%%%%%%%%%%%%%%%%%%%%%%%%%%%%%%%%%
%%%%
In \cite{TW} the formal genus $g$  
 partition and $n$-point correlation functions were introduced  
for a simple vertex operator algebra $V$ 
of strong conformal field theory type with $V$ 
 isomorphic to the contragredient module $V'$. 
%% 
%%%%%%%%%%%%%%%%%%%%%%%%%%%%%%%%%%%%%%%%%%%%%%%%%%%%%%%%%%%%%%%%%%%%%%%%%%%%
%%%%
These functions are formally associated 
to a genus $g$ Riemann surface $\Sigma^{(g)}$ in the Schottky  
scheme of the Section \ref{vtorog}. 
We apply the approach which is a generalization of the genus two 
sewing schemes of\cite{MT2} and genus two Zhu recursion of \cite{GT}.
%% 
%%%%%%%%%%%%%%%%%%%%%%%%%%%%%%%%%%%%%%%%%%%%%%%%%%%%%%%%%%%%%%%%%%%%%%%
%%%%
 The genus $g$ Zhu reduction formulas are described  
with universal coefficients given by derivatives of 
$\Psi^{(g)}_N(z,y)$ and holomorphic $N$-forms $\Theta_{N, a}^l (z)$. 
%%
%%%%%%%%%%%%%%%%%%%%%%%%%%%%%%%%%%%%%%%%%%%%%%%%%%%%%%%%%%%%%%%%%%%
%%
These are generalizations of the genus zero formulas  
of the Section \ref{generererv}  
and of that at genus one with elliptic 
Weierstrass function coefficients\cite{Z1}. 
%%
%%%%%%%%%%%%%%%%%%%%%%%%%%%%%%%%%%%%%%%%%%%%%%%%%%%%%%%%%%%%%%%%%%%%%%%%%%%%%%%%%%%%%%%%%%%%%
%%
\subsection{Genus $g$ $n$-point functions}
\label{publiv}
For each $1 \le a \le g$, let $\{b_a\}$  denote a homogeneous  
$V$-basis, and let $\{\bbar_a\}$ be 
the dual basis with respect to the bilinear 
pairing $\langle .,.\rangle_1$, i.e.,  
 with $\rho=1$.
%%
%%%%%%%%%%%%%%%%%%%%%%%%%%%%%%%%%%%%%%%%%%%%%%%%%%%%%%%%%%%%%%%%%%%%
%%
  Define $b_{-a}$ for $1 \le a\le g$ by \eqref{gensher} 
with formal $\rho_a$ which is identified with a Schottky sewing parameter. 
%%
%%%%%%%%%%%%%%%%%%%%%%%%%%%%%%%%%%%%%%%%%%%%%%%%%%%%%%%%%%%%%%%%%%%%%%%%%%%%%%% 
%%
Then $\{b_{-a}\}$ is a dual basis for the bilinear pairing 
 $\langle ., . \rangle_{\rho_a}$ with adjoint 
given by \eqref{tuta}  
for $u$ quasiprimary of weight $N$.
%%
%%%%%%%%%%%%%%%%%%%%%%%%%%%%%%%%%%%%%%%%%%%%%%%%%%%%%%%%%%%%%%
%%
Let ${\bf b}_+=b_1\otimes\ldots \otimes b_g$
 denote an element of a $V^{\otimes g}$-basis. Let 
$w_a$ for $a\in\{-1, \ldots, - g, 1, \ldots, g\}$ be $2g$ formal variables. 
Consider the genus zero $2g$-point rational function 
%%
%%%%%%%%%%%%%%%%%%%%%%%%%%%%%%%%%%%%%%%%%%%%%%%%%%%%%%%%%
%%
\begin{eqnarray*}
Z_V^{(0)}({\bf b,w})&=& Z_V^{(0)}(b_1,w_1;b_{-1},w_{-1};\ldots;b_g,w_g;b_{-g},w_{-g})
\nn 
&=& Z_V^{(0)}(b_1,w_1;\bbar_1,w_{-1};\ldots;b_g,w_g;\bbar_g,w_{-g})
\prod_{a=1}^g \rho_a^{\wt(b_a)}, 
\end{eqnarray*}
%%
%%%%%%%%%%%%%%%%%%%%%%%%%%%%%%%%%%%%%%%%%%%%%%%%%%%%%%%%%%%%%%%%%%%%%%%
%%
for  
$({\bf b,w})$ $=$ $(b_1$, $w_1$, $b_{-1}$, $w_{-1}$,
 $\ldots$, $b_g$, $w_g$, $b_{-g}$, $w_{-g})$.   
%%
%%%%%%%%%%%%%%%%%%%%%%%%%%%%%%%%%%%%%%%%%%%%%%%%%%%%%%%%%%%%%%%%%%%%%%%%%%
%%
Define the genus $g$ partition function for $({\bf w,\rho})$ $=$ $(w_1$, $w_{-1}$,
 $\rho_1$, $\ldots$, $w_g$, $w_{-g}$, $\rho_g)$ by 
\begin{eqnarray}
\label{kutuyko}
Z^{(g)}_V=Z^{(g)}_V({\bf w,\rho})
=\sum\limits_{{\bf b}_+}Z_V^{(0)}({\bf b,w}),
\end{eqnarray}
%%
%%%%%%%%%%%%%%%%%%%%%%%%%%%%%%%%%%%%%%%%%%%%%%%%%%%%%%%%%%%%%%%%%%%% 
%%
where the sum is over any basis $\{{\bf b}_{+}\}$ of $V^{\otimes g}$.
%%
%%%%%%%%%%%%%%%%%%%%%%%%%%%%%%%%%%%%%%%%%%%%%%%%%%%%%%%%%%%%%%%%%%%
%%%% 
This definition was motivated by the sewing relation 
\eqref{telefon} and ideas of \cite{TW1, MT2}.   
%%
%%%%%%%%%%%%%%%%%%%%%%%%%%%%%%%%%%%%%%%%%%%%%%%%%%%%%%%%%%%%%%%%
%%
$Z^{(g)}_V$ depends on $\rho_a$ via the dual vectors 
 ${\bf b}_{-}=b_{-1}\otimes\ldots \otimes b_{-g}$
 as in \eqref{gensher}. 
%%
%%%%%%%%%%%%%%%%%%%%%%%%%%%%%%%%%%%%%%%%%%%%%%%%%%%%%%%
%%
In particular, setting $\rho_a=0$ 
for some $1 \le a \le g$, $Z^{(g)}_V$ then 
degenerates to a genus $g-1$ partition function. 
%%
%%%%%%%%%%%%%%%%%%%%%%%%%%%%%%%%%%%%%%%%%%%%%%%%%%%%%%%%%%%%%%%%%%%%%%%%
%%
We define the genus $g$ formal $n$-point function  
for $n$ vectors $v_1$, $\ldots$, $v_n\in V$    
inserted at $(y_1$, $\ldots$, $y_n)$ 
for rational genus zero $(n+2g)$-point functions
$Z_V^{(0)}$ $({\bf v,y}$; ${\bf b,w})$ 
$=$ $Z_V^{(0)}$ $(v_1$, $y_1$; $\ldots$ ; $v_n$, $y_n$; $b_{-1}$, $w_{-1}$; 
$\ldots$; $b_g$, $w_g)$ by  
%%
%%%%%%%%%%%%%%%%%%%%%%%%%%%%%%%%%%%%%%%%%%%%%%%%%%%%
%%
\begin{eqnarray}
\label{elco}
Z_V^{(g)}({\bf v,y})=Z^{(g)}_V({\bf v,y};{\bf w,\rho})
=
\sum\limits_{{\bf b}_+}Z_V^{(0)}({\bf v,y};{\bf b,w}).  
\end{eqnarray}
%%
%%%%%%%%%%%%%%%%%%%%%%%%%%%%%%%%%%%%%%%%%%%%%%%%%%%%%%%%%%%%%%%%%%%%%%%%%%%%%%%%
%%
 The corresponding genus $g$ formal 
$n$-point correlation differential form is introduced by 
%%
%%%%%%%%%%%%%%%%%%%%%%%%%%%%%%%%%%%%%%%%%%%%%%%%%%%%%%%%%%%%%%%%%%%%5
%%
$\F_V^{(g)}({\bf v,y}_n)=Z^{(g)}_V({\bf v,y}_n)\;{\bf dy^{\wt(v)}}$.  
%%
%%%%%%%%%%%%%%%%%%%%%%%%%%%%%%%%%%%%%%%%%%%%%%%%%%%%%%%%%%%%%%%%%%%%%%%%%%%%%%%%%%
%%%%
With coboundary operator given by $B({\bf x}_n)$  
 one defines the $n$-th Schottky cohomology $H^n(V, g)$ 
  of the bicomplex $\left(C^n(V, g), \delta^{(g)}({\bf x}_n) \right)$   
with the spaces $C^n(V, 0)$ to be  
$H^n(V, g)=\mbox{\rm Ker} \; \widehat \delta^{(g)}({\bf x}_{n+1})/ 
\mbox{\rm Im} \; \widehat \delta^{(g)}({\bf x}_n)$. 
%%
%%%%%%%%%%%%%%%%%%%%%%%%%%%%%%%%%%%%%%%%%%%%%%%%%%%%%%%%%%%%%%%%%%%%%%%%%%%%%%%%
%%%%
Using modifications of coboundary operators of this paper 
we are able to construct spectral sequences for  
vertex operator algebra complexes which
 can be used in various cohomology construction, in 
particular, on orbifolds.   
%%
%%%%%%%%%%%%%%%%%%%%%%%%%%%%%%%%%%%%%%%%%%%%%%%%%%%%%%%%%%%%%%%%%%%%%
%%
\subsection{Genus $g$ formal M\"obius invariance} 
For Schottky parameters $W_a$ and 
$p\in \mathcal P_2$, define the M\"obius generator
%%
%%%%%%%%%%%%%%%%%%%%%%%%%%%%%%%%%%%%%%%%%%%%%%%%%%%%%%%%%%%%%%%%%%%%%%%%%%%%%
%%
$\D^{p}=\sum_{a\in \{-1, \ldots, - g, 1, \ldots, g\}} p(W_a)\partial_{W_a}$. 
%%
%%%%%%%%%%%%%%%%%%%%%%%%%%%%%%%%%%%%%%%%%%%%%%%%%%%%%%%%%%%%%%%%%%%%%%%%%%%%%%%
%%
This can be written in terms of the 
$w_a$, $\rho_a$ parameters 
for $\frac{1}{i!}\partial^i_z p(z)$ as 
%%
%%%%%%%%%%%%%%%%%%%%%%%%%%%%%%%%%%%%%%%%%%%%%%%%%%%%%%%%%%%%%%%%%%%%%%%%%%%
%%
\begin{eqnarray}
\label{mardot}
\D^p=\sum\limits_{a\in \{-1, \ldots, - g, 1, \ldots, g\}} \left(  
p(w_a)\partial_{w_a}+ \partial_{w_a} p(w_a)\rho_a \partial_{\rho_a} 
+ \frac{1}{2}\partial^2_{w_a} p(w_a) \rho_a\partial_{w_{-a}} \right). 
\end{eqnarray}
%%
%%%%%%%%%%%%%%%%%%%%%%%%%%%%%%%%%%%%%%%%%%%%%%%%%%%%%%%%%%%%%%%%%%%%%%%%%%%
%%%%
The genus $g$ partition function is formally M\"obius invariant \cite{TW} 
%%
%%%%%%%%%%%%%%%%%%%%%%%%%%%%%%%%%%%%%%%%%%%%%%%%%%%%%%%%%%%%%%%%%%%%%%%%%%%%%
%%
\begin{proposition}
\label{kartavo}
$\D^p Z^{(g)}_V=0$ for all $p\in \mathcal P_2$. \hfill $\qed$ 
\end{proposition}
%%
%%%%%%%%%%%%%%%%%%%%%%%%%%%%%%%%%%%%%%%%%%%%%%%%%%%%%%%%%%%%%%%%%%%%%%%%%%%%%%%%%%%%%%%%%
%%
Proposition \ref{kartavo} can be generalized 
to an $n$-point formal form $\F^{(g)}_V({\bf v,y})$ for 
 $n$ vectors $v_k\in V$ of weight $\wt(v_k)$. For $p\in \mathcal P_2$ we define 
for $\D^p$ of \eqref{mardot} 
%%
%%%%%%%%%%%%%%%%%%%%%%%%%%%%%%%%%%%%%%%%%%%%%%%%%%%%%%%%%%%%%%%
%%
\begin{eqnarray}
\label{rabotaj}
\D^p_{{\bf y}}=\D^p+\sum\limits_{k=1}^n
\left(p(y_k)\partial_{y_k}+\wt(v_k) \partial_{y_k} p(y_k)\right). 
\end{eqnarray}
%%
%%%%%%%%%%%%%%%%%%%%%%%%%%%%%%%%%%%%%%%%%%%%%%%%%%%%%%%%%%%%%%%
%%
\begin{proposition}
\label{rabf}
$\D^p_{{\bf y}}\F^{(g)}_V({\bf v,y})+ 
\sum\limits_{k=1}^n \frac{1}{2}\partial^2_{y_k} p(y_k) 
\F^{(g)}_V \left(\ldots ;L(1)v_k,y_k;\ldots \right)=0$,  
for all $p\in \mathcal P_2$. %\hfill $\qed$ 
%%
%%%%%%%%%%%%%%%%%%%%%%%%%%%%%%%%%%%%%%%%%%%%%%%%%%%%%%%%%%%%%%%%%%%%%%%%%%%%%
%%
$\D^p_{{\bf y}}\F^{(g)}_V({\bf v,y})=0$ for quasiprimary states $v_1$, $\ldots$, $v_n$.  
\hfill $\qed$  
\end{proposition}
%%
%%%%%%%%%%%%%%%%%%%%%%%%%%%%%%%%%%%%%%%%%%%%%%%%%%%%%%%%%%%%%%%%%%%%%%%%%%%%%
%%
It is a formal version of Proposition 5.3 (ii)  
of \cite{TW1} concerning  
 genus $g$ meromorphic forms in $n$ variables.
%%
%%%%%%%%%%%%%%%%%%%%%%%%%%%%%%%%%%%%%%%%%%%%%%%%%%%%%%%%%%%%%%%%%%%%%
%%
\subsection{The genus $g$ Ward identity}
\label{bendas}
For quasiprimary $u\in V$ of weight $N$, and $n$ vectors
 $v_k \in V$ of weight $\wt(v_k)$ for $1 \le k \le n$, we consider    
%%
%%%%%%%%%%%%%%%%%%%%%%%%%%%%%%%%%%%%%%%%%%%%%%%%%%%%%%%%%%%%%%%
%%
$\F_V(u, z; {\bf v,y})=Z_V(u, z; {\bf v,y})\;dz^N\;{\bf dy^{\wt(v)}}$. 
%% 
%%%%%%%%%%%%%%%%%%%%%%%%%%%%%%%%%%%%%%%%%%%%%%%%%%%%%%%%%%%%%
%%
Define the formal residue \cite{K, FHL, LL}
\begin{eqnarray}
\label{vozmoh}
\Res_a^l \F^{(g)}_V=\Res_a^l \F^{(g)}_V(u;{\bf v,y})= 
\Res_{z-w_a}\left(z-w_a\right)^l \F^{(g)}_V(u, z; {\bf v,y}) 
\\
\notag
= \sum\limits_{{\bf b}_+}Z_V^{(0)}
(\ldots;u(l).b_a,w_a;\ldots)\;{\bf dy^{\wt(v)}},  
\end{eqnarray}
for $0 \le l\le 2N-2$, and $a\in \{-1, \ldots, - g, 1, \ldots, g\}$.  
%% 
%%%%%%%%%%%%%%%%%%%%%%%%%%%%%%%%%%%%%%%%%%%%%%%%%%%%%%%%  
%% 
 Equation \eqref{vozmoh}
 follows from vertex operator algebra  
 locality and  associativity. Equation \eqref{septemb}
 implies a general Ward identity for 
genus $g$ correlation functions \cite{TW}.   
%%
%%%%%%%%%%%%%%%%%%%%%%%%%%%%%%%%%%%%%%%%%%%%%%%%%%%%%%%%%%%%%%%%%%% 
%% 
\begin{proposition} 
\label{kurst}
Let $u\in V$ be quasiprimary of weight $N$. %and let 
%%
%%%%%%%%%%%%%%%%%%%%%%%%%%%%%%%%%%%%%%%%%%%%%%%%%%%%%%%%  
%% 
Then for $p_a^l $ of \eqref{ereroitf}, 
and $P^{(l)}(y_0)$ acting on the dummy variable $y_0$,  
%% 
%%%%%%%%%%%%%%%%%%%%%%%%%%%%%%%%%%%%%%%%%%%%%%%%%%%%%%%%%%%%%%%%%%%%%
%% 
\begin{eqnarray*}
\sum\limits_{a=1}^g  \sum\limits_{l=0}^{2N-2}  p_a^l    
\Res_a^l \F^{(g)}_V(u;{\bf v,y})
=\sum\limits_{k=1}^n \sum\limits_{l=0}^{2N-2} \frac{\partial^l_{y_k} \left(p(y_k) \right) }{l!} 
\F^{(g)}_V(\ldots; u(l).v_k, y_k;\ldots) dy_k^{l+1-N}. 
&&
%%%%
\end{eqnarray*}
\end{proposition}
%%
%%%%%%%%%%%%%%%%%%%%%%%%%%%%%%%%%%%%%%%%%%%%%%%%%%%%%%%%%%%%%%%
%%
With  
 $\frac{1}{l!}\partial^l P(y_0)= -\sum_{a=1}^g \Res_a^l $, 
$\frac{1}{l!}\partial^l P(y_k)=\frac{1}{l!}\partial^l_{y_k} p(y_k)$,
 it is easy to see that 
the identity above can be represented as
%%
%%%%%%%%%%%%%%%%%%%%%%%%%%%%%%%%%%%%%%%%%%%%%%%%%%%%%%%%%%%%%%%
%%
 $\sum_{k=0}^n$ $\frac{1}{l!}$ $\partial^l P (y_k)$   
$\sum_{l=0}^{2N-2}$   
$\F^{(g)}_V$ $(\ldots$;  $u(l).v_k$, $y_k$; $\ldots)$ \; $dy_k^{l+1-N}=0$. 
%%
%%%%%%%%%%%%%%%%%%%%%%%%%%%%%%%%%%%%%%%%%%%%%%%%%%%%%%%%%%%%%%%%%%%%%%%%%%%%
%% 
\subsection{The genus $g$ Zhu recursion} 
\label{systemo} 
We now recall genus $g$ Zhu reduction formulas
 generalizing Propositions  
\ref{raisi} and \ref{hamen}.  
We define $\Pi^{(g)}_N(z,y)=\pi^{(g)}_N(z, y)\;dz^N \;dy^{1-N}$ 
for $f^{(g)}_{N, i, j}(z,y)=\pi^{(g)}_N(z,y)$
 of \eqref{eschechl}  
%%
%%%%%%%%%%%%%%%%%%%%%%%%%%%%%%%%%%%%%%%%%%%%%%%%%%%%%%%%%%%
%%
 to be that 
determined by $\Psi^{(g)}_N(z,y)$ 
  of \eqref{erpoerksr}, $z$, $y\in \Omega_0 (\Gamma)$, 
%%
%%%%%%%%%%%%%%%%%%%%%%%%%%%%%%%%%%%%%%%%%%%%%%%%%%%%%%%%%%
%%
$\Psi^{(g)}_N(z,y)=
\sum_{\gamma\in\Gamma} \Pi^{(g)}_N(\gamma z,y)$,  
$\Pi^{(g)}_N(z,y)=\Pi^{(g)}_N(z, y; {\bf A}_{2N-2})$,    
%%
%%%%%%%%%%%%%%%%%%%%%%%%%%%%%%%%%%%%%%%%%%%%%%%%%%%%%%%%%%%
%% 
with $f_l (z)$ of \eqref{pwrepefk}. 
%%
%%%%%%%%%%%%%%%%%%%%%%%%%%%%%%%%%%%%%%%%%%%%%%%%%%%%%%%%%%%%%%%
%%%%
With that convergence of the coefficient  
functions appearing in the genus $g$ Zhu  
reduction in terms of derivatives of $\Psi^{(g)}_N(z,y_k)$  
 and the $N$-form spanning set $\{\Theta_{N,a}^l (z)\}$ is provided. 
%%
%%%%%%%%%%%%%%%%%%%%%%%%%%%%%%%%%%%%%%%%%%%%%%%%%%%%%%%%%%%%%%%%%%%%%%%%%%
%%
In particular, we use the sewing formulas 
of Theorem \ref{palre}  
for $\Psi^{(g)}_N(z,y)$ and the formula of Proposition \ref{paraskas}    
for $\Theta_{N,a}^l(z)$. 
%%  
%%%%%%%%%%%%%%%%%%%%%%%%%%%%%%%%%%%%%%%%%%%%%%%%%%%%%%%%%%%%%%%%%%%%%%%%%%%%%%
%% 
In \cite{TW} the following theorem for the case of
 quasiprimary genus $g$ Zhu recursion was proven. 
%% 
%%%%%%%%%%%%%%%%%%%%%%%%%%%%%%%%%%%%%%%%%%%%%%%%%%%%%%%%%%%%%%%%%%%%%%%%%%%%%%%%
%%
\begin{theorem}
\label{mardato}
 Let $V$ be a simple vertex operator algebra  of strong 
conformal field theory-type with $V$ isomorphic to $V'$.
The genus $g$ correlation  differential form for quasiprimary
 $u$ of weight $N\ge 1$ inserted 
at $z$ and  $v_1$, $\ldots$, $v_n\in V$ 
inserted at $y_1$,$\ldots$, $y_n$ respectively,  
satisfies the reduction identity  
%%
%%%%%%%%%%%%%%%%%%%%%%%%%%%%%%%%%%%%%%%%%%%%%%%%%%%%%%%%%%
%%
\begin{eqnarray}
\notag
\label{rpoerpokf}
 \F^{(g)}_V(u, z; {\bf v,y})=\sum\limits_{a=1}^g \sum\limits_{l=0}^{2N-2}  
\Theta_{N,a}^l (z)\Res_a^l  \F^{(g)}_V(u;{\bf v,y}) 
\qquad \qquad \qquad && 
\\
+\sum\limits_{k=1}^n\sum\limits_{j\ge 0}
\frac{1}{j!} \partial^j_{y_k} 
\left( \Psi^{(g)}_N(z, y_k) \right) 
\F^{(g)}_V(\ldots;u(j).v_k,y_k; \ldots)\;dy_k^j. \hfill \qed &&
\end{eqnarray}
\end{theorem} 
%%
%%%%%%%%%%%%%%%%%%%%%%%%%%%%%%%%%%%%%%%%%%%%%%%%%%%%%%%%%%%%%%%%%%%%%%%%%%%%%%%%%%
%%
 With $f^{(g)}_{N, 0, 0, j}= 
\sum_{a=1}^g \sum_{l=0}^{2N-2}   
\Theta_{N,a}^l (z)\Res_a^l$ and 
 $f^{(g)}_{N, k, 0, j}(z, y_k)=
\frac{1}{j!} \partial^j_{y_k} \Psi^{(g)}_N(z, y_k)$,  
\eqref{rpoerpokf} becomes 
%%
%%%%%%%%%%%%%%%%%%%%%%%%%%%%%%%%%%%%%%%%%%%%%%%%%%%%%%%%%%%%%%%%   
%% 
 $\F^{(g)}_V(u, z; {\bf v,y}) =  
\sum_{i=0}^n\sum_{j\ge 0} f^{(g)}_{N, 0, 0, j}(z, y_i) \; 
\F^{(g)}_V(\ldots;u(j).v_i, y_i; \ldots)\;dy_i^j$.   
%%
%%%%%%%%%%%%%%%%%%%%%%%%%%%%%%%%%%%%%%%%%%%%%%%%%%%%%%%%%%%%%%%%%%%%%%%%%%%%%
%%
Similar to Corollary \ref{ezdap}, 
 one may generalize Theorem \ref{mardato}.  
%%
%%%%%%%%%%%%%%%%%%%%%%%%%%%%%%%%%%%%%%%%%%%%%%%%%%%%%%%%%%%%%%%%%%%%%%%%%%%%%%%%
%%
\begin{corollary} 
\label{proepttm}
The genus $g$ formal $n$-point  differential
  for a quasiprimary  descendant $L^{(i)}(-1)u$ for 
$u$, $\wt (u)=N$, inserted at $z$, and general vectors 
$v_1$, $\ldots$, $v_n$ inserted at $y_1$,   
$\ldots$, $y_n$ respectively, satisfies the recursive identity
%%
%%%%%%%%%%%%%%%%%%%%%%%%%%%%%%%%%%%%%%%%%%%%%%%%%%%%%%%%%%%%%%%%%%%%%
%%
\begin{eqnarray}
\label{murdoy}
\F^{(g)}_V \left(L^{(i)}(-1)u, z; {\bf v,y}\right)=
\sum\limits_{a=1}^g\sum\limits_{l=0}^{2N-2} \partial_z^{(i)} 
\left( \Theta_{N,a}^l (z) \right) \Res_a^l \F^{(g)}_V(u;{\bf v,y})
 \qquad \qquad &&
\\
\notag 
+\sum\limits_{k=0}^n \sum\limits_{j\ge 0} 
\frac{\partial^i_z \partial^j_{y_k}}{i!j!}    
\left(\Psi^{(g)}_N(z, y_k) \right) 
\F^{(g)}_V(\ldots; u(j).v_k,y_k; \ldots)\;dz^i \;dy_k^j. \hfill \qed && 
\end{eqnarray}
\end{corollary}
%%
%%%%%%%%%%%%%%%%%%%%%%%%%%%%%%%%%%%%%%%%%%%%%%%%%%%%%%%%%%%%%%%%%%5
%%
With dummy $y_0$ and 
$f^{(g)}_{N, 0, i, j}$ $(z, y_0)$ 
$=$ $\frac{1}{i!}$ $\partial^i_z$ $\frac{1}{j!}$ $\partial^j_{y_k}$   
$\Psi^{(g)}_N$ $(z, y_k)$ $=$     
$\sum_{a=1}^g$ $\sum_{l=0}^{2N-2}$ $\partial_z^{(i)}$  
$\Theta_{N,a}^l (z)$ $\Res_a^l.$, 
 $f^{(g)}_{N, 0, i, j}$ $(z, y_k)$ $=$ 
$\frac{1}{i!}$ $\partial^i_z$ $\frac{1}{j!}$ $\partial^j_{y_k}$ $\Psi^{(g)}_N(z,y_k)$  
%%%
 \eqref{murdoy} turns into the form of one-sum reduction operator.    
%% 
%%%%%%%%%%%%%%%%%%%%%%%%%%%%%%%%%%%%%%%%%%%%%%%%%%%%%%%%%%%%%%%%%%%%%%%%%%%%%%%
%% 
$\F^{(g)}_V \left(L^{(i)}(-1)u \right.$, $z$ ; $\left.{\bf v,y}\right)$ 
$=$ $\sum_{k=0}^n$ $\sum_{j\ge 0}$  
$f^{(g)}_{N, k, i, j}(z, y_k)      
\F^{(g)}_V(\ldots; u(j).v_k,y_k; \ldots)\;dz^i \;dy_k^j$.    
%%
%%%%%%%%%%%%%%%%%%%%%%%%%%%%%%%%%%%%%%%%%%%%%%%%%%%%%%%%%%%%%%%%%%%%%%%%%%%%%%%%%%%%%
%%
\subsection{The genus $g$ chain condition}
Let us prove the following 
%%
%%%%%%%%%%%%%%%%%%%%%%%%%%%%%%%%%%%%%%%%%%%%%%%%%%%%%%%%%%%%%%%%%%%%%%%%%%%%%%%%%
%%
\begin{proposition}
With the condition\eqref{pendozo3}   
the chain condition \eqref{prodo} is fulfilled.   
\end{proposition}
%%
%%%%%%%%%%%%%%%%%%%%%%%%%%%%%%%%%%%%%%%%%%%%%%%%%%%%%%%%%%%%%%%%%%%%%%%
%%
\begin{proof}
For the chain condition we have 
\begin{eqnarray*}
\sum 
%% 
%%%%%%%%%%%%%%%%%%%%%%%%%%%%%%%%%%%%%%%%%%%%%%%%%%%%%%%%%%%%%%%%%%%
%%
\left(  
\sum\limits_{a=1}^g \sum\limits_{l_m=0}^{2N_m-2}  
\frac{1}{i!}\partial_{z_m}^{i_m}       
\left( \Theta_{N_m, a_m}^{l_m}(z_m)\right) 
 \Res_{a_m}^{l_m} \F^{(g)}_V(u_m;{\bf v, y})   
\right. \qquad \qquad \qquad \qquad \qquad 
%%
%%%%%%%%%%%%%%%%%%%%%%%%%%%%%%%%%%%%%%%%%%%%%%%%%%%%%%%%%%%%%%%%%% 
\nn
\left. 
+\sum\limits_{k=0}^n \sum\limits_{j\ge 0} 
\frac{1}{i!} \frac{1}{j!} \partial^n_{z_m} \partial^j_{y_k} 
\left(\Psi^{(g)}_{N_m}(z_m, y_k) \right) \;   
 dz_m^{i_m} \; dy_k^j.  
\right) f^{(g)}_{N'_{m'}, k', i'_{m'}, j'} 
\left(z'_{m'}, y'_{k'}\right) A^{i'_{m'}i}_{j'j}.   
\end{eqnarray*}
%%
%%%%%%%%%%%%%%%%%%%%%%%%%%%%%%%%%%%%%%%%%%%%%%%%%%%%%%%%%%%%%%%%%%%%%%
%%
Using the same procedure for formulas of Subsection \ref{diktor}, 
in particular, \eqref{ererepo}--\eqref{mandat}, 
 formulas of Theorem \eqref{palre} and of 
Proposition \eqref{paraskas}, 
and summing for suitable $N_m$, $N'_{m'}$ 
%% 
%%%%%%%%%%%%%%%%%%%%%%%%%%%%%%%%%%%%%%%%%%%%%%%%%%%%%%%%%%%%%%%%%%%%%
%%
we have to get the genus $g$ ward identity of Proposition \ref{kurst} 
for with \eqref{ereroitf}.     
%%
%%%%%%%%%%%%%%%%%%%%%%%%%%%%%%%%%%%%%%%%%%%%%%%%
%%
Namely, 
%% 
%%%%%%%%%%%%%%%%%%%%%%%%%%%%%%%%%%%%%%%%%%%%%%%%
%%
\begin{eqnarray*}
&& \frac{1}{i_m!} \partial_{z_m}^{i_m} \Theta_{N_m, a_m}^{l_m}(z_m)  
= \frac{1}{i!}\partial_{z_m}^{i_m}   
\left(T_{a_m}^{l_m}(z_m)+(-1)^{N_m} \; 
\rho_{a_m}^{N_m-1-l_m}\; T_{-a_m}^{2N_m-2-l_m}(z_m) \right) 
\end{eqnarray*}
%%
%%%%%%%%%%%%%%%%%%%%%%%%%%%%%%
%%
\begin{eqnarray*}
&&
=  
\rho_{a_m}^{-\half l_m} \;  
 \frac{1}{i_m!} \partial_{z_m}^{i_m} L(z_m)(I +  D(I-\widetilde{A})^{-1}A )_{a_m}^{l_m}   
\end{eqnarray*}
%%
%%%%%%%%%%%%%%%%%%%%%%%%%%%%%%%%%%%%%%%%%%%%%%%%%5
%%
\begin{eqnarray*}
&&  
+(-1)^{N_m}\; \rho_{a_m}^{N_m-1-l_m} \; 
\rho_{-a_m}^{-\half (2N_m-2-l_m)} \;  
\frac{1}{i!}\partial_{z_m}^{i_m} L(z_m)  
(I +  D (I-\widetilde{A})^{-1}A )_{-a_m}^{2N_m-2-l_m}.  
\end{eqnarray*}
%%
%%%%%%%%%%%%%%%%%%%%%%%%%%%%%%%%%%%%%%%%%%%%%%%%%%%%%
%%
For an element of the vector $(L(z))$ we obtain 
for $z_m \ne 0$, $|\frac{w_b}{z_m}|<1$,  
%%
%%%%%%%%%%%%%%%%%%%%%%%%%%%%%%%%%%%%%%%%%%%%%%%%%%%%%%
%%
\begin{eqnarray*} 
  \frac{1}{i_m!}\partial^{i_m}_{z_m} L(z_m)= \partial_{z_m}^{(i_m)}  
\left( \rho_b^{\half n} 
\frac{1}{n!} \partial^n_{w_b}
\pi_{N_m}^{(g)}(z_m, w_b)\; dz_m^{N_m}  \right) 
 \qquad \qquad \qquad \qquad \qquad \qquad
  &&
\end{eqnarray*}
%%
%%%%%%%%%%%%%%%%%%%%%%%%%%%%%%%%%%%%%%%%%%%%%%%%%%
%% 
\begin{eqnarray*}
  = \frac{1}{i_m!}\partial_{z_m}^{i_m}  
\left( \rho_b^{\half n}  
\frac{1}{n!} \partial^n_{w_b} 
\left(\frac{1}{z_m-w_b} +\sum\limits_{l_m=0}^{2N_m-2} 
 f_{l_m}(z_m) w_b^{l_m}\right)  
\; dz_m^{N_m}  \right) 
\qquad \qquad \qquad \qquad \qquad \qquad 
 &&
\end{eqnarray*}
%%
%%%%%%%%%%%%%%%%%%%%%%%%%%%%%%%%%%%%%%%
%% 
\begin{eqnarray*}
 =  \frac{\partial_{z_m}^{i_m}}{i_m!}   
\left( \rho_b^{\half n}  
\left(\frac{(-1)^n}{(z_m-w_b)^{n+1}}  
+\sum\limits_{l_m=0}^{2N_m-2}   
f_{l_m}(z_m) l_m(l_m-1)\ldots(l_m-n+1) w_b^{l_m-n}\right) 
 dz_m^{N_m}  \right) &&
\end{eqnarray*}
%%
%%%%%%%%%%%%%%%%%%%%%%%%%%%%%%%%%%%%%%%%%%%%%%%%%%%%%%%%%%
%% 
\begin{eqnarray*}
 =   \rho_b^{\half n} 
\left(\frac{(-1)^n}{(z_m-w_b)^{i_m+n+1}}  
+\sum\limits_{l_m=0}^{2N_m-2}   
 \frac{1}{i_m!}\partial^{i_m}_{z_m} f_{l_m}(z_m) 
l_m(l_m -1) \ldots (l_m-n+1) w_b^{l_m-n}\right)  
\; dz_m^{N_m}   && 
\end{eqnarray*}
%%
%%%%%%%%%%%%%%%%%%%%%%%%%%%%%%%%%%%%%%%%%%%%%%%%%%%%%%%%%%%%%
%%
\begin{eqnarray*}
=   \rho_b^{\half n} 
\left(\frac{(-1)^n}{z_m^{i_m+n+1} \left(1-\frac{w_b}{z_m}\right)^{i_m+n+1}}  
+\sum\limits_{l_m=0}^{2N_m-2}   
 \frac{\partial^{i_m}_{z_m} f_{l_m}(z_m) }{i_m!}
l_m(l_m -1) \ldots (l_m-n+1) w_b^{l_m-n}\right)  
\; dz_m^{N_m} && 
\end{eqnarray*}
%%
%%%%%%%%%%%%%%%%%%%%%%%%%%%%%%%%%%%%%%%%%%%%%%%%%%%%%%%%%%%
\begin{eqnarray*}
=   \rho_b^{\half n} 
\left(\frac{(-1)^n}{ z_m^{i_m+n+1} } 
\left(1 + \sum\limits_{s \ge 1} \frac{(i_m+n+1)(i_m+n+2)\ldots (i_m+n+1+s)}{s!} 
 \left(\frac{w_b}{z_m}\right)^s \right) \right. && 
\nn
\left. 
+\sum\limits_{l_m=0}^{2N_m-2}   
\frac{1}{i_m!}\partial^{i_m}_{z_m}  f_{l_m}(z_m)
 l_m(l_m -1) \ldots (l_m-n+1) w_b^{l_m-n}\right)  
\; dz_m^{N_m}. && 
\end{eqnarray*}
%%
%%%%%%%%%%%%%%%%%%%%%%%%%%%%%%%%%%%%%%%%%%%%%%%%%%%%%%%%%%%%%%%%%%%%%%%%%%
%%
The identification is the following, 
%% 
%%%%%%%%%%%%%%%%%%%%%%%%%%%%%%%%%%%%%%%%%%%%%%%%%%%%%%%%%%%%%%%%%%%%%%%%%%
%% 
\begin{eqnarray}
\label{pendozo3}
   \partial^{i_m}_{z_m} f_{l_m}(z_m)   
=- \frac{(-1)^n\; m \; i_m!(i_m+n+1)(i_m+n+2)\ldots (i_m+n+1+s)} 
{s! \; l_m (l_m -1) \ldots (l_m-n+1) \;z_m^{i_m+n+s+1} }. 
\end{eqnarray}
%%
%%%%%%%%%%%%%%%%%%%%%%%%%%%%%%%%%%%%%%%%%%%%%%%%%%%%%%%%%%%%%%%%%%%%%%%
%%
Similarly, in the opposite case, $w_b \ne 0$, $|\frac{z_m}{w_b}|<1$, 
we obtain 
%% 
%%%%%%%%%%%%%%%%%%%%%%%%%%%%%%%%%%%%%%%%%%%%%%%%%%%%%%%%%%%
%% 
\begin{eqnarray*} 
  \rho_b^{\half n} 
\left(\frac{(-1)^n}{ w_b^{i_m+n+1} }  
\left(1 + \sum\limits_{s \ge 1} \frac{(-1)^s (i_m+n+1)(i_m+n+2)\ldots (i_m+n+1+s)}{s!} 
 \left(\frac{z_m}{w_b}\right)^s \right) \right. &&  
\end{eqnarray*}
%%
%%%%%%%%%%%%%%%%%%%%%%%%%%%%%%%%%%%%%%%%%%%%%%%%55
%%
\begin{eqnarray*} 
%%\nn
%%
\left. 
+\sum\limits_{l_m=0}^{2N_m-2}   
 \frac{1}{i_m!}\partial^{i_m}_{z_m} f_{l_m}(z_m) 
 l_m(l_m -1) \ldots (l_m-n+1) w_b^{l_m-n}\right)  
\; dz_m^{N_m},  && 
\end{eqnarray*}
%%
%%%%%%%%%%%%%%%%%%%%%%%%%%%%%%%%%%%%%%%%%%%%%%%%%%%%%%%%%%%%%%%%%%%%%%%%
%%
with the identification $z_m=w_b^{\frac{i_m+n+1}{s}+\alpha_m}$, 
$0 < \alpha_m <1$, but  
$\alpha_m=l_m-n \ge 0$. Thus, this case is not possible. 
%%
%%%%%%%%%%%%%%%%%%%%%%%%%%%%%%%%%%%%%%%%%%%%%%%%%%%%%%%%%%%%%%%%%%%%%%%%%%%
%%
For the second term of Ward identity of Proposition \ref{kurst} with 
 $z_m$, $y_k \in \Omega_0(\Gamma)$ we have to get 
$\sum_{k=1}^n$ $\frac{1}{l_m!}$ $\partial^l_{y_k}$ $\left(p(y_k)\right)$ 
$\sum_{l=0}^{2N-2}$   
$\F^{(g)}_V$ $(\ldots$; $u(l).v_k$, $y_k$; $\ldots)$ $dy_k^{l+1-N}=0$. 
%%
%%%%%%%%%%%%%%%%%%%%%%%%%%%%%%%%%%%%%%%%%%%%%%%%%%%%%%%%%%%%%%%%%%%%%%%
%%
Thus we have to identify $z_m$ with some expression of $y_k$,   
%%
%%%%%%%%%%%%%%%%%%%%%%%%%%%%%%%%%%%%%%%%%%%%%%%%%%%%%%%%%%%%%%%%%%%%%% 
%% 
\begin{eqnarray*} 
\sum
\frac{\partial^{i_m}_{z_m} \partial^j_{y_k}   }{i_m!j!}  
\left( \Psi^{(g)}_{N_m}(z_m, y_k) \right) \;      
 dz^{i_m}_m \;dy_k^j \; 
%%
%%%%%%%%%%%%%%%%%%%%%%%%%%%%%%%%%%%%%%%%%%%%%%%%%%%%%%%%%
%%
 f^{(g)}_{N'_{m'}, k', i'_{m'}, j'}(z'_{m'}, y'_{k'})  
\; d(z')^{i'_{m'}}_{m'} \; d(y')^{j'}_{k'} \; A^{i_m i'_{m'}}_{jj'} 
&&     
\end{eqnarray*}
%%
%%%%%%%%%%%%%%%%%%%%%%%%%%%%%%%%%%%%%%%%%%%%%%%%%%%%%%%%%%%%%
%%
\begin{eqnarray*}
= 
\sum
\frac{\partial^{i_m}_{z_m}  \partial^j_{y_k}  }{i_m!j!} 
\left( \sum\limits_{\gamma\in\Gamma} \Pi^{(g)}_N(\gamma z_m, y_k) \right)  
dz_m^{i_m} dy_k^j 
%%
%%%%%%%%%%%%%%%%%%%%%%%%%%%%%%%%%%%%%%%%%
%%
 f^{(g)}_{N'_{m'}, k', i'_{m'}, j'} (z'_{m'}, y'_{k'}) 
 d(z')^{i'_{m'}}_{m'}  d(y')^{j'}_{k'} A^{i_m i'_{m'}}_{jj'} &&  
\end{eqnarray*}
%%
%%%%%%%%%%%%%%%%%%%%%%%%%%%%%%%%%%%%%%%%%%%%%%%%%%
%%
\begin{eqnarray*}
= 
\sum   
\frac{\partial^{i_m}_{z_m}\partial^j_{y_k}  }{i_m!j!}  
\sum\limits_{\gamma \in\Gamma}
\left(\frac{1}{\gamma z_m-y_k} 
+ \sum\limits_{l_m=0}^{2N_m-2} f_{l_m}(\gamma z_m)  
y_k^{l_m}\right) 
dz_m^{i_m}  dy_k^j  &&
\end{eqnarray*}
%%
%%%%%%%%%%%%%%%%%%%%%%%%%%%%%%%%%%%%%%%%
%%
\begin{eqnarray*}
\nn
 f^{(g)}_{N'_{m'}, k', i'_{m'}, j'} (z'_{m'}, y'_{k'}) 
 dz^{i'_{m'}}_{m'}  d(y')^{j'}_{k'}  A^{i_m i'_{m'}}_{jj'}  && 
\end{eqnarray*}
%% 
%%%%%%%%%%%%%%%%%%%%%%%%%%%%%%%%%%%%%%%%%%%%%%
%%
\begin{eqnarray*}
= 
 \sum
\sum\limits_{\gamma\in\Gamma}
%%
%%%%%%%%%%%%%%%%%%%%%%%%%%%%%%%%%%%%%%%%%%%%%%%%%%%%%%%%%%%%%
%%
\left( \frac{(-1)^{i_m}\gamma^{i_m}}{i_m!j!(\gamma z_m-y_k)^{i_m+j+1}} \right.
 \qquad \qquad \qquad \qquad \qquad \qquad \qquad \qquad &&
\end{eqnarray*}
%%%%%%%%%%%%%%%%%%%%%%%%%%%%%%%%%%%%%%%%%%%%%%%%%%%%%%%%%%%%
%%
\begin{eqnarray*} 
 \left. 
  +\sum\limits_{l_m=0}^{2N_m-2} l_m(l_m-1) 
\ldots (l_m -i_m) \gamma^{i_m}   
\frac{1}{i_m!} \frac{1}{j!} \partial^{i_m} f_{l_m}(\gamma z_m) y_k^{l_m-j} \right) \; 
dz_m^{i_m} \;dy_k^j \;   && 
\end{eqnarray*}
%%
%%%%%%%%%%%%%%%%%%%%%%%%%%%%%%%%%%%%%%%%%%%%%%%%%%%%%%%%%%%
%%
\begin{eqnarray*} 
 f^{(g)}_{N'_{m'}, k', i'_{m'}, j'} (z'_{m'}, y'_{k'})
\; dz^{i'_{m'}}_{m'} \; d(y')^{j'}_{k'} \; A^{i_m i'_{m'}}_{jj'}.  &&
\end{eqnarray*}
%%
%%%%%%%%%%%%%%%%%%%%%%%%%%%%%%%%%%%%%%%%%%%%%%%%%%%%%%%%%%%%%%%%%%%%%%
%%
Note that due to the finite polynomial in $y_k$, the second summand 
vanishes after some $j$.  
%%
%%%%%%%%%%%%%%%%%%%%%%%%%%%%%%%%%%%%%%%%%%%%%%%%%%%%%%%%%%%%%%%%%%%%%
%%
Using the same trick as for $g=0$,
 we expand 
the first summand for $\gamma z_m\ne 0$ in terms of $|\frac{y_k}{\gamma z_m}|<1$, 
and sum over $N_m$ (the summation is started from zero now), 
%%
%%%%%%%%%%%%%%%%%%%%%%%%%%%%%%%%%%%%%%%%%%%%%%%%%%%%%%%%%%%%%%%%%%%%%%
%%
\begin{eqnarray*} 
\sum 
\sum\limits_{\gamma\in\Gamma}
%%
%%%%%%%%%%%%%%%%%%%%%%%%%%%%%%%%%%%%%%%%%%%%%%%%%%%%%%%%%%%%%%%%%%%%%%
%%
\left(\frac{(-1)^{i_m}\gamma^{i_m}}{i_m!j! (\gamma z_m)^{i_m+j+1} }  
 \left(\sum\limits_{s \ge 0}
\frac{(i_m+j+1)(i_m+j+2) \ldots (i_m+j+s)}{s!} 
\left( \frac{y_k}{\gamma z_m}\right)^s \right) \right. && 
%%
%%%%%%%%%%%%%%%%%%%%%%%%%%%%%%%%%%%%%%%%%%%%%%%%%%%%%%%%%%%%%%%%%%%%
\nn
 \left. 
+\sum\limits_{l_m=0}^{2N_m-2} \frac{1}{j!} l_m (l_m-1) \ldots (l_m -i_m) \gamma^{i_m}  
\frac{1}{i_m!}\partial^{i_m}_{z_m} f^{(g)}_{l_m}(\gamma z_m) y_k^{l_m-j}\right) \; 
dz_m^{i_m} \;dy_k^j &&
%%
%%%%%%%%%%%%%%%%%%%%%%%%%%%%%%%%%%%%%%%%%%%%%%%%%%%%%%%%%%%%%%%%
\nn
f^{(g)}_{N'_{m'}, k', i'_{m'}, j'}\left(z'_{m'}, y'_k\right)  
\; dz^{i'_{m'}}_{m'} \; d(y')^{j'}_{k'} \; A^{i_m i'_{m'}}_{jj'}. &&
\end{eqnarray*}
%%
%%%%%%%%%%%%%%%%%%%%%%%%%%%%%%%%%%%%%%%%%%%%%%%%%%%%%%%%%%%%%%%%%%%%%%%%%
%%
As for the genus zero case consider $\alpha$, $\beta$, $\xi$ \eqref{merda} 
and corresponding $\alpha'$, $\beta'$, $\xi'$ 
but with $f^{(g)}_{l_m}(\gamma z_m)$.  
%% 
%%%%%%%%%%%%%%%%%%%%%%%%%%%%%%%%%%%%%%%%%%%%%%%%%%%%%%%%%%%%%%%%%%%
%%
Using the same arguments one sees that it is not possible to make 
vanish as a product of two polynomials. 
%%
%%%%%%%%%%%%%%%%%%%%%%%%%%%%%%%%%%%%%%%%%%%%%%%%%%%%%%%%%%%%%%%%%%%
%%
Therefore, we take $s_0 \ge 0$, and with $s=l_m-j$, and, therefore, 
%%
%%%%%%%%%%%%%%%%%%%%%%%%%%%%%%%%%%%%%%%%%%%%%%%%%%%%%%%%%%%%%%%%%%%%%
%%
\begin{eqnarray}
\label{pordolozo1}
 \partial^{i_m} f_{l_m}(\gamma z_m) 
=  \frac{(-1)^{i_m} i_m! \; (i_m+j+1)(i_m+j+2) \ldots (i_m+j+s) }
{s! \; l_m (l_m-1) \ldots (l_m -i_m) (\gamma z_m)^{i_m+j+s+1} }. 
\end{eqnarray}
%%
%%%%%%%%%%%%%%%%%%%%%%%%%%%%%%%%%%%%%%%%%%%%%%%%%%%%%%%%%%%%%%%%%%%
%%
Similar to the genus zero case, the alternative case 
 $\left|\frac{\gamma z_m}{y_k}\right|<1$ is not possible. 
%%
%%%%%%%%%%%%%%%%%%%%%%%%%%%%%%%%%%%%%%%%%%%%%%%%%%%%%%%%%%%%%%%%%%%
%%
Thus the identity of Proposition \ref{kurst} 
 is reconstructed and the chain condition is satified.  
\end{proof}
%%  
%%%%%%%%%%%%%%%%%%%%%%%%%%%%%%%%%%%%%%%%%%%%%%%%%%%%%%%%%%%%%%%%%%%%%%%%%%%%%%%%
%%
For the cohomology, with  ${\bf a}_n$ corresponding to ${\bf v}_n$, 
 we obtain the factor space 
 $H^n (V, g)$ $=$  
 $\prod_{i=1}^n$ $\sum_{k \ge 1}^i$ $\sum_{j \ge 1}$    
 $\Psi^{(g)}_{N, i, k, j}$ $(x_i)/$ 
$\prod_{i'=1}^{n-1}$ $\sum_{k' \ge 1}^{i'}$ $\sum_{ j' \ge 1}$ 
 $\Psi^{(g)}_{N,  i', k', j'}$ $(x'_{i'})$  
$=$  $\mathcal L\left({\bf a}^{-1}_n.{\bf y}_n \right)/$ 
$\mathcal L\left({\bf a}^{-1}_{n-1}.{\bf y}'_n \right)$. 
%% 
%%%%%%%%%%%%%%%%%%%%%%%%%%%%%%%%%%%%%%%%%%%%%%%%%%%%%%%%%%%%%%%%%%%%%%%%%%%%%%%%%
%%
From \cite{TW1} we know that 
 $\Theta_{N,a}^l (z)$, $\Psi^{(g)}_N(z,y)$ terms depend on $N=\wt(u)$ but are
otherwise independent of the vertex operator algebra $V$,
 i.e., they are  analogues of the genus zero $\Pi^{(g)}_N(z, y)$ coefficients   
and the genus one Weierstrass $P_1$ coefficients found in \cite{Z1}.   
%% 
%%%%%%%%%%%%%%%%%%%%%%%%%%%%%%%%%%%%%%%%%%%%%%%%%%%%%%%%%%%%%%%%%%%%%%%%%
%% 
 The equation \eqref{rpoerpokf} is independent of the choice
  $\Psi^{(g)}_N(z,y)$ of 
and the $N$-form spanning set  
$\{\widehat{\Theta}^l_{N,a}(z)\}$. 
%%
%%%%%%%%%%%%%%%%%%%%%%%%%%%%%%%%%%%%%%%%%%%%%%%%%%%%%%%%%%%%%%%%%%%%%%%%%%%%%%%%%%%%%%%%%%
%%%%%%%%%%%%%%%%%%%%%%%%%%%%%%%%%%%%%%%%%%%%%%%%%%%%%%%%%%%%%%%%%%%%%%%%%%%%%%%%%%%%%%%%%%
%%
\section*{Acknowledgment}
The author is supported by the Institute of Mathematics, 
 the Academy of Sciences of the Czech Republic (RVO 67985840). 
%%
%%%%%%%%%%%%%%%%%%%%%%%%%%%%%%%%%%%%%%%%%%%%%%%%%%%%%%%%%%%%%%%%%%%%%%%%%%%%%%%%%%%%%%%%%%
%%%%%%%%%%%%%%%%%%%%%%%%%%%%%%%%%%%%%%%%%%%%%%%%%%%%%%%%%%%%%%%%%%%%%%%%%%%%%%%%%%%%%%%%%%
%%
\section{Appendix: Differential structures on Riemann surfaces}
\label{vtorog}
%%
%%%%%%%%%%%%%%%%%%%%%%%%%%%%%%%%%%%%%%%%%%%%%%%%%%%%%%%%%%%%%%%%%%%%%%%%%%%%%%%%
%%
\subsection{The Bers quasiform}
In this Subsection 
 we recall the Bers quasiform $\Psi^{(g)}_N(z,y)$  
which is defined for $g\ge 2$ and $N\ge 2$. 
%% 
%%%%%%%%%%%%%%%%%%%%%%%%%%%%%%%%%%%%%%%%%%%%%%%%%%%%%%%%%%%%%%%%%%%%%%%%%%%%%%%%
%% 
In order to construct the Bers potential for 
holomorphic $N$-forms, Bers introduced $\Psi^{(g)}_N(z,y)$ in \cite{Be1, Be2}.  
%%
%%%%%%%%%%%%%%%%%%%%%%%%%%%%%%%%%%%%%%%%%%%%%%%%%%%%%%%%%%%%%%%%%%%%%%%%%%%%%%%%%%%
%%
It is also useful for the construction of
the  Laplacian determinant line bundle associated with 
$N$-forms \cite{McIT}.  
%%
%%%%%%%%%%%%%%%%%%%%%%%%%%%%%%%%%%%%%%%%%%%%%%%%%%%%%%%%%%%%%%%%%%%
%%
In the Subsections \ref{bendas} and \ref{systemo}  
 it is shown that  $\Psi^{(g)}_N(z,y)$ and
the associated  $N$-form spanning set  
$\{\Theta_{N, a}^l (z)\}_{1 \le a\le g}^{0 \le l \le 2N-2}$   
 play an important role in genus $g$ Zhu reduction formulas for 
vertex operator algebras.  
%%

%%%%%%%%%%%%%%%%%%%%%%%%%%%%%%%%%%%%%%%%%%%%%%%%%%%%%%%%%%%%%%%%%%%%%%%%%%%%%%%%%%
%%
The Bers quasiform of weight $(N,1-N)$ for $g \ge 2$ 
and $N \ge 2$ is defined by the Poincar\'e  
series \cite{Be1, Be2, TW1} 
for M\"obius invariant $\Pi_N(z,y)$ for all $\gamma\in\SL_2(\C)$,
 $z$,$y\in \Omega_0 (\Gamma)$, where  
 ${\bf A}_{2N-2}=(A_0$, $\ldots$, $A_{2N-2})$, $A_l\in\Lambda(\Gamma)$   
are distinct limit points of $\Gamma$, 
%%
%%%%%%%%%%%%%%%%%%%%%%%%%%%%%%%%%%%%%%
%%
\begin{eqnarray}
\label{olovodo} 
\Psi^{(g)}_N(z,y)=
\sum\limits_{\gamma\in\Gamma} \Pi^{(g)}_N(\gamma z,y), 
\qquad \qquad \qquad\qquad \qquad \qquad \qquad \qquad \qquad \qquad \qquad  &&
\\
\label{erpoerksr}
 \Pi^{(g)}_N(z,y)=\Pi^{(g)}_N(z, y ;{\bf A}_{2N-2})
=\Pi^{(g)}_N(\gamma z, \gamma y;{\bf \gamma A}_{2N-2}) 
= \frac{dz^Ndy^{1-N}}{z-y}\prod_{l=0}^{2N-2} 
\frac{y-A_l}{z-A_l }.&&
\end{eqnarray}
%% 
%%%%%%%%%%%%%%%%%%%%%%%%%%%%%%%%%%%%%%%%%%%%%%%%%%%%%%%%%%%%%%%% 
%%
Note that $\Psi^{(g)}_N(z,y)$ is a bidifferential $(N,1-N)$-quasiform
 meromorphic for $z$, $y\in \Omega_0 (\Gamma)$ with    
simple poles of residue one at $y=\gamma z$ for all $\gamma\in\Gamma$. 
%%
%%%%%%%%%%%%%%%%%%%%%%%%%%%%%%%%%%%%%%%%%%%%%%%%%%%%%%%%%%%%%% 
%%
It is an $N$-differential in $z$ since
%%
%%%%%%%%%%%%%%%%%%%%%%%%%%%%%%%%%%%%%%%%%%%%%%%%%%%%%%%%%%%%%5
%%
$\Psi^{(g)}_N(\gamma z,y) = \Psi^{(g)}_N(z,y)$, $\gamma\in\Gamma$,  
by construction, and it is a quasiperiodic $(1-N)$-form in $y$
 with  
%%
%%%%%%%%%%%%%%%%%%%%%%%%%%%%%%%%%%%%%%%%%%%%%%%%%%%%%%%%%%%%%%% 
%%
$\Psi^{(g)}_N( z,\gamma y) - \Psi^{(g)}_N(z,y)=\chi[\gamma](z,y)$,
 $\gamma\in\Gamma$,  
where $\chi[\gamma](z,y)$ is a holomorphic 
 $N$-form in $z$ \cite{Be1}.  
%%
%%%%%%%%%%%%%%%%%%%%%%%%%%%%%%%%%%%%%%%%%%%%%%%%%%%%%%%%%%%%%%%%%%%%%%%%%%%%%%%%%
%%
In particular, for a Schottky group generator $\gamma_a$, $1 \le a \le g$,
 $y_a=y-w_a$, 
%%
%%%%%%%%%%%%%%%%%%%%%%%%%%%%%%%%%%%%%%%%%%%%%%%%%%%%%%%%%%%%%%%%%%%%%%%%%%%%%
%%
$\chi[\gamma_a](z,y)=-\sum\limits_{l=0}^{2N-2}  
\Theta_{N,a}^l (z)y_a^l \;dy^{1-N}$,  
and $\{\Theta_{N,a}^l (z)\}_{1 \le a\le g}^{0 \le l 2N-2}$  
spans the $d_N=(g-1)(2N-1)$-dimensional space of holomorphic $N$-forms.  
%% 
%%%%%%%%%%%%%%%%%%%%%%%%%%%%%%%%%%%%%%%%%%%%%%%%%%%%%%%%%%%%%%%%%%%%%%%%%%%%%
%%
The Bers quasiform \eqref{erpoerksr} depends on the choice
 of limit set points $\{A_l \}$.
%% 
%%%%%%%%%%%%%%%%%%%%%%%%%%%%%%%%%%%%%%%%%%%%%%%%%%%%%%%%%%%%%%%%%%%%%%%%%%%%%%
%%
We may expand 
$p(y)=\sum_{l=0}^{2N-2} \frac{1}{l!}\partial^l_{w_a} p(w_a)y_a^l$, 
 and also find
for $1 \le a \le g$, $l\in \left\{0, \ldots, 2N-2 \right\}$, 
%%
%%%%%%%%%%%%%%%%%%%%%%%%%%%%%%%%%%%%%%%%%%%%%%%%%%%%%%%%%%%%%%%%
%%
\begin{eqnarray}
\label{ereroitf}
p_a^l = (-1)^{N+1} \rho_a^{N-l-1} 
  \frac{1}{(2N-2-l)!}\partial^{2N-2-l}_{w_{-a}} p(w_{-a})
- \frac{1}{l!}\partial^l_{w_a} \frac{1}{l!}\partial^l_{w_a} p(w_a).  &&  
\end{eqnarray}
%%
%%%%%%%%%%%%%%%%%%%%%%%%%%%%%%%%%%%%%%%%%%%%%%%%%%%%%%%%%%%%%%%%%%%%%%%%%%%%%%%%%%%%%%%
%%%% 
\subsection{The Schottky sewing formulas for  
$\Psi^{(g)}_N(z,y)$ and $\Theta_{N,a}^l (z)$}
\label{diktor}
In this Subsection we review results of \cite{TW1}   
where expansion formulas for the Bers quasiform  
$\Psi^{(g)}_N(z,y)$ and $\Theta_{N,a}^l(z)$ 
were given in terms of the sewing parameters $\rho_a$ 
for $N\ge 2$. 
%%
%%%%%%%%%%%%%%%%%%%%%%%%%%%%%%%%%%%%%%%%%%%%%%%%%%%%%%%%%%%%%%%%%%%%%%%%
%% 
These expressions are very useful in vertex operator algebra
  theory. 
%%
%%%%%%%%%%%%%%%%%%%%%%%%%%%%%%%%%%%%%%%%%%%%%%%%%%%%%%%%%%%%%%%%%%%%
%%
 Let $\Pi^{(g)}_N(z,y)=\pi^{(g)}_N(z,y)\;dz^N\;dy^{1-N}$, for $N\ge 1$, and  
 Lagrange polynomial 
$\mathcal Q_i(y)=\prod_{j\neq i}\frac{y-A_j}{A_i-A_j} \in \mathcal P_{2N-2} (y)$, 
 where $\sum_{i=0}^{2N-2} \mathcal Q_i(y)=1$, 
%% 
%%%%%%%%%%%%%%%%%%%%%%%%%%%%%%%%%%%%%%%%%%%%%%%%%%%%%%%%%%%%%%%%
%%
\begin{eqnarray}
\label{pwrepefk} 
\pi^{(g)}_N(z,y) 
=\frac{1}{z-y}+\sum\limits_{l=0}^{2N-2} f_l(z)y^l= 
\frac{1}{z-y} 
-\sum\limits_{i=0}^{2N-2} \frac{1}{z-A_i}\mathcal Q_i(y).   
\end{eqnarray}  
%% 
%%%%%%%%%%%%%%%%%%%%%%%%%%%%%%%%%%%%%%%%%%%%%%%%%%%%%%%%%%%%%%%%%%%%%5
%%
It is useful to define the following forms  
labeled by $a$, $b\in \{-1, \ldots, - g, 1, \ldots, g\}$ and integers $m$, $n\ge 0$  
constructed from moment integrals of $\Pi^{(g)}_N(z,y)$ as follows:
%%
%%%%%%%%%%%%%%%%%%%%%%%%%%%%%%%%%%%%%%%%%%%%%%%%%%%%%%%%%%%%%%%%%%%%%%
%% 
\begin{eqnarray}
\label{ererepo} 
L_b^n(z)= 
\frac{\rho_b^{\half n} \partial^n_{w_b}}{n!} 
\pi^{(g)}_N (z,w_b)\;dz^N, \;      
%%
%%%%%%%%%%%%%%%%%%%%%%%%%%%%%%%%%%%%%%%%%%%%%%%%%%%%%%%%%%%%%%%%%%%%%%
%%
R_a^m(y)= \frac{(-1)^N\rho_a^{\half (m+1)} \partial^m_{w_{-b}}}{m!} 
  \pi^{(g)}_N (w_{-a},y)\;dy^{1-N},     
\end{eqnarray}
%% 
%%%%%%%%%%%%%%%%%%%%%%%%%%%%%%%%%%%%%%%%%%%%%%%%%%%%%%%%%%%%%%%%%%%%%%%%%%%%%%%%%%%%%%%
%%
where $y_b=y-w_b$ and $z_{-a}=z-w_{-a}$.  
%% 
%%%%%%%%%%%%%%%%%%%%%%%%%%%%%%%%%%%%%%%%%%%%%%%%%%%%%%%%%%%%%%%%%%%%%%%%%%%%%%%%%%%%%%
%%%%
One introduces doubly indexed 
infinite row and column vectors $L(z)=(L_b^n(z))$ and  
$R(z)=(R_a^m(y))$. 
%% 
%%%%%%%%%%%%%%%%%%%%%%%%%%%%%%%%%%%%%%%%%%%%%%%%%%%%%%%%%%%%
%%%%
We also define the doubly indexed matrix  
$A=(A_{ab}^{mn})$ with components for 
 $e_m^n(y)=\sum_{l=0}^{2N-2} \binom{l}{n} f_l^{(m)}(y)y^{l-n}$  
with $f_l$ of \eqref{pwrepefk}, 
%% 
%%%%%%%%%%%%%%%%%%%%%%%%%%%%%%%%%%%%%%%%%%%%%%%%%%%%%%%%%%
%%
\begin{eqnarray*} 
A_{ab}^{mn}=  
(-1)^N\rho_a^{\half(m+\delta_{a, -b}n+1)} 
\left(\delta_{a, -b} e_m^n(w_{-a}) 
+ (1-\delta_{a, -b})
 \frac{\rho_b^{\half n} \partial^m_{w_{-a}} \partial^n_{w_b} }{m!n!} 
\pi^{(g)}_N (w_{-a}, w_b) \right). 
\end{eqnarray*}  
%% 
%%%%%%%%%%%%%%%%%%%%%%%%%%%%%%%%%%%%%%%%%%%%%%%%%%%%%%%%%%%5
%%
We note that $A_{a,-a}^{mn}=0$ for all $n>2N-2$.
One defines the matrix $D$ with components 
%% 
%%%%%%%%%%%%%%%%%%%%%%%%%%%%%%%%%%%%%%%%%%%%%%%%%%%%%%%%%% 
%%
$D_{ab}^{mn}=\delta_{m,n+2N-1} \delta_{a,b}$.  
%%
%%%%%%%%%%%%%%%%%%%%%%%%%%%%%%%%%%%%%%%%%%%%%%%%%%%%%%%%%
%%
Let $\widetilde{A}=AD$.
 These are independent of the $f_l(z)$ terms with 
 $L_b^n(z)D =
\frac{\rho_b^{\half (n+2N-1)}}{(z-w_{b})^{n+2N}}\; dz^N$, 
%% 
%%%%%%%%%%%%%%%%%%%%%%%%%%%%%%%%%%%%%%%%%%%%%%%%%%%%%%%%%
%%
\begin{eqnarray} 
 \label{mandat} 
&& \widetilde{A}_{ab}^{mn}=
  (1-\delta_{a, -b})(-1)^{m+N}\binom{m+n+2N-1}{m}
\frac{\rho_a^{\half (m+1)}\rho_b^{\half (n+2N-1)}} 
{(w_{-a}-w_b)^{m+n+2N}}, 
\end{eqnarray} 
%% 
%%%%%%%%%%%%%%%%%%%%%%%%%%%%%%%%%%%%%%%%%%%%%%%%%%%%%%%%%%%%%%%%%
%%
We define $(I-\widetilde{A})^{-1}=\sum_{k\geq 0}\widetilde{A}^k$ 
 where $I$ denotes the infinite identity matrix.  
Then  
$\Psi^{(g)}_N(z,y)$ can be expressed in terms of
 $\Pi^{(g)}_N$, $L(z)$, $R$, $\widetilde{A}$ as follows \cite{TW1, TW}.
%%
%%%%%%%%%%%%%%%%%%%%%%%%%%%%%%%%%%%%%%%%%%%%%%%%%%%%%%%%%%%%%%%%%%%%%%%%
%% 
\begin{theorem}
\label{palre} 
With $\left(I-\widetilde{A} \right)^{-1}$
 convergent for all $({\bf w,\rho})\in \Cg$ 
for all $N\ge 1$ and $z$, $y\in \D$, 
%% 
%%%%%%%%%%%%%%%%%%%%%%%%%%%%%%%%%%%%%%%%%%%%%%%
%% 
$\Psi^{(g)}_N (z,y)=\Pi^{(g)}_N(z,y)+L(z)D 
\left(I-\widetilde{A} \right)^{-1}R(y)$. \hfill $\qed$  
\end{theorem}
%%
%%%%%%%%%%%%%%%%%%%%%%%%%%%%%%%%%%%%%%%%%%%%%%%%%%%%%%%%%%%%%%%%%%%%%%%%%
%% 
We also find that the holomorphic $N$-form
 $\Theta_{N,a}^l (z)$ of $\chi[\gamma_a](z,y)$ is given by  
%% 
%%%%%%%%%%%%%%%%%%%%%%%%%%%%%%%%%%%%%%%%%%%%%%%%%%%%%%
%% 
\begin{proposition}
\label{paraskas} 
Let $1 \le a\le g$, and $l\in \{0, \ldots, 2N-2\}$, and 
$T_a^l(z)=\rho_a^{-\half l}  
L(X)(I + D(I-\widetilde{A})^{-1}A  )_a^l$. Then 
%% 
%%%%%%%%%%%%%%%%%%%%%%%%%%%%%%%%%%%%%%%%%%%%%%%%%%%%%%%%%%%55
%% 
$\Theta_{N,a}^l(z)
=T_a^l(z)+(-1)^N\rho_a^{N-1-l}T_{-a}^{2N-2-l}(z)$. 
\hfill $\qed$ 
\end{proposition}
%% 
%%%%%%%%%%%%%%%%%%%%%%%%%%%%%%%%%%%%%%%%%%%%%%%%%%%%%%%%%%%%%%%%%%%%%%%%%%%
%%%%%%%%%%%%%%%%%%%%%%%%%%%%%%%%%%%%%%%%%%%%%%%%%%%%%%%%%%%%%%%%%%%%%%%%%%%
%%
\section{Appendix: Vertex operator algebras} 
\label{generererv}
%%
%%%%%%%%%%%%%%%%%%%%%%%%%%%%%%%%%%%%%%%%%%%%%%%%%%%%%%%%%%%%%%%%%%%%%%%%%%
%% 
\subsection{Vertex operator algebras} 
\label{voa} 
%% 
%%%%%%%%%%%%%%%%%%%%%%%%%%%%%%%%%%%%%%%%%%%%%%%%%%%%%%%%%%%%%%%%%%%%%%%%5
%% 
In this Subsection we recall some facts about vertex operator algebras,  
  \cite{DL, FHL, K, LL, MT1}.
 A vertex operator 
 algebra is a quadruple 
$(V,Y(.,.),\vac,\omega)$ consisting 
 of a graded vector space  
$V=\bigoplus_{n\ge 0}V_{(n)}$, with $\dim V_{(n)}<\infty$,
containing two specific elements.  
%%
%%%%%%%%%%%%%%%%%%%%%%%%%%%%%%%%%%%%%%%%%%%%%%%%%%%%%%%%%%%%%%%%%%%%%%%
%%
 Those are called the vacuum vector $\vac\in V_{(0)}$ 
  and the Virasoro conformal vector $\omega\in V_{(2)}$.    
%% 
%%%%%%%%%%%%%%%%%%%%%%%%%%%%%%%%%%%%%%%%%%%%%%%%%%%%%%%%%%%%%%%%%%%%
%%%%
For each $u\in V$, one associates a vertex operator which is 
a formal Laurent series in $z$ given by    
$Y(u,z)=\sum_{n\in\Z} u(n)z^{-n-1}$,   
with modes $u(n) \in \End(V)$.  
%% 
%%%%%%%%%%%%%%%%%%%%%%%%%%%%%%%%%%%%%%%%%%%%%%%%%%%%%%%%%%%%%%%%%%%
%% 
The lower truncation condition is assumed: 
 for each $u$, $v\in V$, $u(n).v=0$ for all
 $n \gg 0$. 
%%%%%%%%%%%%%%%%%%%%%%%%%%%%%%%%%%%%%%%%%%%%%%%%%%%%%%%%%%%%
%%%%
The creativity condition is 
 $u=u(-1).\vac$, $u(n).\vac =0$ for all $n\ge 0$. 
%%
%%%%%%%%%%%%%%%%%%%%%%%%%%%%%%%%%%%%%%%%%%%%%%%%%%%%%%%%%%%%%%%%%%%%%%
%%%%
For formal variables $z$, $y$, one uses 
the binomial expansion for $m\in\Z$,   
$(z+y)^m=\sum_{k\ge 0}\binom{m}{k} z^{m-k}y^k$. 
%% 
%%%%%%%%%%%%%%%%%%%%%%%%%%%%%%%%%%%%%%%%%%%%%%%%%%%%%%%%%%%%%%%%%%%%%
%%%%
The vertex operators also obey locality condition 
$(z-y)^N[Y(u,z),Y(v,y)]=0$, $N\gg 0$.  
%% 
%%%%%%%%%%%%%%%%%%%%%%%%%%%%%%%%%%%%%%%%%%%%%%%%%%%%%%%%%%%%%%%%%%%%%5
%%%%
For the Virasoro conformal vector $\omega$
$Y(\omega,z) =\sum_{n\in\Z}L(n)z^{-n-2}$,  
where the operators 
 $L(n)=\omega(n+1)$ satisfy the Virasoro algebra commutation relations 
$[L(m),L(n)]=(m-n)L(m+n)+\frac{C}{2}\binom{m+1}{3}\delta_{m,-n}\Id_V$,  
for a constant central charge $C\in\C$. 
%%
%%%%%%%%%%%%%%%%%%%%%%%%%%%%%%%%%%%%%%%%%%%%%%%%%%%%%%%%%%%%%%%%%%%%%%%%
%%%% 
The translation property for vertex operators is given by 
$Y(L(-1)u,z)=\partial Y(u,z)$. 
The grading on $V$ is defined via the $L(0)$ Virasoro mode, i.e., 
 $V_{(n)}=\{v\in V:L(0)v=nv\}$ where
$v\in V_{(n)}$ is the conformal weight $\wt(v)=n$.   
%% 
%%%%%%%%%%%%%%%%%%%%%%%%%%%%%%%%%%%%%%%%%%%%%%%%%%%%%%%%%%%  
%%
 For $u \in V_{(N)}$  
\begin{eqnarray}
\label{genroooor} 
u(j):V_{(k)}\rightarrow V_{(k+N-j-1)}. 
\end{eqnarray}
%% 
%%%%%%%%%%%%%%%%%%%%%%%%%%%%%%%%%%%%%%%%%%%%%%%%%%%%%%%%%%%%%%%%%%%%
%%
The commutation rule is the following, for all $u$, $v\in V$, 
 $[u(k)$, $Y(v,z)]$ $=$ $\left(\sum_{j\ge 0}\right.$ $Y(u(j).v,z)$ 
 $\left. \partial_z^{(j)}\right)z^k$.   
%% 
%%%%%%%%%%%%%%%%%%%%%%%%%%%%%%%%%%%%%%%%%%%%%%%%%%%%%%%%%%%%%%%%%%%
%% 
Vertex algebra element enjoy the associativity identity, 
 for each $u$, $v\in V$ 
there exists $M\ge 0$ such that
$(z+y)^M Y(Y(u,z)v,y)=(z+y)^M Y(u,z+y)Y(v,y)$.  
%%  
%%%%%%%%%%%%%%%%%%%%%%%%%%%%%%%%%%%%%%%%%%%%%%%%%%%%%%%%%%%%%%%%%%
%%
Associated with the formal M\"obius map  
$z\rightarrow \rho/z$, for a given scalar $\rho \neq 0$,  
we define an adjoint vertex operator \cite{FHL, L} 
%%
%%%%%%%%%%%%%%%%%%%%%%%%%%%%%%%%%%%%%%%%%%%%%%%%%%%%%%%%%%%%%%%%%%
%%%%
$Y_\rho^\dagger(u,z)$ $=$ $\sum_{n\in\Z}$ $u_\rho^\dagger(n)$ $z^{-n-1}$   
$=$ $Y\left(e^{\frac{z}{\rho}L(1)}\right.$ $\left.(-\frac{\rho}{z^2}\right)^{L(0)} 
u$, $\left. \frac{\rho}{z}\right)$.  
%%
%%%%%%%%%%%%%%%%%%%%%%%%%%%%%%%%%%%%%%%%%%%%%%%%%%%%%%%%%%%%%%%%%%%%%%
%% 
We write $Y^\dagger(u,z)$ for the adjoint when $\rho=1$.
Let $\{b\}$ be a homogeneous basis for $V$ with the dual basis $\{\bbar \}$ 
with respect to the bilinear pairing above. 
%%
%%%%%%%%%%%%%%%%%%%%%%%%%%%%%%%%%%%%%%%%%%%%%%%%%%%%%%%%%%%%
%%
For each $1 \le a \le g$, let $\{b_a\}$ denote a 
homogeneous $V$-basis and let $\{\bbar _a\}$ 
be the dual basis with respect to the bilinear pairing 
$\langle ., .\rangle_1$,  i.e., with  
$\rho=1$. 
%% 
%%%%%%%%%%%%%%%%%%%%%%%%%%%%%%%%%%%%%%%%%%%%%%%%%%%%%%%%%%%%%%%%%%%%
%%
In the Sections \ref{genuszero0}--\ref{apala} we identify 
  a Schottky sewing parameter with the formal parameter $\rho_a$.  
 Define for $1 \le a \le g$, for 
\begin{eqnarray}
\label{gensher}
b_{-a}=\rho_a^{\wt(b_a)}\bbar _a. 
\end{eqnarray}
%%
%%%%%%%%%%%%%%%%%%%%%%%%%%%%%%%%%%%%%%%%%%%%%%%%%%%%%%%%%%%%%%%%%%%%%%%%%%%
%% 
Then $\{b_{-a}\}$ is a dual basis for the bilinear pairing 
 $\langle .,.\rangle_{\rho_a}$ 
with adjoint 
for $u\in V$ quasiprimary, i.e., $L(1)u=0$, of weight $N$. 
%%
%%%%%%%%%%%%%%%%%%%%%%%%%%%%%%%%%%%%%%%%%%%%%%%%%%%%%%%%%%%%%%%%%%%%%%%%%%
%%
The identification $L_\rho^\dagger(n)=\rho^n L(-n)$, follows from 
%%
%%%%%%%%%%%%%%%%%%%%%%%%%%%%%%%%%%%%%%%%%%%%%%%%%%%%%%%%%%%%%%%%%%%%%%%
%% 
\begin{eqnarray}
\label{tuta} 
u^\dagger_{\rho_a}(m)=(-1)^N\rho_a^{m+1-N}u(2N-2-m). 
\end{eqnarray} 
%% 
%%%%%%%%%%%%%%%%%%%%%%%%%%%%%%%%%%%%%%%%%%%%%%%%%%%%%%%%%%%%%%%%%%%%%%%
%%%%
Let ${\bf b}_+=b_1\otimes\ldots \otimes b_g$ 
denote an element of a $V^{\otimes g}$-basis. 
Then let $w_a$ for $a\in \{-1$, $\ldots$, $-g$, $1$, $\ldots$, $g\}$
 be $2g$ formal variables.   
%%
%%%%%%%%%%%%%%%%%%%%%%%%%%%%%%%%%%%%%%%%%%%%%%%%%%%%%%%%%%%%%%%%%%%%%%%%%%%%%%%%%%%%%%%%%
%%
A bilinear pairing $\langle .,. \rangle_\rho$ on $V$   
is called invariant if for all $u$, $v$, $w\in V$, 
%% 
%%%%%%%%%%%%%%%%%%%%%%%%%%%%%%%%%%%%%%%%%%%%%%%%%%%%%%%%%%%%%%%%%%%%%%%%%%%%%%
%%
$\langle Y (u, z)v, w\rangle_\rho = \langle v, Y_\rho^\dagger
(u, z) w\rangle_\rho$. 
%% 
%%%%%%%%%%%%%%%%%%%%%%%%%%%%%%%%%%%%%%%%%%%%%%%%%%%%%%%%%%%%%%%%%%%%%%%%%%%%%%
%% 
The pairing $\langle .,.\rangle_\rho$
 is symmetric and invertible where $\langle u, v\rangle_\rho=0$ for 
$\wt(u)\neq\wt(v)$ \cite{FHL} and  
%%
%%%%%%%%%%%%%%%%%%%%%%%%%%%%%%%%%%%%%%%%%%%%%%%%%%%%%%%%%%%%%%%%%%%%%%%%%%%%%
%% 
$\langle u, v\rangle_\rho=\rho^N\langle u, v\rangle_1$, 
$N=\wt(u)=\wt(v)$ for homogeneous $u$, $v$.  
%%
%%%%%%%%%%%%%%%%%%%%%%%%%%%%%%%%%%%%%%%%%%%%%%%%%%%%%%%%%%%%%%%%%%%%%%%%%%%%%%% 
%%
In this paper we assume that $V$ is of strong conformal field theory-type,
 i.e., $V_{(0)}= \C\vac$ and $L(1)V_{(1)}= 0$. 
Then the bilinear pairing with normalization 
$\langle \vac, \vac \rangle_\rho=1$
 is unique\cite{L}.
%%
%%%%%%%%%%%%%%%%%%%%%%%%%%%%%%%%%%%%%%%%%%%%%%%%%%%%%%%%%%%%%%%%%%%%%%%%%%%%%
%% 
 We also assume that $V$ is simple and isomorphic to the 
contragredient $V$-module $V'$ \cite{FHL}. 
Then the bilinear pairing is non-degenerate \cite{L}.
%% 
%%%%%%%%%%%%%%%%%%%%%%%%%%%%%%%%%%%%%%%%%%%%%%%%%%%%%%%%%%%%%%%%%%%%%%%%%%%%
%% 
\subsection{The group of automorphisms of formal parameters}
\label{groupppp}
Asume that $V$ is a quasi-conformal vertex operator algebra. 
Let us recall some facts \cite{BZF} relating 
 generators of Virasoro algebra with the group of 
automorphisms of local coordinates. 
%%
%%%%%%%%%%%%%%%%%%%%%%%%%%%%%%%%%%%%%%%%%%%%%%%%%%%%%%%%%%%%%%%%%%%%%%%%%%%%%
%%
An element of ${\rm Aut}_z \; \mathcal O^{(1)}$ is representable by the map
given by the power series, in particular, in an exponential form
through $\beta_k \in \mathbb C$ via combinations of $a_k$, $k\ge 0$,
%%
%%%%%%%%%%%%%%%%%%%%%%%%%%%%%%%%%%%%%%%%%%%%%%%
%%%%
\begin{equation}
\label{lempa} 
z \mapsto \rho=\rho(z), \; 
\rho(z) = \sum\limits_{k \ge 1} a_k z^k, 
\;
f(z) = \exp \left(  \sum\limits_{k > -1} \beta_{k }\; z^{k+1} \partial_{z} \right) 
\left(\beta_0 \right)^{z \partial_z}.z.  
\end{equation}
%%
%%%%%%%%%%%%%%%%%%%%%%%%%%%%%%%%%%%%%%%%%%%%%%%%%%%%%%%%%%%%%%%%%%%%%%%%%%%%%%%%%%%%%%%%%%%%%%%%
 In terms of differential operators 
 a representation of Virasoro algebra modes
 is given by \cite{K} for $m \in \Z$,  
%%
%%%%%%%%%%%%%%%%%%%%%%%%%%%%%%%%%%%%%%%%%%%%%%%%%%%%%%%%%%%%%%%%%%%%%%%%
%% 
\begin{equation}
\label{repro}
L_W(m) \mapsto - \zeta^{m+1}\partial_\zeta. 
\end{equation}
%%
%%%%%%%%%%%%%%%%%%%%%%%%%%%%%%%%%%%%%%%%%%%%%%%%%%%%%%%%%%%%%%%%% 
%%
Using \eqref{lempa} we obtain a system of recursive equations 
  solved for all $\beta_k$. 
%%
%%%%%%%%%%%%%%%%%%%%%%%%%%%%%%%%%%%%%%%%%%%%%%%%%%%%%%%%%%%%%%%%%%%%%%%%%%%%%%
%%
One finds for $v \in V$ of a Virasoro generator commutation formula 
\begin{eqnarray}
\label{solzhe}
&&
\left[L_W(n), Y_W (v, z) \right] 
=  \sum\limits_{m \geq -1}  
 \frac{1}{(m+1)!} \left(\partial^{m+1}_z z^{m+1}  \right)\;  
Y_W (L_V(m) v, z).   
\end{eqnarray}
%%
%%%%%%%%%%%%%%%%%%%%%%%%%%%%%%%%%%%%%%%%%%%%%%%%%%%%%%%%%%%%%%%%%%%%%%%%%%%%%%%%
%%
 One introduces the operator 
$\beta=-\sum_{n \geq -1} \beta_n L_W(n)$, 
for a vector field  
$\beta(z)\partial_z= \sum_{n \geq -1} \beta_n z^{n+1} \partial_z$,  
which belongs to local Lie algebra of ${\rm Aut}\; \mathcal O^{(1)}$. 
%%
%%%%%%%%%%%%%%%%%%%%%%%%%%%%%%%%%%%%%%%%%%%%%%%%%%%%%%%%%%%%%%%%%%%%%%%%%%%%%%%%
%%%%
From the expansion of $\beta(z)\partial_z$ we obtain 
%%
%%%%%%%%%%%%%%%%%%%%%%%%%%%%%%%%%%%%%%%%%%%%%%%%%%%%%%%%%%%%%%%%%%%%%%%%%%%%%%%%
%%%% 
\begin{lemma} 
\label{tulgas}
$\left[\beta, Y_W (v, z) \right] 
=  - \sum_{m \geq -1}  
 \frac{1}{(m+1)!} \left(\partial^{m+1}_z \beta(z)  \right)\;  
Y_W (L_V(m) v, z)$.   \hfill $\qed$ 
\end{lemma}
%%
%%%%%%%%%%%%%%%%%%%%%%%%%%%%%%%%%%%%%%%%%%%%%%%%%%%%%%%%%%%%%%%%%%%%%%%%%%%%%%%%%%
%%%%  
%%
When a vertex operator algebra carries an action of ${\rm Der}\; \mathcal O^{(n)}$ 
with commutation formula of Lemma \ref{tulgas} for any   
$v \in  V$, $z=z_j$, $1 \le j \le n$,
  the element $L_V(-1)= - \partial_z$ is the translation  
operator $L_V(0) = -z\partial_z$ that 
 acts semi-simply with integral eigenvalues, 
and the Lie subalgebra  
${\rm Der}_+ \; \mathcal O^{(n)}$ acts locally nilpotently, 
then one calls it quasi-conformal.  
%%
%%%%%%%%%%%%%%%%%%%%%%%%%%%%%%%%%%%%%%%%%%%%%%%%%%%%%%%%%%%%%%%%%%%%%%%%%%%%%%%%%%
%%%%
A vector $A$ of a quasi-conformal 
 vertex algebra $V$ satisfying  $L_V(k) A = 0$, $k > 0$,   
 $L_W(0) A = D(A) A$, is called  
primary of conformal dimension $D(A) \in  \mathbb Z_+$. 
%%  
%%%%%%%%%%%%%%%%%%%%%%%%%%%%%%%%%%%%%%%%%%%%%%%%%%%%%%%%%%%%%%%%%%%%%%%%%%%%%%%%%%%%
%%
The invariance of vertex operators multiplied by conformal weight differentials
follows from the formula of Lemma \ref{tulgas}. 
%%
%%%%%%%%%%%%%%%%%%%%%%%%%%%%%%%%%%%%%%%%%%%%%%%%%%%%%%%%%%%%%%%%%%%%%%%%%%
%%%%
A conformal vertex algebra  
is a conformal vertex algebra $V$ 
 equipped with an action of 
Virasoro algebra and therefore its Lie subalgebra ${\rm Der}_0 \; \mathcal O^{(n)}$
given by the Lie algebra of ${\rm Aut} \; \mathcal O^{(n)}$.  
%%
%%%%%%%%%%%%%%%%%%%%%%%%%%%%%%%%%%%%%%%%%%%%%%%%%%%%%%%%%%%%%%%%%%%%%%%%%%%%%%%%%%%%%%%%%%%
%%%%
By using the identification \eqref{repro},
one introduces the linear operator 
 $P\left(f (\zeta) \right)= 
\exp \left( \sum_{m >  0} (m+1)\; \beta_m \; L_V(m) \right) \beta_0^{L_W(0)}$,   
 representing $f(\zeta)$ \eqref{lempa} 
via Lemma \ref{tulgas} 
for quasi-conformal vertex algebra
%%
%%%%%%%%%%%%%%%%%%%%%%%%%%%%%%%%%%%%%%%%%%%%%%%%% 
%%
Under the action of all operators 
$P(f)$, $f \in  {\rm Aut}\; \mathcal O^{(1)}$ 
on vertex algebra elements $v\in V_{(n)}$
the number of terms in finite, 
 and subspaces $V_{(\le m)} = \bigoplus_{ n \ge K}^m V_{(n)}$ are stable. 
%%
%%%%%%%%%%%%%%%%%%%%%%%%%%%%%%%%%%%%%%%%%%%%%%%%%%%%%%%%%%%%%%%%%%%%%%%%%%%
%%
 One has \cite{BZF} the following 
%%
%%%%%%%%%%%%%%%%%%%%%%%%%%%%%%%%%%%%%%%%%%%%%%%%%%%%%%%%%%%%%%%%%%%%%%%%%%
%%%% 
\begin{lemma}
 The map $f \mapsto P(f)$ is a 
  a representation of ${\rm Aut}\; \mathcal O^{(1)}$ on $V$,  
 $P(f_1 * f_2) = P(f_1) \; P(f_2)$,   
 which is the inductive limit of the
representations $V_{(\le m)}$, $m\ge K$ with some $K$.  \hfill \qed 
\end{lemma}
%%
%%%%%%%%%%%%%%%%%%%%%%%%%%%%%%%%%%%%%%%%%%%%%%%%%%%%%%%%%%%%%%%%%%%%%%%%%%%%%%%%%%%%%%%%%%%%%%%
%%%%%%%%%%%%%%%%%%%%%%%%%%%%%%%%%%%%%%%%%%%%%%%%%%%%%%%%%%%%%%%%%%%%%%%%%%%%%%%%%%%%%%%%%%%%%%%
%% 

%%
\end{document}